\documentclass[11pt]{article}
\usepackage{amsmath}
\usepackage{amsthm}

\usepackage{amsfonts}
\usepackage{amssymb}
\usepackage{mathtools}
\usepackage{soul}
\usepackage{hyperref}

\usepackage[sort,compress]{cite}

\usepackage[draft]{todonotes}   % notes showed

\usepackage{cases}
\usepackage{xcolor}

\newcommand{\dt}{\partial_t}

\newcommand{\dvh}{div_h}
\newcommand{\nablah}{\nabla_h}
\newcommand{\deltah}{\Delta_h}
\DeclareMathOperator{\sg}{sg}
\newcommand{\dz}{\partial_z}

\newcommand{\idx}{\,d\vec{x}}
\newcommand{\idxh}{\,dxdy}
\newcommand{\vech}[1]{\vec{#1}_h}
\newcommand{\subeqref}[2]{$\eqref{#1}_{#2}$}
\newcommand{\abs}[2]{| #1 |^{#2}}
\newcommand{\norm}[2]{\bigl\Arrowvert #1 \bigr\Arrowvert_{#2}}
\newcommand{\normh}[2]{\bigl| #1 \bigr|_{#2}}

\newcommand{\Lnorm}[1]{L^{#1}}
\newcommand{\Hnorm}[1]{H^{#1}}

\newtheorem{thm}{Theorem}[section]
\newtheorem{lm}{Lemma}

\newtheorem{prop}{Proposition}
\numberwithin{equation}{section}
\allowdisplaybreaks
\title{Zero Mach Number Limit of the Compressible Primitive Equations: Ill-prepared Initial Data}

\author{Xin Liu\footnote{Department of Mathematics, Texas A{\&}M University, College Station, TX 77843, USA. Email: stleonliu@gmail.com} \,\, and \,  Edriss S. Titi\footnote{Department of Mathematics, Texas A{\&}M University, College Station,  TX 77840, USA.  Department of Applied Mathematics and Theoretical Physics, University of Cambridge, Cambridge CB3 0WA, UK.
		Department of Computer Science and Applied Mathematics, Weizmann Institute of Science, Rehovot 76100, Israel. Email: {titi@math.tamu.edu}\, and \, {Edriss.Titi@maths.cam.ac.uk}}
}

\date{July 30, 2022}
\begin{document}
\allowdisplaybreaks
\maketitle

\begin{abstract}
	In the work, we consider the zero Mach number limit of compressible primitive equations in the domain $\mathbb R^2 \times 2\mathbb T $ or $\mathbb T^2 \times 2\mathbb T $. We identify the limit equations to be the primitive equations with the incompressible condition. The convergence behaviors are studied in both $\mathbb R^2 \times 2\mathbb T $ and $\mathbb T^2 \times 2\mathbb T $, respectively.
	This paper takes into account the high oscillating acoustic waves and is an extension of our previous work in \cite{LT2018LowMach1}.
	
	\bigskip
	
\end{abstract}

\tableofcontents

%\begin{abstract}
%
%\end{abstract}
%
%\tableofcontents

\section{Introduction}

\subsection{Zero Mach number limit}

As an example of multi-scale analysis,
the zero Mach number limit problem for compressible flows has been an important classical problem in the study of hydrodynamic equations. Pioneered by Klainerman and Majda in \cite{Klainerman1982,Klainerman1981}, it is shown that the solutions of inviscid compressible Euler equations converge to that of inviscid incompressible Euler equations for the isentropic flows in $ \mathbb R^d $ and $ \mathbb T^d $, $ d \in \mathbb Z^+ $. The initial data are well-prepared and almost incompressible. The result is further studied in bounded domains for nonisentropic flows by Schochet in \cite{SchochetCMP1986,Schochet1988}. In \cite{Ukai1986}, the author establishes the zero Mach number limit with general (ill-prepared) initial data in $ \mathbb R^d $. As pointed out in \cite{Metivier2001}, in comparison to the case when the initial data are well-prepared, the time derivatives are no longer uniformly bounded with respect to the Mach number when the system is complemented with ill-prepared initial data. This leads to high frequency acoustic waves with large amplitude.

%Consequently, one has to study carefully the large-amplitude high-frequency solutions in the ill-prepared data case.

We have studied the zero Mach number limit of compressible primitive equations with well-prepared initial data in \cite{LT2018LowMach1}. In this work, we are considering such a singular limit problem with ill-prepared initial data. %As mentioned before,
%we have to study the high-frequency solutions associated with the system.

To present the general idea of singular limit problems with ill-prepared initial data, consider the following equation of the unknown function $ G $,
\begin{equation}\label{eq:anti-symmetry-system}
%\begin{cases}
\dt G + \dfrac{1}{\varepsilon} L(\partial_x) G  = \mathcal N(G),
%\end{cases}
\end{equation}
where $ \varepsilon \in (0,\infty) $ is the singular parameter, $ L(\partial_x) $ is an anti-symmetric differential operator with constant coefficients, and $ \mathcal N(G) $ is the nonlinearity of $ G $. Then one can perform the $ H^s $ estimate for some large $s \in \mathbb Z^+ $, on system \eqref{eq:anti-symmetry-system}. Since $ L(\partial_x) $ is anti-symmetric, $ \norm{G}{H^s} $ is uniformly bounded with respect to $ \varepsilon $, at least locally in time.
If one considers the projection of \eqref{eq:anti-symmetry-system} in the kernel of operator $ L(\partial_x) $, which is called the non-singular part of the equation,  it follows that $$ \dt  P_{\ker(L(\partial_x))} G = P_{\ker(L(\partial_x))} \mathcal N(G) $$ is uniformly bounded in $ H^{s'} $ space for some $ s' \in \mathbb Z^+ \cup\lbrace 0 \rbrace $ with respect to $ \varepsilon $. On the other hand, $ G - P_{\ker(L(\partial_x))} G $ satisfies% an equation with the following linear homogeneous part,
\begin{equation}\label{eq:oscillation-system}
\begin{aligned}
& \dt (G - P_{\ker(L(\partial_x))} G) + \dfrac{1}{\varepsilon} L(\partial_x) (G - P_{\ker(L(\partial_x))} G) \\
& ~~~~ ~~~~ = \mathcal N(G) - P_{\ker(L(\partial_x))} \mathcal N(G) ,
\end{aligned}
\end{equation}
which has wave packet solutions with fast oscillations as $ \varepsilon \rightarrow 0^+ $. Consequently, $ G - P_{\ker(L(\partial_x))} G $ is oscillating at high frequency.
Moreover, as $ \varepsilon \rightarrow 0^+ $, the $ H^s $ norm preserving property of the anti-symmetric operator $ L(\partial_x) $ implies that the amplitude of the oscillations does not vanish in general.
Instead, the limit of the oscillatory part is driven by certain PDEs.
To resolve the singular limit problem of \eqref{eq:anti-symmetry-system} as $ \varepsilon \rightarrow 0^+ $, one has to study the interactions of the non-singular part and the oscillatory part in the nonlinearity $ \mathcal N(G) $ in order to identify the limit equations. Such a method of studying the singular limit problem of \eqref{eq:anti-symmetry-system} is developed by Schochet in \cite{Schochet1994} for hyperbolic PDEs with applications to the incompressible limit problem of Euler equations, nonlinear wave equations and the theory of weakly nonlinear geometric optics. Later, this method is further developed for some parabolic equations by Gallagher in \cite{Gallagher1998}. We remark here that, if equation \eqref{eq:anti-symmetry-system} is complemented with well-prepared initial data, the amplitude of oscillations is small, and consequently, the interaction between the non-singular part and the oscillatory part is much weaker.

In the study of zero Mach number limit of hydrodynamic equations for isentropic flows, if one writes the equations in the form of \eqref{eq:anti-symmetry-system}, the corresponding kernel of $ L(\partial_x) $ consists the solenoidal velocity field. The equations corresponding to the fast oscillation equations \eqref{eq:oscillation-system} are referred to as the acoustic wave equations. Using these terminologies, the theorem developped by Ukai in \cite{Ukai1986} is basically showing that the acoustic waves decay to zero as $ \varepsilon \rightarrow 0^+ $ in $ \mathbb R^d $. Such a fact is the consequence of the Strichartz estimates for linear wave equations, as pointed out by Desjardins and Grenier in \cite{DesjardinsGrenier1999} (see, e.g., \cite{Ginibre1995,Lindblad1995,Keel1998} about the Strichartz estimate). In $ \mathbb T^d $, Lions and Masmoudi study the resonance of the high frequency osciallations and show that in the sense of distribution, the solutions to the compressible Navier--Stokes equations converge to that to the incompressible Navier--Stokes equations as the Mach number goes to zero in \cite{Lions1998a}. Later in \cite{Masmoudi2001}, Masmoudi studies the incompressible, inviscid limit, with low Mach number and large Reynolds number, of compressible Navier--Stokes equations and identities the equations %propergating the limit of fast oscillations
in both $ \mathbb R^d $ and $ \mathbb T^d $. We refer readers, for further developments, to \cite{Danchin2002,Danchin2002per,Danchin2005,Schochet2005Review}.

Additionally, we would like to mention that
%We additionally mention that,
when considering the non-isentropic flows, the corresponding operator $ L(\partial_x) $ in \eqref{eq:anti-symmetry-system} has coefficients depending on space and time variables. This causes non-trivial difficulties to study the resonances between the oscillations (see, e.g., \cite{Metivier2001,Metivier2003,Alazard2006,Alazard2005,AlazardReview}). Moreover, when taking into account the stratification effect of gravity, the compressible Navier--Stokes equations with gravity may converge to the Oberbeck-Boussinesq equations or the anelastic equations, depending on the strength of the gravity effect. We refer readers, for more discussions of related topics, to \cite{Feireisl2015a,Feireisl2011,Feireisl2008,Masmoudi2007,Danchin2008,Brandolese2013,Wroblewskakaminska2017,Novotny2011,EFeireisl2012,Rajagopal1996,FeireislSingularLimits}. Also, for more multi-scale analysis, we refer readers to
\cite{Klein2000,Klein2001,RKlein2010,Klein2005,MajdaAtmosphereOcean}.

\subsection{The compressible primitive equations}
As mentioned before, we aim at studying the low Mach number limit of the compressible primitive equations. As discussed in our previous work \cite{LT2018LowMach1}, this is part of the justification of the PE diagram (see Figure 1 in \cite{LT2018LowMach1}). We refer readers, for more background of the compressible primitive equations, to \cite{LT2018LowMach1}.

Let $ \varepsilon \in (0,\infty) $ denote the Mach number, and
let $ \rho_\varepsilon \in \mathbb R, v_\varepsilon \in \mathbb R^2, w_\varepsilon \in \mathbb R $ represent the density, the horizontal and the vertical velocities, respectively. Then the compressible primitive equations can be written as, after rescaling the original CPE, similar to that of the compressible Navier--Stokes equations (see, e.g., \cite{FeireislSingularLimits}):
\begin{equation*}\label{CPE-0}\tag{CPE}
\begin{cases}
	\dt \rho_\varepsilon + \dvh (\rho_\varepsilon v_\varepsilon) + \dz (\rho_\varepsilon w_\varepsilon) = 0  & \text{in} ~ \Omega_h \times (0,1) , \\
	\dt (\rho_\varepsilon v_\varepsilon) + \dvh( \rho_\varepsilon v_\varepsilon\otimes v_\varepsilon) + \partial_z (\rho_\varepsilon w_\varepsilon v_\varepsilon) & \\
	~~~~ ~~~~ + \dfrac{1}{\varepsilon^2}\nablah P(\rho_\varepsilon) = \dvh \mathbb S(v_\varepsilon) + \partial_{zz} v_\varepsilon & \text{in} ~ \Omega_h \times (0,1) , \\
	\partial_z P(\rho_\varepsilon) = 0 & \text{in} ~ \Omega_h \times (0,1).
\end{cases}
\end{equation*}
where $ P(\rho_\varepsilon) = \rho_\varepsilon^\gamma $, $ \mathbb S (v_\varepsilon) = \mu (\nablah v_\varepsilon + \nablah v_\varepsilon^\top) + (\lambda - \mu) \dvh v_\varepsilon \mathbb{I}_2 $ represent the pressure for isentropic flows and the viscous stress tensor for Newtonian flows, respectively. Here, we assume that $ \mu, \lambda > 0 $ and $ \gamma > 1 $.
% In this work, we assume the constant viscosity coefficients $ \mu, \lambda > 0 $.
We consider $ \Omega_h = \mathbb R^2 $ or $ \mathbb T^2 $ in this paper, where $ \mathbb T^2 $ represents the periodic domain with period $ 1 $ in both directions in $ \mathbb R^2 $. \eqref{CPE-0} is complemented with the stress-free and non-permeable boundary conditions:
\begin{equation*}\tag{BC-CPE} \label{bc-cpe}
	\dz v_\varepsilon\big|_{z=0,1} = 0,~ w_\varepsilon\big|_{z=0,1} = 0.
\end{equation*}
Hereafter, we have and will use $ \nabla_h, \dvh $ and $ \Delta_h $ to represent the horizontal gradient, the horizontal divergence, and the horizontal Laplace operator, respectively; that is,
\begin{gather*}
	\nabla_h := \biggl(\begin{array}{c}
		\partial_x \\ \partial_y
	\end{array}\biggr) , ~ \dvh := \nabla_h \cdot, ~
	\Delta_h := \dvh \nabla_h \cdot .
\end{gather*}
Notice, \eqref{CPE-0} with \eqref{bc-cpe} is invariant with respect to the following symmetry:
\begin{equation}\label{SYM}\tag{SYM}
	 v_\varepsilon ~ \text{and} ~ w_\varepsilon ~\text{are even and odd, respectively, in the $ z $-variable}.
\end{equation}
Owing to such symmetry, in order to study the limit system of \eqref{CPE-0} as $ \varepsilon \rightarrow 0^+ $, it suffices to consider the following system:
\begin{equation}\label{CPE}
\begin{cases}
	\dt \rho_\varepsilon + \dvh (\rho_\varepsilon v_\varepsilon) + \dz (\rho_\varepsilon w_\varepsilon) = 0  & \text{in} ~ \Omega_h \times 2 \mathbb T , \\
	\dt (\rho_\varepsilon v_\varepsilon) + \dvh( \rho_\varepsilon v_\varepsilon\otimes v_\varepsilon) + \partial_z (\rho_\varepsilon w_\varepsilon v_\varepsilon) & \\
	~~~~ ~~~~ + \dfrac{1}{\varepsilon^2}\nablah P(\rho_\varepsilon) = \dvh \mathbb S(v_\varepsilon) + \partial_{zz} v_\varepsilon & \text{in} ~ \Omega_h \times 2 \mathbb T , \\
	\partial_z P(\rho_\varepsilon) = 0 & \text{in} ~ \Omega_h \times 2 \mathbb T,
\end{cases}
\end{equation}
subject to the periodic boundary condition and symmetry \eqref{SYM}.
Here $ 2 \mathbb T $ is the periodic domain with period $ 2 $ in $ \mathbb R $.

We remark that the restrictions of solutions to \eqref{CPE} in $ \Omega_h\times [0,1] $ solve \eqref{CPE-0} with \eqref{bc-cpe}, provided that the solutions exist and are regular enough.

We can rewrite \subeqref{CPE}{2} as, provided that $ \rho_\varepsilon > 0 $,
\begin{align*}
	& \dt v_\varepsilon  + v_\varepsilon\cdot \nablah v_\varepsilon + w_\varepsilon\dz v_\varepsilon  + \dfrac{1}{\varepsilon^2} \nablah \bigl( \dfrac{\gamma}{\gamma-1}  \rho_\varepsilon^{\gamma-1}\bigr) = \dfrac{\mu}{\rho_\varepsilon} \deltah v_\varepsilon \\
	& ~~~~ ~~~~ + \dfrac{\lambda}{\rho_\varepsilon} \nablah \dvh v_\varepsilon + \dfrac{1}{\rho_\varepsilon} \partial_{zz} v_\varepsilon.
\end{align*}
Also from the continuity equation \subeqref{CPE}{1}, one can derive
\begin{equation*}
	\dt \rho_\varepsilon^{\gamma-1} + v_\varepsilon\cdot\nablah \rho_\varepsilon^{\gamma-1} + (\gamma -1 ) \rho_\varepsilon^{\gamma-1} \dvh v_\varepsilon + (\gamma-1) \rho_\varepsilon^{\gamma-1} \dz w_\varepsilon = 0.
\end{equation*}
We consider $ \rho_\varepsilon^{\gamma-1} $ with perturbations around the constant state given by $ \rho_\varepsilon^{\gamma-1} = \bar\rho^{\gamma-1} \in (0,\infty) $. Then we define the perturbation variable $ \xi_\varepsilon $ by
$$ \xi_\varepsilon := \dfrac{1}{\varepsilon} \log \biggl( \dfrac{\gamma}{\gamma-1} \dfrac{\rho_\varepsilon^{\gamma-1}}{c^2} \biggr) , $$
with $ c^2 := \dfrac{\gamma}{\gamma-1} \bar\rho^{\gamma-1} $. Let $ \alpha := \dfrac{1}{\gamma-1} $, $ c_1 := \bigl( \dfrac{\gamma}{(\gamma-1)c^2} \bigr)^{\alpha} $. With such notations, $ \rho_\varepsilon = c_1^{-1} \exp{(\varepsilon \alpha \xi_\varepsilon)} $. Without loss of generality,  we take $ \bar \rho \equiv 1 $. Hence \eqref{CPE} can be written as
\begin{equation}\label{eq:perturbation}
	\begin{cases}
		\dt \xi_\varepsilon + v_\varepsilon\cdot\nablah \xi_\varepsilon + \dfrac{\gamma-1}{\varepsilon} ( \dvh v_\varepsilon+ \dz w_\varepsilon ) = 0 &\text{in} ~ \Omega_h \times 2 \mathbb T,\\
		\dt v_\varepsilon + v_\varepsilon \cdot \nablah v_\varepsilon + w_\varepsilon \dz v_\varepsilon + \dfrac{c^2 e^{\varepsilon \xi_\varepsilon}}{\varepsilon} \nablah \xi_\varepsilon \\
		~~~~ = c_1 e^{-\varepsilon \alpha\xi_\varepsilon} \bigl( \mu \deltah v_\varepsilon + \lambda \nablah \dvh v_\varepsilon + \partial_{zz} v_\varepsilon) & \text{in} ~ \Omega_h\times 2\mathbb T, \\
		\dz \xi_\varepsilon = 0 & \text{in} ~ \Omega_h\times 2\mathbb T,
%		\\
%		w_\varepsilon = 0 & \text{on} ~ \Omega_h \times \mathbb Z .
	\end{cases}
\end{equation}
complemented with periodic initial data with symmetry \eqref{SYM}.
%Notice, the dynamics of $ \xi_\varepsilon, v_\varepsilon, w_\varepsilon $ are fully determined by their dynamics in $ \Omega := \Omega_h \times (0,1) $.
In fact, \eqref{eq:perturbation} is already in the form of \eqref{eq:anti-symmetry-system}. In this work, we study the asymptotic limit of \eqref{eq:perturbation} as $ \varepsilon \rightarrow 0^+ $. In fact, we will demonstrate that the existence time of strong solutions to \eqref{eq:perturbation} has a uniform lower bound, independent of $ \varepsilon \in (0,\varepsilon_0) $ for some small $ \varepsilon_0\in (0,1) $. In addition, in the sense of distribution,
the limit equations of \eqref{eq:perturbation}, as $ \varepsilon \rightarrow 0^+ $, are the primitive equations \eqref{eq:primitive-eq} with the incompressible condition, below.
Also, the associated acoustic wave equations for the three dimensional system \eqref{eq:perturbation} are only two dimensional (see \eqref{acoustic-wave}, below). Consequently, we are able to adopt the strategy of studying the acoustic wave equations for the compressible Navier--Stokes equations in $ \mathbb R^2 $ and $ \mathbb T^2 $ to investigate the oscillatory part of the equations. This is done in section \ref{sec:limit-equation}, below.

The rest of this work is organized as follows. In the next section, we will introduce some notations as well as some function spaces. Also, the main theorems in this work are stated in both $ \Omega_h = \mathbb R^2 $ and $ \Omega_h = \mathbb T^2 $. Next, in section \ref{sec:uniform-stablitiy}, we establish the uniform local well-posedness of strong solutions to system \eqref{eq:perturbation}. This is done with uniform  {\it a priori} estimates in section \ref{subsec:apriori-est}, a local existence theorem in section \ref{subsec:local-well-posedness} and a continuity argument in section \ref{subsec:uniform-stability}. In section \ref{sec:limit-equation}, we first identity the primitive equations, i.e., \eqref{eq:primitive-eq}, below, as the limit of system \eqref{eq:perturbation} as $ \varepsilon \rightarrow 0^+ $ in the sense of distribution. Then in section \ref{subsec:convergence-whole}, we argue that the acoustic waves decay to zero in the case when $ \Omega_h = \mathbb R^2 $; in section \ref{subsec:convergence-periodic}, we study the oscillation equations and identify the limit equations of oscillations in the case when $ \Omega_h = \mathbb T^2 $. Consequently, we conclude the compactness in both cases as stated in the main theorems, below.

\section{Preliminaries and main theorems}
In this work, we denote by
\begin{align*}
	\overline f (x,y) : = \int_0^1 f (x,y,z') \,dz', ~ \text{and} ~ \widetilde f(x,y,z) := f(x,y,z) - \overline f(x,y),
\end{align*}
as the average and the fluctuation of any function $ f = f(x,y,z) $ over the $ z $-variable.
We use $ \partial_h $ to denote the horizontal derivatives, i.e., $ \partial_h \in \lbrace \partial_x, \partial_y \rbrace $. $ \dt $ and $ \dz $ denote the time derivative and the vertical derivative, respectively. For any function $ f $, $ \partial_g f $ is sometimes denoted as $ f_g $, for $ g \in \lbrace t, x, y, z, h \rbrace $. Similar notations are also adopted for higher order derivatives.

As in \cite{Lions1992}, we introduce the following function spaces. Denote by
\begin{align*}
	& C_\sigma^\infty : = \lbrace u \in C_0^\infty (\Omega_h \times2\mathbb T ; \mathbb R^2) |  \int_0^1 \dvh u \,dz = 0 \rbrace,\\
	& C_\tau^\infty : = \lbrace u \in C_0^\infty (\Omega_h \times2\mathbb T ; \mathbb R^2) | u = \nablah \psi, ~ \text{for} ~ \psi \in C^\infty(\Omega_h; \mathbb R^2) \rbrace.
\end{align*}
Then with respect to the $ L^2 $-inner product, $ C_\sigma^\infty \perp C_\tau^\infty $ and $ C^\infty_0 = C_\sigma^\infty \cup C_\tau^\infty  $. We denote the closures in the $ L^2 $ norm of $ C_\sigma^\infty , C_\tau^\infty $ as $ L_\sigma^2 = L_\sigma^2(\Omega_h \times2\mathbb T;\mathbb R^2), L_\tau^2 := L_\tau^2(\Omega_h \times2\mathbb T;\mathbb R^2) $, respectively. $ \mathcal D' = \mathcal D_\sigma' \cup \mathcal D_\tau', H^s = H_\sigma^s\cup H_\tau^s $ are the spaces of test functions and Sobolev functions in $ \Omega_h \times 2 \mathbb T $, defined similarly for $ s \in \mathbb Z^+ $. We define the projection operators $ \mathcal P_\sigma, \mathcal P_\tau $ in the following.
Let $ u \in C_0^\infty(\Omega_h\times 2\mathbb T; \mathbb R^2) $. Consider the elliptic problem
\begin{equation}\label{def:projection-poission-equation}
	\begin{gathered}
	\deltah \psi_u = \int_0^1 \dvh u \,dz, \\%~~ \int_{\Omega_h} \psi_u \idxh = 0,\\
	\begin{cases}
	\lim\limits_{\abs{(x,y)}{} \rightarrow \infty} \psi_u = 0 & \text{in the case when} ~ \Omega_h = \mathbb R^2, \\
	\int_{\Omega_h} \psi_u \idxh = 0 & \text{in the case when} ~ \Omega_h = \mathbb T^2.	
	\end{cases}
	\end{gathered}
\end{equation}
%with $ \psi_u \in C_0^\infty (\Omega_h; \mathbb R) $.
Let
\begin{equation}\label{def:projection}
	\mathcal P_\tau u : = \nablah \psi_u, ~~ \mathcal P_\sigma u : = u - \nablah \psi_u.
\end{equation}
Then $ \mathcal P_\sigma, \mathcal P_\tau $ are the projection operators from $ C_0^\infty $ to $ C^\infty_\sigma, C^\infty_\tau $, respectively. Also, by a density argument, $ \mathcal P_\sigma, \mathcal P_\tau $ can act on functions in $ L^2 , H^s $, $ s \in(0,\infty) $.
In particular, the standard elliptic estimates yield, provided that $ u \in H^s(\Omega_h\times 2\mathbb T ;\mathbb R^2) $,  $ s \in (0,\infty) $,
\begin{equation*}
	\normh{\psi_u}{H^{s+1}} \lesssim \normh{\overline u}{H^s}\lesssim \norm{u}{H^s}.
\end{equation*}
Thus $ \mathcal P_\sigma u \in H^s_\sigma, \mathcal P_\tau u \in H^s_\tau $, $ u = \mathcal P_\sigma u + \mathcal P_\tau u $ and
\begin{equation}\label{def:projcetion-boundness}
	\norm{\mathcal P_\sigma u }{H^s} \lesssim \norm{u}{H^s}, ~~\norm{\mathcal P_\tau u }{H^s} \lesssim \norm{u}{H^s}.
\end{equation}

System \eqref{eq:perturbation} is complemented with initial data $ (\xi_0, v_0) \in H^2(\Omega_h\times 2\mathbb T;\mathbb R) \times H^2(\Omega_h\times 2\mathbb T;\mathbb R^2) $ with the compatibility conditions:
\begin{equation}\label{cmpbc}
\dz \xi_0 = 0; ~ \text{$ v_0$ is even in the $ z $-variable.}
\end{equation}
Also, in the case when $ \Omega_h = \mathbb R^2 $, the far field condition
\begin{equation}\label{far-fields}
%	\lim\limits_{\abs{(x,y)}{}\rightarrow \infty}\abs{\xi_\varepsilon}{} = 0,
	\lim\limits_{\abs{(x,y)}{}\rightarrow \infty} \abs{v_\varepsilon}{} = 0,
\end{equation}
is also imposed. We denote the constant $ M \in (0,\infty) $ satisfying
\begin{equation}\label{initial-data}
%	\begin{gathered}
		\dfrac{1}{2}\norm{v_0}{H^2}^2 + \dfrac{c^2}{\gamma-1} \norm{\xi_0}{H^2}^2 < M,
%	\end{gathered}
\end{equation}
to be the bound of initial data $ (\xi_0, v_0) $.
Now we describe our main theorems in this work. The first theorem is stating that the $ H^2 $ norms of the solutions to \eqref{eq:perturbation} are uniformly bounded with respect to $ \varepsilon $, provided that $ \varepsilon $ is small enough.
\begin{thm}\label{thm:uniform-stability} With initial data  $(\xi_0, v_0) \in H^2(\Omega_h\times 2\mathbb T;\mathbb R) \times H^2(\Omega_h\times 2\mathbb T;\mathbb R^2)  $ satisfying the compatibility conditions in \eqref{cmpbc}, let $ M $ be the upper bound of the initial data, i.e., \eqref{initial-data}. Then
	there are positive constants $ \varepsilon_1 \in (0,1), T \in (0,\infty) $ depending only $ M $, such that for any $ \varepsilon \in (0,\varepsilon_1) $, there is a unique strong solution $ (\xi_\varepsilon, v_\varepsilon) $  to system \eqref{eq:perturbation} in the time interval $ [0, T] $.  $ (\xi_\varepsilon, v_\varepsilon) $ satisfies
\begin{equation}\label{thm-regularity}
	\begin{gathered}
		\xi_\varepsilon \in L^\infty(0,T; H^2(\Omega_h\times 2\mathbb T)), ~\dt \xi_\varepsilon \in L^\infty(0,T; H^1(\Omega_h\times 2\mathbb T)),\\
		v_\varepsilon \in L^\infty(0,T; H^2(\Omega_h\times 2\mathbb T)) \cap L^2(0,T; H^3(\Omega_h\times 2\mathbb T)), \\
		\dt v_\varepsilon \in L^\infty(0,T; L^2(\Omega_h\times 2\mathbb T)) \cap L^2(0,T; H^1(\Omega_h\times 2\mathbb T)).
	\end{gathered}
\end{equation}
In addition, there exist positive constants $ M_0, M_1, M_2 \in (0,\infty) $ independent of $ \varepsilon $ such that
\begin{equation}\label{thm-uniform-stability}
	\begin{gathered}
		\norm{\xi_\varepsilon}{L^\infty(0,T;\Hnorm{2}(\Omega_h\times 2\mathbb T))}^2 < M_0, ~ \norm{v_\varepsilon}{L^\infty(0,T;\Hnorm{2}(\Omega_h\times 2\mathbb T))}^2 < M_1,\\
		\norm{\nabla v_\varepsilon}{L^2(0,T;\Hnorm{2}(\Omega_h\times 2\mathbb T))}^2 < M_2.
	\end{gathered}
\end{equation}
\end{thm}

Next, as $ \varepsilon \rightarrow 0^+ $, formally, we expect the solutions $ v_\varepsilon $ obtained in Theorem \ref{thm:uniform-stability} will converge to a solution to the following primitive equations:
\begin{equation}\label{eq:primitive-eq}
	\begin{cases}
		\dvh v_p + \dz w_p = 0  & \text{in} ~ \Omega_h\times 2 \mathbb T,  \\
		\dt v_p + v_p\cdot \nablah v_p + w_p\dz v_p + \nablah P \\
		~~~~ ~~~~ = c_1 (\mu \deltah v_p + \lambda\nablah \dvh v_p +\partial_{zz}v_p)   & \text{in} ~ \Omega_h\times 2 \mathbb T, \\
		\dz P = 0  & \text{in} ~ \Omega_h\times 2 \mathbb T,%\\
%		w_p = 0 & \text{on} ~ \Omega_h\times \mathbb Z ,
	\end{cases}
\end{equation}
complemented
with initial data $ v_p|_{t=0} = v_{p,0} = \mathcal P_\sigma v_0 \in H^2_\sigma(\Omega_h\times 2\mathbb T) $. The convergence behaviors are different in the case when $ \Omega_h = \mathbb R^2 $ and the case when $ \Omega_h = \mathbb T^2 $. We summarize the results in the following:
\begin{thm}\label{thm:convergence}
	Under the same assumptions as in Theorem \ref{thm:uniform-stability}, as $ \varepsilon \rightarrow 0^+ $, one has
	\begin{equation}\label{thm-weak-convergence}
		\begin{gathered}
			v_\varepsilon \buildrel\ast\over\rightharpoonup v_p ~~~ \text{weak-$\ast$ in} ~ L^\infty(0,T; H^2(\Omega_h\times 2\mathbb T)), \\
			v_\varepsilon \rightharpoonup v_p ~~~ \text{weakly in} ~ L^2(0,T; H^3(\Omega_h\times 2\mathbb T)),
		\end{gathered}
	\end{equation}
	where $ v_p \in L^\infty(0,T;H^2(\Omega_h\times 2\mathbb T)) \cap L^2(0,T; H^3(\Omega_h\times 2\mathbb T)) $ is the unique strong solution to \eqref{eq:primitive-eq} with initial data $ v_p|_{t=0} = v_{p,0} = \mathcal P_\sigma v_0 \in H^2_\sigma(\Omega_h\times 2\mathbb T) $. Moreover, the following strong convergence holds,
	\begin{equation}\label{thm-strong-convergence}
		\mathcal P_\sigma v_\varepsilon \rightarrow v_p ~~~ \text{in} ~ C([0,T];H^1_{\sigma,loc} (\Omega_h\times 2\mathbb T)) \cap L^2(0,T;H^2_{\sigma,loc}(\Omega_h\times 2\mathbb T)).
	\end{equation}
	In addition,
	\begin{itemize}
		\item in the case when $ \Omega_h = \mathbb R^2 $, as $ \varepsilon \rightarrow 0^+ $,
		\begin{equation}\label{aw:strong-convergence-01}
				\xi_\varepsilon \rightarrow 0, ~ \mathcal P_\tau v_\varepsilon \rightarrow 0  ~~~ \text{in} ~ L^2(0,T;W^{\frac 1 2 , 6}(\mathbb R^2)),
			\end{equation}
		and therefore,
		\begin{equation}\label{thm-strong-convergence-whole}
		%	\begin{aligned}
			\norm{v_\varepsilon - v_p}{L^2(0,T;L^6_{loc}(\mathbb R^2\times 2 \mathbb T))} + \norm{\xi_\varepsilon}{L^2(0,T;L^6(\mathbb R^2\times 2 \mathbb T))} \rightarrow 0;
		%	\end{aligned}
		\end{equation}
		\item in the case when $ \Omega_h = \mathbb T^2 $, as $ \varepsilon \rightarrow 0^+ $, only the weak convergences in \eqref{thm-weak-convergence} hold, and there exists a function $ V^o \in L^\infty(0,T;H^2(\mathbb T^2)) \cap C(0,T;H^1(\mathbb T^2)) $ and a constant $ g^o $, satisfying equation \eqref{aw-per:limit-eq-osc}, below in section \ref{subsec:convergence-periodic}, such that, as $ \varepsilon \rightarrow 0^+ $,
		\begin{equation}\label{thm-strong-convergence-periodic}
%			\begin{gathered}
				\norm{\biggl( \begin{array}{c}\xi_\varepsilon - g^o\\
				 v_\varepsilon - v_p	
				 \end{array}\biggr)
				 - \mathcal L \bigl(\dfrac t \varepsilon \bigr) V^o }{L^\infty(0,T;H^1(\mathbb T^2\times 2 \mathbb T))} \rightarrow 0
% ,\\
%			 	\norm{v_\varepsilon - v_p - (\mathcal L \bigl(\dfrac t \varepsilon \bigr) V^o)_2}{L^\infty(0,T;H^1(\mathbb T^2 \times 2 \mathbb T))} \rightarrow 0,
%			 \end{gathered}
		\end{equation}
		where $ \mathcal L $ is the solution operator to equation \eqref{aw:def-group-operator}, defined in \eqref{aw:def-group-operator-00}, below in section \ref{sec:limit-equation}.
	\end{itemize}
\end{thm}

Theorem \ref{thm:uniform-stability} is the direct consequence of Proposition \ref{prop:uniform-stability}, below. Theorem \ref{thm:convergence} is the consequence of Propositions \ref{prop:convergence-eq-distribution}, \ref{prop:strong-convergence-whole-space}, \ref{prop:convergence-periodic}, below.

%%%%%%%%%%%%%%%%%%%%

\section{Uniform stability}\label{sec:uniform-stablitiy}
In this section, we will establish the uniform local existence of solutios to \eqref{eq:perturbation} with respect to $ \varepsilon \in (0,\varepsilon_0) $ for some $ \varepsilon_0 \in (0,\infty) $ small enough. This is done via a series of {\it a priori} estimates, a local well-posedness theorem and a continuity argument. To simplify the presentation, in this section, we shorten the notations by dropping the subscript $ \varepsilon $. That is, we denote $ \xi = \xi_\varepsilon, v = v_\varepsilon, w = w_\varepsilon $.

\subsection{\textit{\textbf{A priori}} estimates}\label{subsec:apriori-est}
We first establish some {\it a priori} estimates, which are independent of $ \varepsilon $. Indeed, for $ \varepsilon $ small enough, the {\it a priori} estimates obtained in this subsection allow us to establish a uniform existence time in subsection \ref{subsec:uniform-stability}.
We remind readers system \eqref{eq:perturbation}:% in $ \Omega_h\times 2\mathbb T  $:
\begin{equation*}\tag{\ref{eq:perturbation}}
	\begin{cases}
		\dt \xi + v\cdot\nablah \xi + \dfrac{\gamma-1}{\varepsilon} ( \dvh v + \dz w ) = 0 & \text{in} ~ \Omega_h\times 2\mathbb T ,\\
		\dt v + v \cdot \nablah v + w \dz v + \dfrac{c^2 e^{\varepsilon \xi}}{\varepsilon} \nablah \xi \\
		~~~~ = c_1 e^{-\varepsilon \alpha\xi} \bigl( \mu \deltah v + \lambda \nablah \dvh v + \partial_{zz} v) & \text{in} ~ \Omega_h\times 2\mathbb T , \\
		\dz \xi = 0 & \text{in} ~ \Omega_h\times 2\mathbb T. % , \\
%		w = 0 & \text{on} ~ \Omega_h\times \lbrace z \in \mathbb Z \rbrace.
	\end{cases}
\end{equation*}
We will show the following:
\begin{prop}\label{prop:apriori-estimate}
Let $ (\xi, v) $ be a local strong solution to \eqref{eq:perturbation} in the time interval $ [0, T] $, $ T \in (0,\infty) $, and $ (\xi, v)\bigr|_{t=0} = (\xi_0, v_0) \in H^2(\Omega_h\times 2\mathbb T) $ with the compatibility conditions in \eqref{cmpbc}. Then the following inequality holds: for any $ \delta \in (0,1) $ with corresponding $ C_\delta \simeq \delta^{-1} $,
\begin{equation}\label{ue:spatial-derivative}
%	\boxed{
	\begin{aligned}
		& \sup_{0\leq t\leq T} \biggl\lbrace \dfrac{1}{2} \norm{v(t)}{H^2}^2 +  \dfrac{c^2}{2(\gamma-1)} e^{-\varepsilon \norm{\xi(t)}{H^2}} \norm{\xi(t)}{H^2}^2 \biggr\rbrace \\
		& ~ + c_1 \int_0^T \biggl( \ e^{-\varepsilon \alpha\norm{\xi(t)}{H^2}} \bigl( \mu \norm{\nablah v(t)}{H^2}^2   + \lambda \norm{\dvh v(t)}{H^2}^2 \\
		& ~+ \norm{\dz v(t)}{H^2}^2  \bigr) \biggr) \,dt \leq \dfrac{1}{2} \norm{v_0}{H^2}^2 +  \dfrac{c^2e^{\varepsilon \norm{\xi(0)}{H^2}}}{2(\gamma-1)}  \norm{\xi_0}{H^2}^2 \\
		& ~ + \delta \int_0^T \norm{\nabla v(t)}{H^2}^2 \,dt +\varepsilon  \int_0^T \biggl( \mathcal H_1 (\norm{\xi(t)}{\Hnorm{2}},\norm{v(t)}{\Hnorm{2}} ) \\
		& ~ \times \norm{\nabla v(t)}{H^2}^2 \biggr) \,dt  + C_\delta \int_0^T \mathcal H_2(\norm{\xi(t)}{\Hnorm{2}},\norm{v(t)}{\Hnorm{2}} ) \,dt,
	\end{aligned}
%	}
\end{equation}
where $ \mathcal H_1(\cdot), \mathcal H_2(\cdot) $ are smooth and bounded functions of the arguments. Also, $ \mathcal H_1(0) = 0, \mathcal H_2(0) = 0 $. Moreover, with the same notations, below, we have the inequalities:
\begin{align}
	& \label{ue:2001}
	\norm{\dt e^{\varepsilon\xi}}{\Hnorm{1}} \leq \mathcal H_2(\norm{\xi}{\Hnorm{2}},\norm{v}{\Hnorm{2}}), \\
	& \norm{\dt \mathcal P_\sigma v}{\Lnorm2} \leq \mathcal H_2(\norm{\xi}{\Hnorm{2}},\norm{v}{\Hnorm{2}}),  \label{ue:203} \\
	& \norm{\dt \mathcal P_\sigma v}{H^1} \leq \mathcal H_1(\norm{\xi}{\Hnorm{2}},\norm{v}{\Hnorm{2}}) \norm{\nabla v}{H^2} \nonumber\\
	& ~~~~ + \mathcal H_2(\norm{\xi}{\Hnorm{2}},\norm{v}{\Hnorm{2}}),	 \label{ue:204}\\
%	& \norm{\dt \xi}{\Lnorm2} \leq \varepsilon^{-1} \mathcal H_2(\norm{\xi}{\Hnorm{2}},\norm{v}{\Hnorm{2}}), {\label{ese:001}} \\
	& \norm{\dt \xi}{H^1} \leq \varepsilon^{-1} \mathcal H_2(\norm{\xi}{\Hnorm{2}},\norm{v}{\Hnorm{2}}), {\label{ese:002}}\\
	& \norm{\dt v}{\Lnorm 2} \leq (1 + \varepsilon^{-1}) \mathcal H_2(\norm{\xi}{\Hnorm{2}},\norm{v}{\Hnorm{2}}) , {\label{ese:003}} \\
	& \norm{\dt v}{H^1} \leq \mathcal H_1(\norm{\xi}{\Hnorm{2}},\norm{v}{\Hnorm{2}}) \norm{\nabla v}{H^2} {\nonumber}\\
	& ~~~~ + (1+\varepsilon^{-1}) \mathcal H_2(\norm{\xi}{\Hnorm{2}},\norm{v}{\Hnorm{2}}). {\label{ese:004}}
\end{align}
\end{prop}
In order to establish such {\it a priori} estimates, we first represent the vertical velocity $ w $ in terms of $ (\xi, v) $. In fact, after averaging \subeqref{eq:perturbation}{1} in the $ z $-variable, one has
\begin{equation*}
%	\begin{gathered}
		\dt \xi + \overline{v} \cdot \nablah \xi + \dfrac{\gamma-1}{\varepsilon} \dvh \overline v = 0,
%	\end{gathered}
\end{equation*}
{and consequently}, after comparing the above equation with \subeqref{eq:perturbation}{1}, it follows that
\begin{equation*}
\widetilde v \cdot\nablah \xi + \dfrac{\gamma-1}{\varepsilon} ( \dvh \widetilde v + \dz w ) = 0.
\end{equation*}
Then the following representations of the vertical velocity and its derivatives hold (recall that $ \alpha = 1/(\gamma -1 ) $):
\begin{align}
	{\label{def:vertical-velocity}} w & = - \int_0^z \biggl( \varepsilon\alpha \widetilde v \cdot\nablah \xi + \dvh \widetilde v \biggr) \,dz ,\\
	w_z & = - \varepsilon\alpha \widetilde v \cdot\nablah \xi - \dvh \widetilde v , \label{vertical-velocity-dz} \\
	w_h & = - \int_0^z \biggl(\varepsilon\alpha (\widetilde v \cdot\nablah \xi)_h + \dvh \widetilde v_h \biggr) \,dz , \label{vertical-velocity-dh}\\
	w_{zz} & =  - \varepsilon \alpha \widetilde v_z \cdot \nablah \xi - \dvh \widetilde v_z \label{vertical-velocity-dzz},\\
	w_{hz} & = - \varepsilon \alpha \widetilde v_h \cdot \nablah \xi - \varepsilon \alpha \widetilde v \cdot \nablah \xi_h - \dvh \widetilde v_h. \label{vertical-velocity-dhz}
%	& w_{hh} = - \int_0^z \dfrac{\varepsilon}{\gamma-1} (\tilde v \cdot\nablah \xi)_{hh} + \dvh \tilde v_{hh} \,dz ,
\end{align}
In the following, we separate the proof of Proposition \ref{prop:apriori-estimate} in three parts: estimates on the horizontal derivatives; estimates on the vertical derivatives; and estimates on the time derivatives.

\subsubsection*{Estimates on the horizontal derivatives}
Denote by $ \partial_h \in \lbrace \partial_x,\partial_y \rbrace $. Then after applying $ \partial_h^2 = $ to \eqref{eq:perturbation}, it follows
\begin{equation}\label{eq:horizontal-derivative}
	\begin{cases}
		\dt \xi_{hh} + \dfrac{\gamma-1}\varepsilon ( \dvh v_{hh} + \dz w_{hh} ) = \mathcal G & \text{in} ~ \Omega_h\times 2\mathbb T, \\
		\dt v_{hh} + \dfrac{c^2e^{\varepsilon \xi}}{\varepsilon} \nablah \xi_{hh} - c_1 e^{-\varepsilon \alpha\xi} ( \mu \deltah v_{hh} \\
		~~~~ + \lambda\nablah \dvh v_{hh} + \partial_{zz} v_{hh} ) =  \mathcal F_1 + \mathcal F_2 + \mathcal F_3 + \mathcal F_4 & \text{in} ~ \Omega_h\times 2\mathbb T,
	\end{cases}
\end{equation}
where we have denoted by
\begin{equation*}
	\begin{aligned}
		& \mathcal G : =  - (v\cdot\nablah \xi)_{hh}, \\
		& \mathcal F_1 : = - (v\cdot \nablah v)_{hh} - (w\dz v)_{hh},\\
		& \mathcal F_2 : =   - c^2 e^{\varepsilon \xi}(\xi_{hh} + \varepsilon \xi_{h}^2 ) \nablah \xi - 2 c^2 e^{\varepsilon\xi} \xi_{h}\nablah \xi_{h}, \\
		& \mathcal F_3: = - 2 \varepsilon c_1 \alpha e^{-\varepsilon \alpha \xi} \xi_h (\mu \deltah v_{h} + \lambda \nablah \dvh v_h + \partial_{zz} v_h ),\\
		& \mathcal F_4: = - \varepsilon c_1\alpha e^{-\varepsilon \alpha \xi} (\xi_{hh} - \varepsilon \alpha \xi_h^2 ) (\mu \deltah v + \lambda \nablah \dvh v + \partial_{zz} v ).
	\end{aligned}
\end{equation*}
%Hereafter, it is denoted by $ \cdot_{hh} = \partial_h^2 \cdot $ and $ \cdot_h = \partial_h \cdot $.
After taking the $ L^2 $-inner product of \subeqref{eq:horizontal-derivative}{2} with $ v_{hh} $ in $ \Omega_h\times 2\mathbb T $, it holds
\begin{align*}
	& \dfrac{d}{dt} \biggl\lbrace \dfrac 1 2 \int \abs{v_{hh}}{2} \idx \biggr\rbrace + \int \dfrac{c^2 e^{\varepsilon \xi}}{\varepsilon} \nablah \xi_{hh} \cdot v_{hh} \idx + c_1 \int \biggl( e^{-\varepsilon \alpha \xi} \bigl( \mu \abs{\nablah v_{hh}}{2} \\
	& ~~~~ + \lambda \abs{\dvh v_{hh}}{2}
	 + \abs{v_{hhz}}{2} \bigr) \biggr) \idx
%	\\
%	& ~~~~
	=  \int \mathcal F_1 \cdot v_{hh} \idx + \int \mathcal F_2 \cdot v_{hh} \idx \\
	& ~~~~ +  \int \mathcal F_3 \cdot v_{hh} \idx +  \int \mathcal F_4 \cdot v_{hh} \idx
	 + \varepsilon c_1 \alpha \int \biggl( e^{-\varepsilon \alpha \xi} \bigl( \mu (\nablah \xi \cdot \nablah v_{hh}) \cdot v_{hh}\\
	& ~~~~ + \lambda ( v_{hh} \cdot \nablah \xi) \dvh v_{hh} \bigr) \biggr) \idx  =: I_1 + I_2 + I_3 + I_4 + I_5.
\end{align*}
On the other hand, after applying integration by parts and substituting \subeqref{eq:horizontal-derivative}{1}, one has
\begin{align*}
	& \int \dfrac{c^2 e^{\varepsilon \xi}}{\varepsilon} \nablah \xi_{hh} \cdot v_{hh} \idx = - \int \dfrac{c^2 e^{\varepsilon \xi}}{\varepsilon} \xi_{hh} \dvh v_{hh} \idx \\
	& ~~~~ - \int c^2 e^{\varepsilon \xi} \xi_{hh} ( v_{hh}\cdot \nablah \xi ) \idx
	= \int \dfrac{c^2 e^{\varepsilon\xi}}{\gamma-1} \xi_{hh} \bigl( \dt \xi_{hh} - \mathcal G - \dfrac{\gamma -1}{\varepsilon}\dz w_{hh}  \bigr) \idx \\
	& ~~~~ - \int c^2 e^{\varepsilon \xi} \xi_{hh} ( v_{hh}\cdot \nablah \xi ) \idx
	 = \dfrac{d}{dt} \biggl\lbrace \dfrac{c^2}{2(\gamma-1)} \int e^{\varepsilon\xi}\abs{\xi_{hh}}{2} \idx   \biggr\rbrace \\
	& ~~~~ - \dfrac{c^2}{2(\gamma-1)} \int e^{\varepsilon \xi} \varepsilon \xi_t \abs{\xi_{hh}}{2} \idx - \dfrac{c^2}{\gamma-1} \int e^{\varepsilon\xi} \xi_{hh}\mathcal G \idx \\
	& ~~~~ - c^2 \int e^{\varepsilon \xi} \xi_{hh} (v_{hh}\cdot \nablah \xi) \idx.
\end{align*}
Then after writing
\begin{align*}
	& I_6 : = \dfrac{c^2}{2(\gamma-1)} \int e^{\varepsilon \xi} \varepsilon \xi_t \abs{\xi_{hh}}{2} \idx,
	& I_7 : = \dfrac{c^2}{\gamma-1} \int e^{\varepsilon\xi} \xi_{hh}\mathcal G \idx, \\
	& I_8 : = c^2 \int e^{\varepsilon \xi} \xi_{hh} (v_{hh}\cdot \nablah \xi) \idx,
\end{align*}
one has
\begin{equation}\label{ue:001}
%\color{blue}
	\begin{aligned}
		& \dfrac{d}{dt} \biggl\lbrace \dfrac 1 2 \norm{v_{hh}}{\Lnorm{2}}^2 + \dfrac{c^2}{2(\gamma-1)} \norm{e^{\varepsilon\xi /2}\xi_{hh}}{\Lnorm{2}}^2 \biggr\rbrace + c_1\bigl( \mu \norm{e^{-\varepsilon\alpha\xi/2} \nablah v_{hh}}{\Lnorm{2}}^2  \\
		& ~~~~ ~~~~ + \lambda \norm{e^{-\varepsilon\alpha\xi/2} \dvh v_{hh}}{\Lnorm{2}}^2 + \norm{e^{-\varepsilon\alpha\xi/2} v_{hhz}}{\Lnorm{2}}^2 \bigr) = \sum_{i=1}^8 I_i.
	\end{aligned}
\end{equation}

Now we will estimate the right-hand side of \eqref{ue:001}.
%%%%%%%%%%%%%%%%%%
%\begin{center}
%\boxed{
%\color{purple}
%I_{1} \rightarrow I_{3}, I_{4} \rightarrow I_{5} \rightarrow I_{7} \rightarrow I_{2}, I_{6}, I_{8}
%}
%\end{center}
%%%%%%%%%%%%%%%%%%
After applying integration by parts, $ I_1 $ can be written as
\begin{align*}
	& I_1 = % - \int \bigl( (v\cdot\nablah v)_{hh} + (w\dz v)_{hh} \bigr) \cdot v_{hh} \idx =
	\int \biggl(  \dfrac{1}{2} \dvh v \abs{v_{hh}}{2} - ( 2 v_h \cdot\nablah v_h+ v_{hh} \cdot\nablah v) \cdot v_{hh} \biggr) \idx\\
	& ~~~~ + \int \dfrac{1}{2} w_z \abs{v_{hh}}{2} \idx + \int  \biggl( w_h \dz v_h \cdot v_{hh} + w_h \dz v \cdot v_{hhh} \\
	& ~~~~ - 2 w_{h} \dz v_{h} \cdot v_{hh} \biggr) \idx =: I_1' + I_1'' + I_1'''.
\end{align*}
Then it follows, after substituting \eqref{vertical-velocity-dz} and \eqref{vertical-velocity-dh}, that
\begin{align*}
	& I_1' \lesssim \int \abs{\nablah v}{} \abs{ \nablah^2 v}{2} \idx \lesssim \norm{\nablah v}{\Lnorm{2}} \norm{\nablah^2 v}{\Lnorm{3}}\norm{ \nablah^2 v}{\Lnorm{6}} % \lesssim \norm{v_h}{2}(\norm{v_{hh}}{2}^{1/2} \norm{\nabla v_{hh}}{2}^{1/2} + \norm{v_{hh}}{2} ) \\
%	& ~~~~ ~~~~ \times (\norm{\nabla v_{hh}}{2} + \norm{v_{hh}}{2} )
	\\
	& ~~ \lesssim \norm{\nablah v}{\Lnorm{2}} \bigl( \norm{\nablah^2 v}{\Lnorm{2}}^{1/2} \norm{\nabla\nablah^2 v}{\Lnorm{2}}^{1/2} + \norm{\nablah^2 v}{\Lnorm{2}}\bigr) \bigl( \norm{\nablah^2 v}{\Lnorm{2}} \\
	& ~~~~  + \norm{\nabla \nablah^2 v}{\Lnorm{2}} \bigr)
	\lesssim \delta \norm{\nabla \nablah^2 v}{\Lnorm{2}}^2 + C_\delta \bigl(1 + \norm{\nablah v}{\Lnorm{2}}^4\bigr) \norm{\nablah^2 v}{\Lnorm{2}}^2, \\
	& I_1'' = - \dfrac{1}{2} \int \bigl( \dfrac{\varepsilon}{\gamma-1} \widetilde v \cdot\nablah \xi + \dvh \widetilde v \bigr) \abs{v_{hh}}{2} \idx = \dfrac{1}{2} \int \biggl( \bigl( \dfrac{\varepsilon}{\gamma-1} \xi \dvh \widetilde v \\
	& ~~~~ - \dvh \widetilde v \bigr) \abs{v_{hh}}{2} \biggr) \idx  + \int \dfrac{\varepsilon}{\gamma-1} \xi ( \widetilde v \cdot\nablah v_{hh}) \cdot v_{hh} \idx\\
	& ~~ \lesssim \varepsilon \int \abs{\xi}{} ( \abs{ \nablah v}{} \abs{\nablah^2 v}{2} + \abs{v}{}\abs{\nablah^2 v}{} \abs{\nablah^3 v}{} ) \idx  + \int \abs{\nablah v}{}\abs{\nablah^2 v}{2} \idx \\
	& ~~ \lesssim \varepsilon \norm{\xi}{\Lnorm{\infty}} \norm{\nablah^2 v}{\Lnorm{3}} \bigl( \norm{\nablah v}{\Lnorm{2}}  \norm{\nablah^2 v}{\Lnorm{6}} + \norm{v}{\Lnorm{6}}\norm{\nablah^3 v}{\Lnorm{2}} \bigr)\\
	& ~~~~ + \norm{\nablah v}{\Lnorm{2}} \norm{\nablah^2 v}{\Lnorm{3}} \norm{\nablah^2 v}{\Lnorm{6}} \lesssim   \norm{v}{\Hnorm{1}} \bigl( 1 + \varepsilon \norm{\xi}{\Hnorm{2}} \bigr) \\
	& ~~~~ \times \bigl( \norm{\nablah^2 v}{\Lnorm{2}}^{1/2} \norm{\nabla \nablah^2 v}{\Lnorm{2}}^{1/2} + \norm{\nablah^2 v}{\Lnorm{2}} \bigr) \bigl( \norm{\nabla \nablah^2 v}{\Lnorm{2}} + \norm{\nablah^2 v}{\Lnorm{2}} \bigr) \\
	& ~~ \lesssim \delta \norm{\nabla \nablah^2 v}{\Lnorm{2}}^2 + C_\delta \bigl( \varepsilon^4 \norm{\xi}{H^2}^4 + 1 \bigr) \bigl( \norm{v}{H^1}^4 + 1 \bigr) \norm{\nablah^2 v}{\Lnorm{2}}^2,\\
	& I_1''' = - \int \biggl\lbrack \int_0^z \biggl( \dfrac{\varepsilon}{\gamma-1} (\widetilde{v} \cdot \nablah \xi )_h + \dvh \widetilde{v}_h \biggr) \,dz \times \biggl( \dz v_h \cdot v_{hh} + \dz v \cdot v_{hhh} \\
	& ~~~~ - 2  \dz v_h \cdot v_{hh} \biggr) \biggr\rbrack \idx \lesssim \int_0^1 \varepsilon \bigl( \normh{v}{\Lnorm{\infty}} \normh{\nablah^2 \xi}{\Lnorm{2}}  + \normh{\nablah v}{\Lnorm4} \normh{\nablah \xi}{\Lnorm4} \bigr) \,dz \\
	& ~~~~ \times \int_0^1 \bigl( \normh{\dz v_h}{\Lnorm4} \normh{v_{hh}}{\Lnorm4} + \normh{\dz v}{\Lnorm\infty} \normh{v_{hhh}}{\Lnorm2} \bigr) \,dz  + \int_0^1 \normh{\nablah^2 v}{\Lnorm3}\,dz \\
	& ~~~~ \times  \int \bigl( \normh{\dz v_h}{\Lnorm3}\normh{v_{hh}}{\Lnorm3} + \normh{\dz v}{\Lnorm6} \normh{v_{hhh}}{\Lnorm2} \bigr) \,dz
%	& ~~~~ \lesssim \bigl( \varepsilon \norm{v}{H^2} \norm{\xi}{H^2} + \norm{v}{H^2} \bigr) \cdot \int_0^1 (\normh{\dz v_h}{2}^{1/2} \normh{\nablah \dz v_{h}}{2}^{1/2} + \normh{\dz v_h}{2}) \\
%	& ~~~~ ~~~~ \times (\normh{v_{hh}}{2}^{1/2} \normh{\nablah v_{hh}}{2}^{1/2} + \normh{v_{hh}}{2}) + (\normh{\dz v}{H^1} + 1)\log^{1/2} (e + \normh{\dz v}{H^2}) \normh{v_{hhh}}{2} \,dz \\
%	& ~~~~
	\lesssim \int_0^1 \varepsilon \normh{v}{H^2} \normh{\xi}{H^2} \,dz \\
	& ~~~~ \times \int_0^1 \biggl(  (\normh{\dz v_h}{\Lnorm2}^{1/2} \normh{\nablah \dz v_{h}}{\Lnorm2}^{1/2} + \normh{\dz v_h}{\Lnorm2})  (\normh{v_{hh}}{\Lnorm2}^{1/2} \normh{\nablah v_{hh}}{\Lnorm2}^{1/2} + \normh{v_{hh}}{\Lnorm2})\\
	& ~~~~ + \normh{\dz v}{H^2} \normh{v_{hhh}}{\Lnorm2} \biggr) \,dz + \int_0^1 \bigl( \normh{\nablah^2 v}{\Lnorm{2}}^{2/3} \normh{\nablah^3 v}{\Lnorm{2}}^{1/3} + \normh{\nablah^2 v}{\Lnorm{2}} \bigr) \,dz \\
	& ~~~~ \times  \int_0^1 \biggl( ( \normh{\dz v_h}{\Lnorm2}^{2/3}\normh{\dz v_{hh}}{\Lnorm2}^{1/3} + \normh{\dz v_h}{\Lnorm2} )( \normh{v_{hh}}{\Lnorm2}^{2/3} \normh{\nablah v_{hh}}{\Lnorm2}^{1/3} \\
	& ~~~~ + \normh{v_{hh}}{\Lnorm2} ) + \normh{\dz v}{H^1} \normh{v_{hhh}}{\Lnorm2} \biggr) \,dz  \lesssim \varepsilon \norm{v}{H^2} \norm{\xi}{H^2}\\
	& ~~~~ \times \bigl(\norm{\nabla v_{hh}}{\Lnorm2}^2 + \norm{v}{H^2}^2 \bigr) + \norm{v}{H^2}^{2}\norm{\nabla v_{hh}}{\Lnorm2} + \norm{v}{H^2}^3 \\
	& ~~~~ + \norm{v}{H^2}^{5/3} \norm{\nabla v_{hh}}{\Lnorm2}^{4/3} \lesssim \bigl( \varepsilon \norm{v}{H^2} \norm{\xi}{H^2} + \delta \bigr) \norm{\nabla v_{hh}}{\Lnorm2}^2 \\
	& ~~~~ + \bigl(\varepsilon \norm{v}{H^2}\norm{\xi}{H^2} + \norm{v}{H^2} + \norm{v}{H^2}^3 \bigr) \norm{v}{H^2}^2,
\end{align*}
where we have applied the Minkowski, H\"older, Sobolev embedding, and Young inequalities.
%where we have applied the Brezis-Gallouet inequality as in the following
%%\begin{equation}
%%	\normh{u}{\infty}% \lesssim \normh{u}{H^1} \bigl( 1 + \log \dfrac{\normh{\deltah u}{2}}{\normh{u}{H^1}})^{1/2}
%%	\lesssim (\normh{u}{H^1} + 1)\log^{1/2}(e+ \normh{u}{H^2})
%%\end{equation}
%\begin{align*}
%	& \normh{\dz v}{\infty}% \lesssim \normh{u}{H^1} \bigl( 1 + \log \dfrac{\normh{\deltah u}{2}}{\normh{u}{H^1}})^{1/2}
%	\lesssim (\normh{\dz v}{H^1} + 1)\log^{1/2}(e+ \normh{\dz v}{H^2}) \\
%	& ~~~~ ~~~~ \lesssim
%\end{align*}
Hence we have
\begin{equation}
%	\boxed{
	\begin{aligned}
		& I_1 \lesssim \bigl( \varepsilon \norm{v}{H^2} \norm{\xi}{H^2} + \delta \bigr) \norm{\nabla v_{hh}}{2}^2 + \bigl(\varepsilon \norm{v}{H^2}\norm{\xi}{H^2} \\
		& ~~~~ + \varepsilon^4 \norm{v}{H^1}\norm{\xi}{H^2}^4
		 + \varepsilon^4 \norm{v}{H^1}^4\norm{\xi}{H^2}^4  + \norm{v}{H^2}^4\\
		& ~~~~ + 1 \bigr) \norm{v}{H^2}^2.
	\end{aligned}
%	}
\end{equation}
Next, to estimate $ I_3, I_4, I_5 $, applying the H\"older, Sobolev embedding, and Young inequalities yields,
\begin{align*}
	& I_3 = - 2 \varepsilon c_1 \alpha \int e^{-\varepsilon \alpha \xi} \xi_h \biggl( ( \mu \deltah v_h + \lambda \nablah \dvh v_h) \cdot v_{hh} - \dz v_h \cdot v_{hhz} \biggr) \idx \\
	& ~~ \lesssim \varepsilon \int_0^1 e^{\varepsilon \alpha \norm{\xi}{\Lnorm\infty}} \normh{\xi_h}{\Lnorm4}  ( \normh{\nablah^2 v_{h}}{\Lnorm2}\normh{v_{hh}}{\Lnorm4} + \normh{\dz v_h}{\Lnorm4} \normh{v_{hhz}}{\Lnorm2} )  \,dz \\
	& ~~ \lesssim \varepsilon \norm{\xi}{H^2}e^{\varepsilon \alpha \norm{\xi}{\Lnorm\infty}} \int_0^1 \normh{\nabla \nablah^2 v}{\Lnorm2} ( \normh{\nabla v_h}{\Lnorm2}^{1/2}  \normh{\nabla v_{hh}}{\Lnorm2}^{1/2}  +  \normh{\nabla v_h}{\Lnorm2}) \,dz \\
	& ~~ \lesssim \delta \norm{\nabla \nablah^2 v}{\Lnorm2}^2 + C_\delta \bigl( \varepsilon^4 \norm{\xi}{H^2}^4 + \varepsilon^2 \norm{\xi}{H^2}^2  \bigr) e^{4 \varepsilon \alpha \norm{\xi}{H^2}} \norm{v}{H^2}^2  ,\\
	& I_4 = - \varepsilon c_1 \alpha \int \biggl( e^{-\varepsilon \alpha\xi} \bigl( \xi_{hh} - \varepsilon \alpha\xi_h^2 \bigr) \cdot \bigl( ( \mu \deltah v + \lambda \nablah \dvh  v ) \cdot v_{hh} \\
	& ~~~~ - \dz v \cdot v_{hhz} \bigr) \biggr) \idx
	 \lesssim \varepsilon \int_0^1 \biggl( e^{\varepsilon \alpha \norm{\xi}{\Lnorm\infty}}( \normh{\xi_{hh}}{\Lnorm2} + \varepsilon \normh{\xi_{h}}{\Lnorm4}^2 )( \normh{\nablah^2 v}{\Lnorm4}^2 \\
	& ~~~~ + \normh{\dz v}{\Lnorm\infty} \normh{v_{hhz}}{\Lnorm2} ) \biggr) \,dz
	 \lesssim \varepsilon  e^{\varepsilon \alpha \norm{\xi}{\Lnorm\infty}}( \norm{\xi}{H^2} + \varepsilon \norm{\xi}{H^2}^2 ) \\
	& ~~~~ \times \int_0^1 \biggl( \normh{\nablah^2 v}{\Lnorm2} \normh{\nablah^3 v}{\Lnorm2} + \normh{\nablah^2 v}{\Lnorm2}^2
	+ \normh{ \nablah^2 \dz v}{\Lnorm2}^2 + \normh{\dz v}{H^1}^2 \biggr) \,dz \\
	& ~~ \lesssim \varepsilon  \bigl( \norm{\xi}{H^2} + \varepsilon \norm{\xi}{H^2}^2 \bigr) e^{\varepsilon \alpha \norm{\xi}{H^2}}
	\bigl( \norm{\nabla \nablah^2 v}{\Lnorm2}^2 + \norm{v}{H^2}^2\bigr), \\
	& I_5 \lesssim \varepsilon \int_0^1 e^{\varepsilon \alpha \norm{\xi}{\Lnorm\infty}}\normh{\nablah \xi}{\Lnorm4} \normh{\nablah v_{hh}}{\Lnorm2} \normh{v_{hh}}{\Lnorm4}\,dz \lesssim \varepsilon e^{\varepsilon \alpha \norm{\xi}{\Lnorm\infty}}\norm{\xi}{H^2} \\
	& ~~~~ \times \int_0^1 \normh{\nablah v_{hh}}{\Lnorm2} \bigl( \normh{v_{hh}}{\Lnorm2}^{1/2} \normh{\nablah v_{hh}}{\Lnorm2}^{1/2}
	 + \normh{v_{hh}}{\Lnorm2} \bigr) \,dz \lesssim \delta \norm{\nabla v_{hh}}{\Lnorm2}^2\\
	& ~~~~ + C_\delta \bigl( \varepsilon^4 \norm{\xi}{H^2}^4 + \varepsilon^2 \norm{\xi}{H^2}^2 \bigr)e^{4\varepsilon \alpha \norm{\xi}{H^2}}\norm{v}{H^2}^2.
\end{align*}
In order to estimate $ I_6,  I_7 $, after substituting \subeqref{eq:perturbation}{1} and applying integration by parts, it holds
\begin{align*}
	& I_6 = \dfrac{c^2 \alpha}{2} \int e^{\varepsilon\xi} \abs{\xi_{hh}}{2} ( - \varepsilon v \cdot \nablah \xi - (\gamma-1) \dvh v  ) \idx , \\
	& I_7 %  = - c^2 \alpha \int e^{\varepsilon \xi} \xi_{hh} ( v\cdot \nablah \xi)_{hh} \idx
	= - c^2 \alpha \int e^{\varepsilon \xi} \xi_{hh} ( v_{hh} \cdot\nablah \xi + 2 v_h \cdot \nablah \xi_h ) \idx \\
	& ~~~~ + \dfrac{c^2 \alpha}{2} \int e^{\varepsilon\xi} \abs{\xi_{hh}}{2}(  \dvh v +{ % \color{blue}
	\varepsilon v \cdot \nablah \xi } ) \idx.
\end{align*}
Therefore, one has
%\begin{equation}
%\boxed{
\begin{align*}
	& I_6 + I_7 + I_8 = - c^2 \alpha \int e^{\varepsilon \xi} \xi_{hh} ( v_{hh} \cdot\nablah \xi + 2 v_h \cdot \nablah \xi_h ) \idx \\
	& ~~~~ + \dfrac{(2-\gamma)c^2 \alpha}{2} \int e^{\varepsilon\xi} \abs{\xi_{hh}}{2}  \dvh v \idx
	 + c^2 \int e^{\varepsilon \xi}\xi_{hh}(v_{hh} \cdot\nablah \xi) \idx\\
	& ~~ \lesssim e^{\varepsilon\norm{\xi}{\Lnorm\infty}}  \int_0^1 \normh{\xi_{hh}}{\Lnorm2} ( \normh{v_{hh}}{\Lnorm4} \normh{\nablah \xi}{\Lnorm4} + \normh{\nablah v}{\Lnorm\infty} \normh{\nablah \xi_{h}}{\Lnorm2}  ) \idx \\
	& ~~ \lesssim e^{\varepsilon\norm{\xi}{\Lnorm\infty}}  \int_0^1 \normh{\xi}{H^2}^2 \bigl( \normh{\nablah^3 v}{\Lnorm2} + \normh{v}{\Hnorm{2}}\bigr)  \idx  \lesssim \delta \norm{\nablah^3 v}{\Lnorm2}^2 \\
	& ~~~~ + C_\delta e^{2\varepsilon \norm{\xi}{H^2}} \norm{\xi}{H^2}^2 \bigl( \norm{\xi}{H^2}^2 + \norm{v}{H^2} \bigr),
\end{align*}
where we have applied the H\"older, Sobolev embedding and Young inequalities.
%}
%\end{equation}
Similarly, $ I_2 $ can be estimated as following:
\begin{align*}
	& I_2 \lesssim e^{\varepsilon\norm{\xi}{\Lnorm\infty}} \int_0^1 \bigl( \normh{\nablah \xi_{h}}{\Lnorm2} + \varepsilon \normh{\xi_h}{\Lnorm4}^2 \bigr) \normh{\nablah \xi}{\Lnorm4} \normh{v_{hh}}{\Lnorm4} \,dz \\
	& ~~ \lesssim e^{\varepsilon \norm{\xi}{\Lnorm\infty}} \int_0^1 ( \normh{\xi}{H^2}^2 + \normh{\xi}{H^2}^3)  ( \normh{\nablah v_{hh}}{\Lnorm2}^{1/2} \normh{v_{hh}}{\Lnorm2}^{1/2} + \normh{v_{hh}}{\Lnorm2} ) \,dz\\
	& ~~ \lesssim \delta \norm{\nabla v_{hh}}{\Lnorm2}^2 + C_\delta \bigl( \norm{\xi}{H^2}^{8/3} + \norm{\xi}{H^2}^4 \bigr) e^{4\varepsilon/3 \norm{\xi}{H^2}} \\
	& ~~~~ ~~~~ \times  \norm{v}{H^2} + \bigl( \norm{\xi}{H^2}^2 + \norm{\xi}{H^2}^3 \bigr) e^{\varepsilon \norm{\xi}{H^2}} \norm{v}{H^2}.
\end{align*}
Therefore, \eqref{ue:001} implies, after summing up the estimates above with $ \partial_h \in \lbrace \partial_x, \partial_y \rbrace $,
\begin{equation}\label{ue:002}
%\boxed{
	\begin{aligned}
		& \dfrac{d}{dt} \biggl\lbrace \dfrac 1 2 \norm{\nablah^2 v}{\Lnorm2}^2 + \dfrac{c^2}{2(\gamma-1)} \norm{e^{\varepsilon\xi /2}\nablah^2 \xi}{\Lnorm2}^2 \biggr\rbrace \\
		& ~~~~ + c_1\bigl( \mu \norm{e^{-\varepsilon\alpha\xi/2} \nablah^3 v}{\Lnorm2}^2
		 + \lambda \norm{e^{-\varepsilon\alpha\xi/2} \nablah^2 \dvh v}{\Lnorm2}^2 \\
		& ~~~~ + \norm{e^{-\varepsilon\alpha\xi/2}\nablah^2 v_{z}}{\Lnorm2}^2 \bigr) \leq \delta \norm{\nabla \nablah^2 v}{\Lnorm2}^2 \\
		& ~~~~ + \varepsilon \mathcal H_1 ( \norm{\xi}{H^2}, \norm{v}{H^2} ) \norm{\nabla \nablah^2 v}{\Lnorm2}^2 + C_\delta \mathcal H_2 ( \norm{\xi}{H^2}, \norm{v}{H^2} ) ,
	\end{aligned}
%}
\end{equation}
where
\begin{equation}\label{ue:003}
	\mathcal H_1 = \mathcal H_1 ( \norm{\xi}{H^2}, \norm{v}{H^2} ), ~ \mathcal H_2 = \mathcal H_2 ( \norm{\xi}{H^2}, \norm{v}{H^2} ),
\end{equation}
are two regular functions of the arguments and $ \mathcal H_1 (0) = \mathcal H_2 (0) = 0 $. We will adopt the same notations for functions with such properties in this work.

After taking the $ L^2 $-inner  product of \subeqref{eq:perturbation}{2}, and the horizontal derivative of \subeqref{eq:perturbation}{2} with $ v $, $ v_h $, respectively,
repeating similar arguments as above yields the following estimate:
\begin{equation}\label{ue:004}
%	\boxed{
	\begin{aligned}
		& \dfrac{d}{dt} \biggl\lbrace \dfrac{1}{2}  \norm{v}{\Lnorm2}^2 + \dfrac{1}{2} \norm{\nablah v}{\Lnorm2}^2 + \dfrac{c^2}{2(\gamma-1)}  \norm{e^{\varepsilon\xi/2} \xi}{\Lnorm2}^2\\
		& ~~~~ + \dfrac{c^2}{2(\gamma-1)} \norm{e^{\varepsilon\xi/2} \nablah \xi}{\Lnorm2}^2  \biggr\rbrace  + c_1 \bigl( \mu \norm{e^{-\varepsilon\alpha\xi/2}\nablah v}{\Lnorm2}^2 \\
		& ~~~~ + \mu \norm{e^{-\varepsilon\alpha\xi/2}\nablah^2 v}{\Lnorm2}^2 + \lambda \norm{e^{-\varepsilon\alpha\xi/2} \dvh v}{\Lnorm2}^2 \\
		& ~~~~ + \lambda \norm{e^{-\varepsilon\alpha\xi/2} \nablah \dvh v}{\Lnorm2}^2 + \norm{e^{-\varepsilon\alpha\xi/2} v_z}{\Lnorm{2}}^2 \\
		& ~~~~ + \norm{e^{-\varepsilon\alpha\xi/2} \nablah v_{z}}{\Lnorm2}^2 \bigr) \leq \delta \norm{\nabla v}{H^2}^2  + C_\delta \mathcal H_2 (\norm{\xi}{\Hnorm2},\norm{v}{\Hnorm{2}}).
	\end{aligned}
%	}
\end{equation}

\subsubsection*{Estimates on the vertical derivatives}
After applying $ \dz  $ to \subeqref{eq:perturbation}{2}, it holds
\begin{equation}\label{eq:dz-derivative}
	\begin{aligned}
		& \dt v_z - c_1 e^{-\varepsilon \alpha \xi} (\mu \deltah v_z + \lambda \nablah\dvh v_z + \partial_{zz} v_z ) = - v_z \cdot\nablah v \\
		& ~~~~ ~~~~ - v\cdot\nablah v_z - w_z \dz v - w \dz v_{z}.
	\end{aligned}
\end{equation}
Again, applying $ \dz $ to \eqref{eq:dz-derivative} yields,
\begin{equation}\label{eq:dzz-derivative}
	\begin{aligned}
		& \dt v_{zz} - c_1 e^{-\varepsilon \alpha \xi} ( \mu \deltah v_{zz} + \lambda \nablah \dvh v_{zz} + \partial_{zz} v_{zz} )  = - v \cdot \nablah v_{zz} \\
		& ~~~~ - 2 v_z \nablah v_{z}
		 - v_{zz} \cdot \nablah v - w_{zz} \dz v - 2 w_z \dz v_z - w\dz v_{zz}.
	\end{aligned}
\end{equation}
Next, we take the $ L^2 $-inner produce of \eqref{eq:dzz-derivative} with $ v_{zz} $. After applying integration by parts, one has
\begin{equation}\label{ue:101}
	\begin{aligned}
		 & \dfrac{d}{dt} \biggl\lbrace \dfrac{1}{2} \norm{v_{zz}}{\Lnorm2}^2 \biggr\rbrace + c_1 \bigl(\mu \norm{e^{-\varepsilon\alpha\xi/2}\nablah v_{zz}}{\Lnorm2}^2 + \lambda \norm{e^{-\varepsilon\alpha\xi/2}\dvh v_{zz}}{\Lnorm2}^2 \\
		 & ~~~~ + \norm{e^{-\varepsilon\alpha\xi/2}v_{zzz}}{\Lnorm2}^2  \bigr) = I_9 + I_{10} + I_{11}, % + I_{12},
	\end{aligned}
\end{equation}
where
\begin{align*}
	& I_9 : = - \int ( v\cdot\nablah v_{zz} + 2 v_z \nablah v_z + v_{zz} \cdot\nablah v ) \cdot v_{zz} \idx,\\
	& I_{10} : = - \int (w_{zz}\dz v + \dfrac{3}{2} w_z \dz v_z ) \cdot v_{zz} \idx, \\
%	& I_{11} = - \int w \dz v_{zz} \cdot v_{zz} \idx = \dfrac{1}{2} \int \abs{v_{zz}}{2} w_z \idx ,\\
	& I_{11}  : = c_1 \varepsilon \alpha \int e^{-\varepsilon\alpha\xi} ( \mu ( \nablah \xi \cdot \nablah v_{zz} ) \cdot v_{zz} + \lambda (v_{zz} \cdot \nablah \xi ) \dvh v_{zz} ) \idx.
\end{align*}
%Notice, from \eqref{def:vertical-velocity}, we have
%\begin{equation}\label{def:dzz-vertical-velocity}
%	\begin{aligned}
%		w_z = & - \varepsilon \alpha \tilde v \cdot \nablah \xi - \dvh \tilde v, \\
%		w_{zz} = & - \varepsilon \alpha \widetilde v_z \cdot \nablah \xi - \dvh \widetilde v_z.
%	\end{aligned}
%\end{equation}
After applying the
H\"older, Sobolev embedding and Young inequalities, it holds,
\begin{align*}
	& I_9 \lesssim \bigl( \norm{v}{\Lnorm6} \norm{\nablah v_{zz}}{\Lnorm2} + \norm{v_z}{\Lnorm6}\norm{\nablah v_z}{\Lnorm2} + \norm{v_{zz}}{\Lnorm2} \norm{\nablah v}{\Lnorm6} \bigr) \\
	& ~~~~ \times  \norm{v_{zz}}{\Lnorm3}
	\lesssim \bigl( \norm{v}{H^1} \norm{\nablah v_{zz}}{\Lnorm2} + \norm{v}{H^2}^2 \bigr) \bigl( \norm{v_{zz}}{\Lnorm2}^{1/2}\norm{\nabla v_{zz}}{\Lnorm2}^{1/2}\\
	& ~~~~ + \norm{v_{zz}}{\Lnorm2} \bigr)
	\lesssim \delta \norm{\nabla v_{zz}}{\Lnorm2}^2 + C_\delta \bigl( \norm{v}{H^2}^4 + \norm{v}{H^2}\bigr)  \norm{v}{H^2}^2, \\
	& I_{10} =  \varepsilon \alpha \int \biggl( (\widetilde v_z \cdot \nablah \xi )( \dz v \cdot v_{zz}) + \dfrac{3}{2} ( \widetilde v \cdot \nablah \xi)( \dz v_z \cdot v_{zz}) \biggr) \idx \\
	& ~~~~ + \int \biggl( \dvh \widetilde v_z  ( \dz v \cdot v_{zz} ) + \dfrac{3}{2} \dvh \widetilde v ( \dz v_z \cdot v_{zz}) \biggr) \idx \\
	& ~~ \lesssim \varepsilon \alpha  \bigl( \norm{v_z}{\Lnorm4} \norm{\nablah \xi}{\Lnorm4}\norm{v_z}{\Lnorm3} + \norm{v}{\Lnorm6} \norm{\nablah \xi}{\Lnorm6} \norm{v_{zz}}{\Lnorm2} \bigr) \norm{v_{zz}}{\Lnorm6}\\
	& ~~~~ + \bigl( \norm{\nablah v_{z}}{\Lnorm2} \norm{v_z}{\Lnorm3} + \norm{\nablah v}{\Lnorm3} \norm{v_{zz}}{\Lnorm2}\bigr) \norm{v_{zz}}{\Lnorm6}\\
	& ~~ \lesssim \bigl( \varepsilon\alpha  \norm{\xi}{H^2} + 1 \bigr) \norm{v}{H^2}^2  \bigl( \norm{\nabla v_{zz}}{\Lnorm2} + \norm{v}{H^2} \bigr) \\
	& ~~ \lesssim \delta \norm{\nabla v_{zz}}{\Lnorm2}^2 + C_\delta \bigl(\varepsilon^2\alpha^2 \norm{\xi}{H^2}^2 +1 \bigr)\bigl(\norm{v}{H^2}^2 + \norm{v}{H^2}\bigr)\norm{v}{H^2}^2,\\
	& I_{11} \lesssim \varepsilon \alpha e^{\varepsilon \alpha \norm{\xi}{\Lnorm\infty}} \norm{\nablah \xi}{\Lnorm6} \norm{v_{zz}}{\Lnorm3} \norm{\nablah v_{zz}}{\Lnorm2} \lesssim \varepsilon \alpha e^{\varepsilon \alpha \norm{\xi}{H^2}} \norm{\xi}{H^2} \\
	& ~~~~ \times \bigl( \norm{v_{zz}}{\Lnorm2}^{1/2} \norm{\nabla v_{zz}}{\Lnorm2}^{1/2}  + \norm{v_{zz}}{\Lnorm2} \bigr) \norm{\nablah v_{zz}}{\Lnorm2} \lesssim \delta \norm{\nabla v_{zz}}{\Lnorm2}^2 \\
	& ~~~~ + C_\delta \bigl(\varepsilon^4  \norm{\xi}{H^2}^4 + \varepsilon^2 \norm{\xi}{H^2}^2  \bigr) e^{4 \varepsilon \alpha \norm{\xi}{H^2}}\norm{v}{H^2}^2,
\end{align*}
where we have substituted \eqref{vertical-velocity-dz} and \eqref{vertical-velocity-dzz} into $ I_{10} $.
Therefore, we have
\begin{equation}\label{ue:102}
%\boxed{
	\begin{aligned}
		& \dfrac{d}{dt} \biggl\lbrace \dfrac{1}{2} \norm{v_{zz}}{\Lnorm2}^2 \biggr\rbrace + c_1 \bigl(\mu \norm{e^{-\varepsilon\alpha\xi/2}\nablah v_{zz}}{\Lnorm2}^2 + \lambda \norm{e^{-\varepsilon\alpha\xi/2}\dvh v_{zz}}{\Lnorm2}^2 \\
		 & ~~~~ + \norm{e^{-\varepsilon\alpha\xi/2}v_{zzz}}{\Lnorm2}^2  \bigr) \leq \delta \norm{\nabla v_{zz}}{\Lnorm2}^2 +  C_\delta \mathcal H_2( \norm{\xi}{\Hnorm{2}}, \norm{v}{\Hnorm{2}} ).
	\end{aligned}
%}
\end{equation}

Next, we establish the estimate of $ v_{hz} $. Apply $ \partial_h $ to \eqref{eq:dz-derivative}. It follows
\begin{equation}\label{eq:dhz-derivative}
	\begin{aligned}
		& \dz v_{hz} - c_1 e^{-\varepsilon \alpha \xi} ( \mu \deltah v_{hz} + \lambda \nablah \dvh v_{hz} + \partial_{zz} v_{hz} ) = - v_{hz} \cdot \nablah v \\
		& ~~~~ - v_z \cdot \nablah v_h
		 - v_h \cdot \nablah v_z - v\cdot\nablah v_{hz} - w_{hz} \dz v - w_z \dz v_h \\
		& ~~~~ - w \dz v_{hz}  - w_h \dz v_{z} - c_1 \varepsilon \alpha e^{-\varepsilon\alpha \xi} \xi_h ( \mu \deltah v_z + \lambda \nablah \dvh  v_z\\
		& ~~~~ + \partial_{zz} v_z ).
	\end{aligned}
\end{equation}
Take the $ L^2 $-inner product of \eqref{eq:dhz-derivative} with $ v_{hz} $ and apply integration by parts in the resultant equation. It follows,
\begin{equation}\label{ue:103}
	\begin{aligned}
		& \dfrac{d}{dt} \biggl\lbrace \dfrac{1}{2} \norm{v_{hz}}{\Lnorm2}^2 \biggr\rbrace + c_1 ( \mu \norm{e^{-\varepsilon \alpha \xi/2} \nablah v_{hz}}{\Lnorm2}^2 + \lambda \norm{e^{-\varepsilon\alpha\xi/2} \dvh v_{hz}}{\Lnorm2}^2 \\
		& ~~~~ + \norm{e^{-\varepsilon \alpha\xi/2} v_{hzz}}{\Lnorm2}^2) = I_{12} + I_{13} + I_{14} + I_{15} + I_{16},
	\end{aligned}
\end{equation}
where
\begin{equation*}
	\begin{aligned}
		& I_{12} : = - \int ( v_{hz} \cdot \nablah v + v_z \cdot\nablah v_h + v_h \cdot\nablah v_z - \dfrac{1}{2} ( \dvh v ) v_{hz} ) \cdot v_{hz} \idx, \\
		& I_{13} : = - \int ( w_{hz} \dz v + w_z \dz v_{h} - \dfrac{1}{2} ( w_z ) v_{hz} ) \cdot v_{hz} \idx, \\
		& I_{14} : = - \int w_h ( \dz v_z \cdot v_{hz}) \idx, \\
		& I_{15} : = - c_1 \varepsilon \alpha \int e^{-\varepsilon \alpha \xi} \xi_h ( \mu \deltah v_z + \lambda \nablah \dvh v_z + \partial_{zz} v_z ) \cdot v_{hz} \idx,\\
		& I_{16} : = c_1 \varepsilon \alpha \int e^{-\varepsilon\alpha \xi} ( \mu ( \nablah \xi \cdot\nablah v_{hz}) \cdot v_{hz} + \lambda (v_{hz} \cdot \nablah \xi) \dvh v_{hz} )\idx.
	\end{aligned}
\end{equation*}
%Similar as before, we first write down the form of $ w_{hz} $, which is, from \eqref{def:dzz-vertical-velocity},
%\begin{equation}\label{def:dhz-vertical-velocity}
%	w_{hz} = - \varepsilon \alpha \tilde v_h \cdot \nablah \xi - \varepsilon \alpha \tilde v \cdot \nablah \xi_h - \dvh \tilde v_h.
%\end{equation}
After substituting \eqref{vertical-velocity-dhz}, \eqref{vertical-velocity-dz} in $ I_{13} $ and \eqref{vertical-velocity-dh} in $ I_{14} $, applying the H\"older,  Sobolev embedding, Minkowski and Young inequalities yields,
\begin{align*}
	& I_{12} \lesssim ( \norm{\nablah v_{z}}{\Lnorm2} \norm{\nablah v}{\Lnorm 3} + \norm{v_z}{\Lnorm 3} \norm{\nablah v_h}{\Lnorm2} ) \norm{v_{hz}}{\Lnorm6} \\
	& ~~ \lesssim \norm{v}{H^2}^2 \bigl( \norm{\nabla v_{hz}}{\Lnorm2} + \norm{v_{hz}}{\Lnorm2} \bigr) \lesssim \delta \norm{\nabla v_{hz} }{\Lnorm2}^2 \\
	& ~~~~ + C_\delta \norm{v}{H^2}^4 + \norm{v}{H^2}^3,\\
	& I_{13} =  \varepsilon \alpha \int \biggl( ( \widetilde v_h \cdot\nablah \xi + \widetilde v \cdot \nablah \xi_h) (\dz v \cdot v_{hz}) + \dfrac{1}{2} ( \widetilde v \cdot\nablah \xi )( v_{hz}  \cdot v_{hz}) \biggr) \idx  \\
	& ~~~~ + \int \biggl( \dvh \widetilde v_h (\dz v \cdot v_{hz}) + \dfrac{1}{2} \dvh \widetilde v ( v_{hz} \cdot v_{hz}) \biggr) \idx \\
	& ~~ \lesssim \varepsilon\alpha \bigl( (\norm{v_h}{\Lnorm 6} \norm{\nablah\xi}{\Lnorm 2}
	+ \norm{v}{\Lnorm6} \norm{\xi_{hh}}{\Lnorm2}) \norm{\dz v}{\Lnorm6} \\
	& ~~~~ + \norm{v}{\Lnorm6} \norm{\nablah \xi}{\Lnorm6} \norm{v_{hz}}{\Lnorm2} \bigr) \norm{v_{hz}}{\Lnorm6}
	 + \bigl( \norm{\nablah v_h}{\Lnorm2} \norm{\dz v}{\Lnorm3} \\
	& ~~~~  + \norm{\nablah v}{\Lnorm3} \norm{v_{hz}}{\Lnorm2} \bigr)\norm{v_{hz}}{\Lnorm6}
	\lesssim \bigl( \varepsilon \alpha \norm{v}{H^2}^2 \norm{\xi}{H^2} + \norm{v}{H^2}^2 \bigr) \\
	& ~~~~ \times \bigl( \norm{\nabla v_{hz}}{\Lnorm2} + \norm{v_{hz}}{\Lnorm2}\bigr)
	\lesssim \delta \norm{\nabla v_{hz}}{\Lnorm2}^2 \\
	& ~~~~ + C_\delta \bigl( \varepsilon^{2} \norm{\xi}{H^2}^2 \norm{v}{H^2}^2 + \norm{v}{H^2}^2 + 1 \bigr) \norm{v}{H^2}^2, \\
	& I_{14} =  \int \biggl\lbrack \int_0^z \biggl( \varepsilon \alpha (\widetilde v_h \cdot\nablah \xi + \widetilde v \cdot \nablah \xi_h) + \dvh \widetilde v_h \biggr) \,dz \times ( \dz v_z \cdot v_{hz}) \biggr\rbrack \idx \\
	& ~~ \lesssim  \int_0^1 \biggl( \varepsilon ( \normh{v_h}{\Lnorm4} \normh{\nablah \xi}{\Lnorm4} + \normh{v}{\Lnorm\infty} \normh{\nablah \xi_{h}}{\Lnorm2} ) + \normh{\nablah v_h}{\Lnorm2} \biggr) \,dz \\
	& ~~~~ \times \int_0^1 \normh{\dz v_z}{\Lnorm4} \normh{v_{hz}}{\Lnorm4} \,dz
	\lesssim \int_0^1 \bigl( \varepsilon \normh{v}{H^2} \normh{\xi}{H^2} + \normh{v}{H^2} \bigr) \,dz \\
	& ~~~~  \times \int_0^1 \biggl( \bigl(\normh{\dz v_z}{\Lnorm2}^{1/2} \normh{\nablah \dz v_z}{\Lnorm2}^{1/2} + \normh{\dz v_z}{\Lnorm2}\bigr) \bigl( \normh{v_{hz}}{\Lnorm2}^{1/2} \normh{\nablah v_{hz}}{\Lnorm2}^{1/2} \\
	& ~~~~ + \normh{v_{hz}}{\Lnorm2} \bigr) \biggr) \,dz \lesssim \delta \norm{\nabla \nablah v_{z}}{\Lnorm2}^2  + C_\delta \bigl( \varepsilon^2 \norm{\xi}{H^2}^2 \norm{v}{H^2}^2 \\
	& ~~~~ + \norm{v}{H^2}^2 + 1 \bigr) \norm{v}{H^2}^2,\\
	& I_{15} \lesssim \varepsilon e^{\varepsilon \alpha \norm{\xi}{\Lnorm\infty}}  \norm{\nabla^3 v}{\Lnorm2} \norm{v_{hz}}{\Lnorm{3}} \norm{\xi_h}{\Lnorm{6}} \lesssim \delta \norm{\nabla^3 v}{\Lnorm2}^2 \\
	& ~~~~ + C_\delta \varepsilon^4 e^{4 \varepsilon \alpha \norm{\xi}{H^2}} \bigl( \norm{\xi}{\Hnorm{2}}^4 + \norm{\xi}{\Hnorm{2}}^2 \bigr) \norm{v}{H^2}^2, \\
	& I_{16} \lesssim \varepsilon e^{\varepsilon\alpha\norm{\xi}{\Lnorm\infty}} \norm{\nablah \xi}{\Lnorm 6} \norm{\nabla v_{hz}}{\Lnorm2} \norm{v_{hz}}{\Lnorm3} \lesssim \varepsilon e^{\varepsilon \alpha \norm{\xi}{H^2}} \norm{\xi}{H^2} \\
	& ~~~~ \norm{\nabla v_{hz}}{\Lnorm2}\bigl( \norm{v_{hz}}{\Lnorm2}^{1/2} \norm{\nabla v_{hz}}{\Lnorm2}^{1/2} + \norm{v_{hz}}{\Lnorm2} \bigr) \lesssim \delta \norm{\nabla v_{hz}}{\Lnorm2}^2\\
	& ~~~~ + C_\delta \bigl( \varepsilon^4 e^{4\varepsilon \alpha \norm{\xi}{H^2}} \norm{\xi}{H^2}^4 + 1 \bigr) \norm{v}{H^2}^2.
\end{align*}
Therefore, summing up the estimates above with $ \partial_h \in \lbrace \partial_x, \partial_y \rbrace $ leads to,
\begin{equation}\label{ue:104}
%	\boxed{
	\begin{aligned}
		& \dfrac{d}{dt} \biggl\lbrace \dfrac{1}{2} \norm{\nablah v_{z}}{\Lnorm 2}^2 \biggr\rbrace + c_1 \bigl( \mu \norm{e^{-\varepsilon \alpha \xi/2} \nablah^2 v_{z}}{\Lnorm 2}^2 \\
		& ~~~~ + \lambda \norm{e^{-\varepsilon\alpha\xi/2} \nablah \dvh v_{z}}{\Lnorm 2}^2
		 + \norm{e^{-\varepsilon \alpha\xi/2} \nablah v_{zz}}{\Lnorm 2}^2\bigr)\\
		& ~~~~  \leq \delta \norm{\nabla^3 v}{\Lnorm 2}^2 + C_\delta \mathcal H_2(\norm{\xi}{\Hnorm{2}}, \norm{v}{\Hnorm{2}} ).
	\end{aligned}
%	}
\end{equation}
Similarly, after taking the $ L^2 $-inner product of \eqref{eq:dz-derivative} with $ v_z $ and performing estimates as above, one has,
\begin{equation}\label{ue:105}
%\boxed{
\begin{aligned}
	& \dfrac{d}{dt} \biggl\lbrace \dfrac{1}{2} \norm{v_z}{\Lnorm2}^2 \biggr\rbrace + c_1 \bigl(  \mu \norm{e^{-\varepsilon \alpha \xi/2} \nablah v_{z}}{\Lnorm2}^2 + \lambda \norm{e^{-\varepsilon\alpha\xi/2} \dvh v_{z}}{\Lnorm2}^2 \\
		& ~~~~ + \norm{e^{-\varepsilon \alpha\xi/2} v_{zz}}{\Lnorm2}^2\bigr) \leq  \delta \norm{\nabla^3 v}{\Lnorm{2}}^2 + C_\delta \mathcal H_2(\norm{\xi}{\Hnorm{2}}, \norm{v}{\Hnorm{2}} ).
\end{aligned}
%}	
\end{equation}

Consequently, after integrating in the time variable, \eqref{ue:002}, \eqref{ue:004}, \eqref{ue:102}, \eqref{ue:104}, and \eqref{ue:105} yield \eqref{ue:spatial-derivative}.

\subsubsection*{Estimates on the time derivatives}
Now we establish the final pieces of Proposition \ref{prop:apriori-estimate}.
After multiplying \subeqref{eq:perturbation}{1} with $ \varepsilon e^{\varepsilon \xi} $ and averaging the resulting equation in the $ z $-variable, one has
\begin{equation}\label{eq:density}
	\dt e^{\varepsilon \xi} + \varepsilon e^{\varepsilon \xi}\overline v\cdot \nablah  \xi + (\gamma-1) e^{\varepsilon \xi} \dvh \overline v = 0.
\end{equation}
Then applying the triangle inequality and the Sobolev embedding inequality in \eqref{eq:density} leads to
\begin{equation}\label{ue:201}
\begin{aligned}
	& \norm{\dt e^{\varepsilon \xi}}{\Lnorm 2} \lesssim \varepsilon e^{\varepsilon \norm{\xi}{\Lnorm\infty}}\norm{v}{\Lnorm\infty} \norm{\nablah \xi}{\Lnorm2} + e^{\varepsilon \norm{\xi}{\Lnorm\infty}} \norm{\nablah v}{\Lnorm2} \\
	& ~~ \lesssim \varepsilon e^{\varepsilon \norm{\xi}{H^2}} \norm{v}{H^2} \norm{\xi}{H^2} + e^{\varepsilon\norm{\xi}{H^2}} \norm{v}{H^2}\\
	& ~~ \lesssim \mathcal H_2(\norm{\xi}{\Hnorm{2}},\norm{v}{\Hnorm{2}}) .
\end{aligned}
\end{equation}
Meanwhile, we have, after applying $ \partial_h $ to \eqref{eq:density},
\begin{equation}\label{eq:dh-density}
	\begin{aligned}
	& \dt (e^{\varepsilon\xi})_h + \varepsilon^2 e^{\varepsilon \xi} \xi_h \overline v \cdot \nablah \xi + \varepsilon e^{\varepsilon\xi} \overline v_h \cdot \nablah \xi + \varepsilon e^{\varepsilon\xi} \overline v\cdot\nablah \xi_h \\
	& ~~~~ ~~~~ + (\gamma-1) \varepsilon e^{\varepsilon \xi} \xi_h \dvh \overline v + (\gamma - 1)e^{\varepsilon \xi} \dvh \overline v_h = 0.
	\end{aligned}
\end{equation}
Thus it holds,
\begin{equation}\label{ue:202}
	\begin{aligned}
	& \norm{\dt (e^{\varepsilon \xi})_h}{\Lnorm2} \lesssim e^{\varepsilon \norm{\xi}{\Lnorm \infty}} \bigl( \varepsilon^2 \norm{\nablah \xi}{\Lnorm4}^2 \norm{v}{\Lnorm \infty} +\varepsilon \norm{\nablah\xi}{\Lnorm 4} \\
	& ~~~~ \times \norm{ \nablah v}{\Lnorm 4}
	 + \varepsilon \norm{v}{\Lnorm \infty} \norm{\nablah\xi_h}{\Lnorm2}   + \norm{\nablah v_{h}}{\Lnorm 2}  \bigr) \\
	& ~~ \lesssim e^{\varepsilon \norm{\xi}{H^2}} \bigl( \varepsilon^2 \norm{\xi}{H^2}^2 \norm{v}{H^2} + \norm{v}{H^2} \bigr) \lesssim \mathcal H_2(\norm{\xi}{\Hnorm{2}},\norm{v}{\Hnorm{2}}).
	\end{aligned}
\end{equation}
Consequently, \eqref{ue:201} and \eqref{ue:202} imply \eqref{ue:2001}.

On the other hand, after applying the projection operator $ \mathcal P_\sigma $ (defined in \eqref{def:projection}) to \subeqref{eq:perturbation}{2}, we have the following:
\begin{equation}\label{eq:dt-velocity}
\begin{aligned}
	& \dt \mathcal P_\sigma v = - \mathcal P_\sigma (v\cdot\nablah v) - \mathcal P_\sigma (w\dz v) \\
	& ~~~~ + c_1 \mathcal P_\sigma (e^{-\varepsilon\alpha \xi} (\mu \deltah v + \lambda \nablah \dvh v + \partial_{zz} v)).
\end{aligned}
\end{equation}
In order to estimate the $ L^2 $ and $ H^1 $ norms of $ \dt \mathcal P_\sigma v $, we apply the H\"older and Sobolev embedding inequalities as follows:
\begin{equation}\label{ue:206}
\begin{aligned}
	& \norm{v\cdot \nablah v}{\Lnorm2} \lesssim \norm{v}{\Lnorm3} \norm{\nablah v}{\Lnorm6} \lesssim \norm{v}{H^2}^2, \\
	& \norm{e^{-\varepsilon\alpha\xi}(\mu \deltah v + \lambda \nablah \dvh v + \partial_{zz} v)}{\Lnorm2} \lesssim e^{\varepsilon \alpha \norm{\xi}{\Lnorm\infty}}\\
	& ~~~~ ~~~~ \times \norm{\mu \deltah v + \lambda \nablah \dvh v + \partial_{zz} v}{\Lnorm2}  \lesssim e^{\varepsilon\alpha\norm{\xi}{H^2}} \norm{v}{H^2},\\
	&  \norm{\partial (v\cdot \nablah v)}{\Lnorm 2} \lesssim \norm{\partial v\cdot\nablah v}{\Lnorm2} + \norm{v \cdot \nablah \partial v}{\Lnorm 2} \lesssim \norm{v}{H^2}^2,\\
	& \norm{\partial(e^{-\varepsilon\alpha\xi}(\mu \deltah v + \lambda \nablah \dvh v + \partial_{zz} v))}{\Lnorm2} \\
	& ~~~~ \lesssim \varepsilon \norm{e^{-\varepsilon\alpha\xi}\partial \xi (\mu \deltah v + \lambda \nablah \dvh v + \partial_{zz} v)}{\Lnorm2} \\
	& ~~~~ ~~~~ + \norm{e^{-\varepsilon\alpha\xi}\partial(\mu \deltah v + \lambda \nablah \dvh v + \partial_{zz} v)}{\Lnorm2}\\
	& ~~~~ \lesssim e^{\varepsilon\alpha\norm{\xi}{H^2}} \bigl( \norm{\xi}{H^2} + 1 \bigr) \norm{v}{H^3}.
\end{aligned}
\end{equation}
Here $ \partial \in \lbrace \partial_x,\partial_y,\partial_z \rbrace $ denotes the spatial derivatives.
To estimate the $ L^2 $ and $ H^1 $ norms of $ w \dz v $, we first substitute the identities in \eqref{def:vertical-velocity}, \eqref{vertical-velocity-dz}, \eqref{vertical-velocity-dh} to $ w,  \dz w,\partial_h w $, respectively, and write down the following:
\begin{align*}
	& w \dz v = - \int_0^z \bigl( \varepsilon \alpha \widetilde v \cdot \nablah \xi + \dvh \widetilde v \bigr) \,dz \dz v ,\\
	& \partial_h (w \dz v) = \partial_h w \dz v + w \dz v_h \\
	& ~~~~ = - \int_0^z \bigl( \varepsilon \alpha \widetilde v_h \cdot \nablah \xi + \varepsilon \alpha \widetilde v \cdot \nablah \xi_h + \dvh \widetilde v_h \bigr) \,dz \dz v \\
	& ~~~~ ~~~~ - \int_0^z \bigl( \varepsilon \alpha \widetilde v \cdot \nablah \xi + \dvh \widetilde v \bigr) \,dz \dz v_h,\\
	& \partial_z (w\dz v) = \dz w \dz v + w \dz v_z \\
	& ~~~~ = - (\varepsilon \alpha \widetilde v \cdot\nablah \xi + \dvh \widetilde v) \dz v - \int_0^z \bigl( \varepsilon \alpha \widetilde v \cdot \nablah \xi + \dvh \widetilde v \bigr) \,dz \dz v_z.
\end{align*}
Therefore, after applying  the H\"older, Minkowski, and Sobolev embedding inequalities, the following estimates hold:
\begin{equation}\label{ue:205}
\begin{aligned}
	&  \norm{w \dz v}{\Lnorm2}^2 = \int_0^1 \normh{w \dz v}{\Lnorm 2}^2 \,dz \lesssim \int_0^1 \biggl\lbrack  \biggl( \int_0^1 \bigl( \varepsilon \normh{v}{\Lnorm8} \normh{\nablah \xi}{\Lnorm8} \\
	& ~~~~ + \normh{\nablah v}{\Lnorm4} \bigr) \,dz \biggr)^2
	 \times \normh{\dz v}{\Lnorm4}^2 \biggr\rbrack  \,dz \lesssim \int_0^1 \biggl( (\varepsilon^2 \norm{v}{H^1}^2 \norm{\xi}{H^2}^2\\
	& ~~~~ + \norm{v}{H^2}^2 ) \normh{\dz v}{H^1}^2  \biggr) \,dz  \lesssim \bigl(\varepsilon^2 \norm{v}{H^1}^2\norm{\xi}{H^2}^2 + \norm{v}{H^2}^2 \bigr) \norm{v}{H^2}^2 ,\\
	&  \norm{\partial_h(w\dz v)}{\Lnorm2}^2 = \int_0^1 \normh {\partial_h(w \dz v)}{\Lnorm2}^2 \,dz \lesssim \int_0^1 \biggl\lbrack \biggl( \int_0^1 \bigl( \varepsilon \normh{v_h}{\Lnorm4} \normh{\nablah \xi}{\Lnorm4} \\
	& ~~~~ + \varepsilon \normh{v}{\Lnorm \infty} \normh{\nablah \xi_{h}}{\Lnorm2} + \normh{\nablah^2 v}{\Lnorm2} \bigr) \,dz \biggr)^2 \times \normh{\dz v}{\Lnorm\infty}^2 \biggr\rbrack  \,dz \\
	&  ~~~~ + \int_0^1 \biggl\lbrack \biggl( \int_0^1 \bigl( \varepsilon \normh{v}{\Lnorm8} \normh{\nablah \xi}{\Lnorm8} + \normh{\nablah v}{\Lnorm4} \bigr) \,dz \biggr)^2 \times \normh{v_{hz}}{\Lnorm4}^2 \biggr\rbrack \,dz \\
	& ~~ \lesssim \bigl( \varepsilon^2 \norm{v}{H^2}^2 \norm{\xi}{H^2}^2 + \norm{v}{H^2}^2 \bigr) \norm{v}{H^3}^2,\\
	&  \norm{\dz(w\dz v)}{\Lnorm2}^2 \lesssim \norm{\dz w \dz v}{\Lnorm 2}^2 + \int_0^1 \normh{w\dz v_z}{\Lnorm 2}^2 \,dz \\
	& ~~  \lesssim (\varepsilon \norm{v}{\Lnorm6} \norm{\nablah \xi}{\Lnorm 6} + \norm{\nablah v}{\Lnorm 3} )^2 \norm{\dz v}{\Lnorm 6}^2 \\
	& ~~~~  + \int_0^1 \biggl\lbrack \biggl( \int_0^1 \bigl( \varepsilon \normh{v}{\Lnorm 8} \normh{\nablah \xi}{\Lnorm 8} + \normh{\nablah v}{\Lnorm4} \bigr) \,dz \biggr)^2 \times  \normh{v_{zz}}{\Lnorm4}^2 \biggr\rbrack \,dz \\
	& ~~ \lesssim  \bigl(\varepsilon^2 \norm{v}{H^1}^2 \norm{\xi}{H^2}^2 + \norm{v}{H^2}^2 \bigr) \norm{v}{H^3}^2.
\end{aligned}
\end{equation}
Combining the above estimates, together with  \eqref{def:projcetion-boundness}, we have shown \eqref{ue:203} and \eqref{ue:204}.

In addition, notice that
 $ \dt \xi = \varepsilon^{-1} e^{-\varepsilon \xi} \dt e^{\varepsilon\xi} $. From \eqref{ue:2001}, one has
\begin{align}
	& \norm{\dt \xi}{\Lnorm2} \lesssim \varepsilon^{-1} e^{\varepsilon \norm{\xi}{H^2}} \norm{\dt e^{\varepsilon\xi}}{\Lnorm2} \lesssim \varepsilon^{-1} \mathcal H_2(\norm{\xi}{\Hnorm{2}},\norm{v}{\Hnorm{2}}),{\nonumber} \\ %{\tag{\ref{ese:001}}} \\
	& \norm{\dt \xi}{H^1} \lesssim \norm{\dt \xi}{\Lnorm2} + \varepsilon^{-1} e^{\varepsilon\norm{\xi}{H^2}}  \bigl( \norm{\xi_h}{\Lnorm3} \norm{\dt e^{\varepsilon \xi}}{\Lnorm6} + \norm{\dt (e^{\varepsilon\xi})_h}{\Lnorm2} \bigr) {\nonumber} \\
	& ~~ \lesssim \norm{\dt \xi}{\Lnorm2} + \varepsilon^{-1} e^{\varepsilon\norm{\xi}{H^2}}  \bigl( \norm{\xi}{H^2} \norm{\dt e^{\varepsilon \xi}}{H^1} + \norm{\dt (e^{\varepsilon\xi})_h}{2} \bigr) {\nonumber} \\
	& ~~ \lesssim \varepsilon^{-1} \mathcal H_2(\norm{\xi}{\Hnorm{2}},\norm{v}{\Hnorm{2}}). {\tag{\ref{ese:002}}}
\end{align}
On the other hand, \subeqref{eq:perturbation}{2}, \eqref{ue:206}, and \eqref{ue:205} yield
\begin{align}
	& \norm{\dt v}{\Lnorm 2} \lesssim  \mathcal H_2(\norm{\xi}{\Hnorm{2}},\norm{v}{\Hnorm{2}}) + \varepsilon^{-1}\norm{e^{\varepsilon\xi} \nablah \xi}{\Lnorm2} {\nonumber} \\
	& ~~ \lesssim (1 + \varepsilon^{-1}) \mathcal H_2(\norm{\xi}{\Hnorm{2}},\norm{v}{\Hnorm{2}}) , {\tag{\ref{ese:003}}} \\
	& \norm{\dt v}{H^1} \lesssim \mathcal H_1(\norm{\xi}{\Hnorm{2}},\norm{v}{\Hnorm{2}}) \norm{\nabla v}{H^2} {\nonumber} \\
	& ~~~~ + \mathcal H_2(\norm{\xi}{\Hnorm{2}},\norm{v}{\Hnorm{2}}) + \varepsilon^{-1} \norm{e^{\varepsilon\xi}\nablah \xi}{H^1} {\nonumber}\\
	& ~~ \lesssim \mathcal H_1(\norm{\xi}{\Hnorm{2}},\norm{v}{\Hnorm{2}}) \norm{\nabla v}{H^2} {\nonumber}\\
	& ~~~~ + (1+\varepsilon^{-1}) \mathcal H_2(\norm{\xi}{\Hnorm{2}},\norm{v}{\Hnorm{2}}), {\tag{\ref{ese:004}}}
\end{align}
where in the last inequality we have used the fact that
\begin{align*}
	& \norm{\partial_h (e^{\varepsilon\xi} \nablah \xi)}{\Lnorm2} \lesssim \varepsilon \norm{e^{\varepsilon \xi} \xi_h \nablah \xi}{\Lnorm 2} + \norm{e^{\varepsilon\xi} \nablah \xi_h}{\Lnorm 2} \\
	& ~~ \lesssim e^{\varepsilon \norm{\xi}{H^2}}\bigl( \varepsilon \norm{\xi_h}{\Lnorm3} \norm{\xi_h}{\Lnorm 6}+ \norm{\xi}{H^2}\bigr) \\
	& ~~ \lesssim e^{\varepsilon \norm{\xi}{H^2}}\bigl( \varepsilon \norm{\xi}{H^2}^2 + \norm{\xi}{H^2} \bigr).
\end{align*}
This finishes the proof of Proposition \ref{prop:apriori-estimate}.

\subsection{Local-in-time \textit{\textbf{a priori}} estimates and local well-posedness}\label{subsec:local-well-posedness}

%In this work, we assume that the local strong solutions $ (\xi, v) $ to \eqref{eq:perturbation}, with given $ (\xi_0, v_0) \in H^2 (\Omega_h\times 2\mathbb T) $, satisfying the symmetry \eqref{SYM},
%exist in $ [0,T_\varepsilon] $ for some $ T_{\varepsilon} \in (0,\infty) $. In the case when $ \Omega_h = \mathbb T^2 $, this can be proved following similar arguments as in our previous work \cite{LT2018a}; in the case when $ \Omega_h = \mathbb R^2 $, this may be proved by first studying a Dirichlet problem within $ B_R\times 2\mathbb T $, where $ B_R $ is a ball with radius $ R \in (0,\infty) $ centered at the origin, and then using a expanding domain technique to show the existence of strong solutions in $ \mathbb R^2 \times 2\mathbb T $.
%letting $ R \rightarrow \infty $ with some uniform estimate independent of $ R $.
%In this subsection, by assuming the local existence of strong solutions, we will establish the local-in-time a priori estimates, which shows the continuity of $ H^2 $ norm of $ (\xi, v) $ for $ \varepsilon \in (0,\varepsilon_0) $ with $ \varepsilon_0 \in (0,1) $ small enough.
%Indeed, we claim %, without proof,
%the following:
In this subsection, we aim at establishing the following proposition:
\begin{prop}\label{prop:local-existence}
	In the case when either $ \Omega_h = \mathbb T^2 $ or $ \Omega_h = \mathbb R^2 $,
	consider initial data $ (\xi_0, v_0) \in H^2(\Omega_h\times 2\mathbb T) $, satisfying the compatibility conditions in \eqref{cmpbc}. Let $ M_0, M_1 $ be two positive constants satisfying
	\begin{equation}\label{bd:initial-data}
	  \norm{\xi_0}{H^2}^2 \leq M_0, \norm{v_0}{H^2}^2 \leq M_1.
	\end{equation}
	Then for some positive constant $ \varepsilon_0\in (0,1) $ small enough, any $ \varepsilon \in (0,\varepsilon_0) $, there exists $ T_\varepsilon \in (0,\infty) $ such that there exists a unique strong solution $ (\xi, v) $ to \eqref{eq:perturbation} in the time interval $ [0,T_\varepsilon] $ with
	\begin{equation}\label{le-regularity}
	\begin{gathered}
		\xi \in L^\infty(0,T_\varepsilon; H^2(\Omega_h\times 2\mathbb T)), ~\dt \xi \in L^\infty(0,T_\varepsilon; H^1(\Omega_h\times 2\mathbb T)),\\
		v \in L^\infty(0,T_\varepsilon; H^2(\Omega_h\times 2\mathbb T)) \cap L^2(0,T_\varepsilon; H^3(\Omega_h\times 2\mathbb T)), \\
		\dt v \in L^\infty(0,T_\varepsilon; L^2(\Omega_h\times 2\mathbb T)) \cap L^2(0,T_\varepsilon; H^1(\Omega_h\times 2\mathbb T)).
	\end{gathered}
	\end{equation}
	Moreover, there exist positive constants $ C_0, C_1 $ independent of $ \varepsilon $, and $ C_2 = C_2(\varepsilon, C_0M_0, C_1M_1) $ such that
	\begin{equation}\label{le-priori-estimate}
		\begin{gathered}
			\sup_{0\leq t\leq T_\varepsilon} \norm{\xi(t)}{H^2}^2 \leq C_0M_0, ~ \sup_{0\leq t \leq T_{\varepsilon}} \norm{v(t)}{H^2}^2 + \int_0^{T_\varepsilon} \norm{\nabla v(t)}{\Hnorm{2}}^2 \,dt \leq C_1 M_1,\\
			\sup_{0\leq t\leq T_{\varepsilon}} \lbrace \norm{\xi_t(t)}{H^1}^2 + \norm{\dt v(t)}{L^2}^2 \rbrace + \int_0^{T_\varepsilon} \norm{\dt v(t)}{H^1}^2 \,dt \leq C_2.
		\end{gathered}
	\end{equation}
	Here $ C_0 \in (1,\infty) $, $ C_1, C_2 $ are determined by \eqref{local-constrain-006}, \eqref{local-dt-density-1}, and \eqref{local-006}, below, and $ T_\varepsilon $ depends on $ M_0, M_1 $ and $ \varepsilon $.
\end{prop}
Proposition \ref{prop:local-existence} can be shown by applying the Banach fixed point theorem. In the following, without going into too much details, we will only sketch the proper steps to construct this local strong solution.
\begin{proof}[Sketch of constructing strong solutions]
	Let $ \xi',v' $ be regular enough functions. Consider the following linear system associated with \eqref{eq:perturbation}:
	\begin{equation}\label{eq:linearized-eq}
		\begin{cases}
			\dt \xi + \overline v' \cdot \nablah \xi + \dfrac{\gamma-1}{\varepsilon} \dvh \overline v' = 0 & \text{in} ~ \Omega_h \times 2 \mathbb T,\\
			\dt v + v' \cdot \nablah v + w' \dz v + \dfrac{c^2e^{\varepsilon\xi'}}{\varepsilon} \nablah \xi' \\
			~~~~ ~~~~ = c_1 e^{-\varepsilon\alpha\xi'} \bigl( \mu \deltah v + \lambda \nablah \dvh v + \partial_{zz} v \bigr) & \text{in} ~ \Omega_h \times 2 \mathbb T, \\
			\dz \xi = 0 & \text{in} ~ \Omega_h \times 2 \mathbb T,
		\end{cases}
	\end{equation}
	where $ w' $ is given by
	\begin{equation}\label{eq:linearized-vertical-velocity}
		w':= - \int_0^z (\varepsilon\alpha \widetilde v' \cdot \nablah \xi' + \dvh \widetilde v' ) \,dz.
	\end{equation}
	Then for $ (\xi', v') $ with $ (\xi', v')\bigr|_{t=0} = (\xi_0, v_0) $ and satisfying the same regularity and bounds as in \eqref{le-regularity} and \eqref{le-priori-estimate}, there is a unique solution to system \eqref{eq:linearized-eq} with initial data $ (\xi, v)\bigr|_{t=0} = (\xi_0, v_0) $, after applying the standard existence theory for linear hyperbolic and parabolic equations.
	Moreover, similar {\it a priori} estimates as in our previous work \cite{LT2018a} show that the solution $ (\xi, v) $ satisfies the same regularity and bounds of norms in \eqref{le-regularity} and \eqref{le-priori-estimate}. We define the following function framework in order to apply the Banach fixed point theorem.
	
	Consider the function space
	\begin{equation}\label{fixed-point-spaces}
		\begin{aligned}
			& \mathfrak Y_{T_\varepsilon} := \bigl\lbrace (\xi, v) | \xi \in L^\infty(0,T_\varepsilon; L^2(\Omega_h\times 2\mathbb T)),\\
			& ~~~~ v\in L^\infty(0,T_\varepsilon ;L^2(\Omega_h\times 2\mathbb T)) \cap L^2 (0,T_\varepsilon ;H^1(\Omega_h\times2\mathbb T)), \\
			& ~~~~ \dz \xi = 0
			\bigr\rbrace
		\end{aligned}
	\end{equation}
	and the following subset of $ \mathfrak Y_{T_\varepsilon} $
	\begin{equation}\label{fixed-point-set}
%		\begin{aligned}
			\mathfrak X_{T_\varepsilon} := \bigl\lbrace (\xi, v) | (\xi, v) \in \mathfrak Y_{T_\varepsilon} ~ \text{and the bounds in} ~ \eqref{le-priori-estimate} ~ \text{hold}
%			\xi \in L^\infty(0,T_\varepsilon ;H^2(\Omega_h \times 2\mathbb T)), \\
%			& ~~~~ v\in L^\infty(0,T_\varepsilon ;H^2(\Omega_h\times 2\mathbb T)) \cap L^2 (0,T_\varepsilon ;H^3(\Omega_h\times2\mathbb T)), \\
%			& ~~~~ \dz \xi = 0
			\bigr\rbrace.
%		\end{aligned}
	\end{equation}
	For any $ (\xi, v) \in \mathfrak Y_{T_\varepsilon} $, define the norm
	\begin{equation}\label{fixed-point-norms}
		\begin{aligned}
		& \norm{(\xi, v)}{\mathfrak Y_{T_\varepsilon}}^2 := \norm{\xi}{L^\infty(0,T_\varepsilon;L^2(\Omega_h\times 2\mathbb T))}^2+ \norm{v}{L^\infty(0,T_\varepsilon;L^2(\Omega_h\times 2\mathbb T))}^2 \\
		& ~~~~ + \norm{\nabla v}{L^2(0,T_\varepsilon;L^2(\Omega_h\times 2\mathbb T))}^2.
		\end{aligned}
	\end{equation}
	Then $ (\mathfrak Y_{T_\varepsilon}, \norm{\cdot}{\mathfrak Y_{T_\varepsilon}}) $ is a complete metric space.
	In addition, we define the following map: with same initial data $ (\xi',v')\bigr|_{t=0} =
	(\xi,v)\bigr|_{t=0} = (\xi_0, v_0) $,
	\begin{equation}\label{fixed-point-map}
	\begin{aligned}
	& \mathcal T: (\xi',v') \leadsto (\xi, v) ~~~ \text{mapping $(\xi',v') \in \mathfrak X_{T_\varepsilon} \subset \mathfrak Y_{T_\varepsilon} $}\\
	& ~~~~ ~~~~ \text{
		to the solution to system \eqref{eq:linearized-eq}}
	\end{aligned}
	\end{equation}
	with trivial extension outside the set $ \mathfrak X_{T_{\varepsilon}} $.
	
	Then we show that, For $ \varepsilon_0\in (0,1) $ small enough and any $ \varepsilon\in (0,\varepsilon_0) $, there exists $ T_\varepsilon \in (0,\infty) $, such that $ \mathcal T: \mathfrak Y_{T_\varepsilon} \mapsto \mathfrak Y_{T_\varepsilon} $
	is a contraction mapping in $ \mathfrak X_{T_\varepsilon} $. Notice, once this is proved, then the Banach fixed point theorem implies that we have a unique strong solution to \eqref{eq:perturbation} with $ \varepsilon, T_\varepsilon $ as described above.
	%Notice, in this terminology, as mentioned before, $ (\xi, v) = \mathcal T(\xi',v') \in \mathfrak X_{T_\varepsilon} $ for $ (\xi',v') \in \mathfrak X_{T_\varepsilon} $.
	
	In the rest of this subsection, we focus on showing that for $ \varepsilon \in (0,\varepsilon_0) $ with $ \varepsilon_0 $ small enough, and some $ T_{\varepsilon} \in (0,\infty) $ depending on $ \varepsilon $, we have:
	\begin{enumerate}
		\item[(1)] $ \mathcal T $ maps $ \mathfrak X_{T_\varepsilon} $ into $ \mathfrak X_{T_\varepsilon} $; i.e., $ (\xi, v) = \mathcal T(\xi', v') \in \mathfrak X_{T_\varepsilon} $ for any $ (\xi',v') \in \mathfrak X_{T_\varepsilon} $;
		\item[(2)] $ \mathcal T $ is a contraction mapping in $ \mathfrak X_{T_\varepsilon} $ with respect to the topology of $ \mathfrak Y_{T_\varepsilon} $; i.e., for any $ (\xi_i', v_i') \in \mathfrak X_{T_\varepsilon} \subset \mathfrak Y_{T_\varepsilon} $ and $ (\xi_i, v_i) = \mathcal T (\xi_i', v_i') $, $ i=1,2 $, one has
		\begin{equation}\label{contraction}
			\norm{(\xi_{1}-\xi_2,v_{1}-v_{2})}{\mathfrak Y_{T_\varepsilon}} \leq q \norm{(\xi_{1}' - \xi'_2,v_{1}' - v_{2}')}{\mathfrak Y_{T_\varepsilon}},
		\end{equation}
		for some $ q \in (0,1) $.
	\end{enumerate}
We will only show the corresponding {\it a priori} estimates to show $(1)$ and $(2)$.

{\par\noindent \bf Proof of $ (1) $:} {\bf Estimates of $ \xi $.}
First, we shall derive the estimates of $ \xi $ from \subeqref{eq:linearized-eq}{1}. We shall only show the highest order estimates. After applying $ \partial_{hh} $ to \subeqref{eq:linearized-eq}{1}, we have
\begin{equation}\label{local-001}
	\dt \xi_{hh} + \overline{v}' \cdot \nablah \xi_{hh} = - 2  \overline{v}'_h \cdot\nablah \xi_h - \overline{v}'_{hh}\cdot \nablah \xi - \dfrac{\gamma-1}{\varepsilon} \dvh \overline{v}'_{hh}.
\end{equation}
Then after taking the $ L^2 $-inner product of \eqref{local-001} with $ 2 \xi_{hh} $ in $ \Omega_h $, it follows
\begin{align*}
	& \dfrac{d}{dt} \normh{\xi_{hh}}{L^2}^2 = \int_{\Omega_h} \dvh \overline{v}'\abs{\xi_{hh}}{2} \idxh - 4 \int_{\Omega_h} (\overline{v}'_h\cdot\nablah \xi_h) \xi_{hh} \idxh \\
	& ~~~~ ~~~~ - 2 \int_{\Omega_h} ( \overline{v}_{hh}' \cdot \nablah \xi) \xi_{hh} \idxh  - \dfrac{2(\gamma-1)}{\varepsilon} \int_{\Omega_h} \dvh \overline{v}_{hh}' \xi_{hh} \idxh \\
	& ~~~~ \leq C \norm{\nabla v'}{H^2} \normh{\xi}{H^2}^2 + C \varepsilon^{-1} \norm{\nabla  v'}{H^2} \normh{\xi_{hh}}{L^2} .
\end{align*}
Similar arguments also hold for $ \normh{\xi_h}{\Lnorm{2}} $ and $ \normh{\xi}{\Lnorm{2}} $, and therefore we have
\begin{align*}
	& \dfrac{d}{dt} \normh{\xi}{H^2}^2 \leq C \norm{\nabla v'}{H^2} \normh{\xi}{H^2}^2 + C\varepsilon^{-1} \norm{\nabla v'}{H^2}\normh{\xi}{H^2}\\
	& ~~~~ \leq 2 C \norm{\nabla v'}{H^2} \normh{\xi}{H^2}^2 + C \varepsilon^{-2} \norm{\nabla v'}{\Hnorm{2}}.
\end{align*}
Then applying the Gr\"onwall inequality and the H\"older inequality yields
\begin{equation}\label{local-h2-density}
	\begin{aligned}
	&	\sup_{0\leq t \leq T_\varepsilon}\norm{\xi(t)}{H^2}^2 \leq e^{\int_0^{T_\varepsilon} 2 C \norm{\nabla v'}{H^2}  \,dt }\biggl(\norm{\xi_0}{H^2}^2 \\
	& ~~~~  + C  \varepsilon^{-2} \int_0^{T_\varepsilon}  \norm{\nabla v'}{H^2} \,dt \biggr) \leq e^{2CT_\varepsilon^{1/2}(\int_0^{T_\varepsilon}  \norm{\nabla  v'}{H^2}^2  \,dt)^{1/2} }\biggl(\norm{\xi_0}{H^2}^2 \\
	& ~~~~ + C  \varepsilon^{-2} T_{\varepsilon}^{1/2}(\int_0^{T_\varepsilon}  \norm{\nabla v'}{H^2}^2 \,dt)^{1/2}\biggr) \\
	& ~~~~ \leq e^{2 C T_\varepsilon^{1/2}(C_1M_1)^{1/2})}(M_0 + C\varepsilon^{-2}  T_\varepsilon^{1/2}(C_1M_1)^{1/2}) \leq C_0 M_0,
	\end{aligned}
\end{equation}
where we have chosen $ T_\varepsilon $ sufficiently small in the last inequality and $ C_0 \in (1,\infty) $.

On the other hand, from \subeqref{eq:linearized-eq}{1}, we have
\begin{equation}\label{local-dt-density-1}
	\begin{aligned}
		& \normh{\dt \xi}{H^1}^2 \leq C \normh{\overline{v}'\cdot \nablah \xi}{H^1}^2 + \varepsilon^{-2} C \normh{\nablah \overline{v}'}{H^1}^2 \leq C (\norm{\xi}{H^2}^2 + \varepsilon^{-2}) \norm{v'}{H^2}^2\\
		& ~~~~ \leq C (C_0M_0 + \varepsilon^{-2}) C_1M_1 < \infty .
	\end{aligned}
\end{equation}
%Therefore, for $ T $ small enough,
%\begin{equation}\label{local-dt-density-2}
%	\int_0^T \norm{\dt \xi}{H^1}^2 \,dt \leq C (C_0M_0/2 + 1) C_1M_1 T \leq C_0M_0/2.
%\end{equation}
%This finishes the estimates on $ \xi $.
{\bf Estimates of $ v $.}
Next we shall present the estimates of $ v $. Similarly, we will only sketch the highest order estimates. After applying $ \partial_{hh} $ to \subeqref{eq:linearized-eq}{2}, one obtains
\begin{equation}\label{local-002}
	\begin{aligned}
		& \dt v_{hh} + v' \cdot \nablah v_{hh} +  w' \dz v_{hh} - c_1 e^{-\varepsilon\alpha\xi'} ( \mu \deltah v_{hh} + \lambda \nablah\dvh v_{hh} + \partial_{zz}v_{hh} ) \\
		& ~~~~  =  - \bigl(\dfrac{c^2 e^{\varepsilon \xi'}}{\varepsilon} \nablah \xi' \bigr)_{hh} - v'_{hh}\cdot \nablah v - 2  v'_{h} \cdot \nablah v_h -  w'_{hh} \dz v - 2  w'_h \dz v_h \\
		& ~~~~ ~~~~ - \varepsilon c_1\alpha e^{-\varepsilon \alpha \xi'} ( \xi'_{hh} - \varepsilon \alpha (\xi'_h)^2 ) (\mu \deltah v + \lambda \nablah \dvh v + \partial_{zz} v ) \\
		& ~~~~ ~~~~ - 2 \varepsilon c_1 \alpha e^{-\varepsilon \alpha  \xi'}\xi'_h (\mu \deltah v_{h} + \lambda \nablah \dvh v_h + \partial_{zz} v_h ).
	\end{aligned}
\end{equation}
Then after taking the $ L^2 $-inner product of \eqref{local-002} with $ 2 v_{hh} $ in $ \Omega_h \times 2 \mathbb T $, it follows, after applying integration by parts,
\begin{align*}
	& \dfrac{d}{dt} \norm{v_{hh}}{L^2}^2 + 2 c_1 \int e^{-\varepsilon \alpha \xi'} ( \mu \abs{\nablah v_{hh}}{2} + \lambda \abs{\dvh v_{hh}}{2} + \abs{\dz v_{hh}}{2}) \idx \\
	& ~~~~ = \int \bigl( \dfrac{2c^2 e^{\varepsilon \xi'}}{\varepsilon}\nablah \xi'\bigr)_h \cdot v_{hhh} \idx +  \int \biggl( \dvh v' \abs{v_{hh}}{2} -2 (v'_{hh} \cdot\nablah v) \cdot v_{hh} \\
	& ~~~~ ~~~~ - 4 ( v'_h \cdot \nablah v_h) \cdot v_{hh} - \dvh \widetilde{v}' \abs{v_{hh}}{2} \biggr) \idx - 2 \int \biggl( w'_h \dz v_h \cdot v_{hh} \\
	& ~~~~ ~~~~ - w'_h \dz v \cdot v_{hhh} \biggr) \idx   - \varepsilon\alpha \int ( \widetilde{v}' \cdot \nablah \xi') \abs{v_{hh}}{2} \idx \\
	& ~~~~ - 2 \varepsilon c_1 \alpha \int e^{-\varepsilon \alpha \xi'} ({\xi}'_{hh} - \varepsilon \alpha ({\xi}_h')^2 ) (\mu \deltah v + \lambda \nablah \dvh v + \partial_{zz} v ) \cdot v_{hh} \idx \\
	& ~~~~ - 4 \varepsilon c_1 \alpha \int e^{-\varepsilon\alpha \xi'} {\xi}'_h (\mu \deltah v_{h} + \lambda \nablah \dvh v_h + \partial_{zz} v_h ) \cdot v_{hh}\idx \\
	& ~~~~ + 2 \varepsilon c_1 \alpha \int e^{-\varepsilon\alpha\xi'} (\mu (\nablah \xi' \cdot \nablah v_{hh}) \cdot v_{hh} + \lambda (v_{hh} \cdot \nablah \xi') (\dvh v_{hh}) ) \idx\\
	& ~~~~ =: L_1 + L_2 + L_3 + L_4 + L_5 + L_6 + L_7,
\end{align*}
where we have substituted the identity
\begin{equation*}
%	\begin{aligned}
		w_z' = - \varepsilon\alpha \widetilde{v}' \cdot \nablah \xi' - \dvh \widetilde{v}',
%	\end{aligned}
\end{equation*}
which is obtained by taking $ \dz$ to
\eqref{eq:linearized-vertical-velocity},

Similarly as in the last section, from \eqref{eq:linearized-vertical-velocity}, we have
\begin{equation}\label{local-vertical-velocity-001}
%	\begin{aligned}
		w'_h = - \int_0^z \bigl( \varepsilon \alpha \widetilde{v}'_h \cdot \nablah \xi' + \varepsilon \alpha \widetilde{v}'\cdot\nablah \xi'_h + \dvh \widetilde{v}'_h \bigr) \,dz.
%	\end{aligned}
\end{equation}
Then following similar arguments as before, one can derive, for any $ \delta \in (0,1) $ and fixed $ \varepsilon \in (0,1) $,
\begin{equation}
	\begin{aligned}
	& L_1 + L_2 + L_3 + L_4 + L_5 + L_6 + L_7 \leq  \bigl( \delta + \varepsilon \mathcal H_1(\norm{\xi'}{H^2}, \norm{v'}{H^2})  \bigr)\\
	& ~~~~\times \norm{\nabla v}{H^2}^2
	 + \delta \norm{\nabla v'}{H^2}^2 \norm{v}{H^2}^2 +  \mathcal H_2(\varepsilon, C_\delta, \norm{\xi'}{H^2}, \norm{v'}{H^2}) \norm{v}{H^2}^2 \\
	 & ~~~~ +  \mathcal H_3(\varepsilon, C_\delta, \norm{\xi'}{H^2}, \norm{v'}{H^2}),
	 \end{aligned}
\end{equation}
where $ C_\delta \simeq \delta^{-1} $. We remind readers that we have been using $ \lbrace \mathcal H_i\rbrace_{i=1,2,3}$ to denote regular functions of the arguments with property $ \mathcal H_i(0) = 0 $.
%\par\noindent\rule{\linewidth}{1pt}
%{\color{blue}
Details are listed below, for readers' reference:
\begin{align*}
	& L_1 \lesssim \norm{v_{hhh}}{L^2} \bigl( \varepsilon^{-1} e^{\varepsilon \norm{\xi'}{H^2}} \norm{\xi'}{H^2} + e^{\varepsilon \norm{\xi'}{H^2}} \norm{\xi'}{H^2}^2 \bigr)\\
	& ~~ \lesssim \delta \norm{\nabla v}{H^2}^2 + C_\delta e^{2 \varepsilon \norm{\xi'}{H^2}} \bigl( \varepsilon^{-2} \norm{\xi'}{H^2}^2 + \norm{\xi'}{H^2}^4 \bigr),\\
	& L_2 \lesssim \norm{\nabla v'}{L^6} \norm{v_{hh}}{L^3} \norm{v_{hh}}{L^2} + \norm{v'_{hh}}{L^2} \norm{\nablah v}{L^3} \norm{v_{hh}}{L^6} \\
	& ~~ \lesssim \norm{v'}{H^2} \bigl( \norm{v}{H^2}^{3/2} \norm{\nabla v}{H^2}^{1/2}
	+ \norm{v}{H^2} \norm{\nabla v}{H^2}\bigr) \\
	& ~~ \lesssim \delta \norm{\nabla v}{H^2}^2 + C_\delta \bigl(\norm{v'}{H^2}^{2} + 1\bigr) \norm{v}{H^2}^2, \\
	& L_3 \lesssim \int_0^1 \biggl(\varepsilon \normh{ v_h'}{L^4} \normh{ \xi_h'}{L^4} + \varepsilon \normh{ v'}{L^\infty} \normh{ \xi'}{H^2} \biggr) \,dz \cdot \int_0^1 \biggl( \normh{\dz v_h}{L^4} \normh{v_{hh}}{L^4} \\
	& ~~~~  + \normh{\dz v}{L^\infty} \normh{v_{hhh}}{L^2} \biggr)\,dz
	 + \int_0^1 \normh{v_{hh}'}{L^4}\,dz \cdot \int_0^1 \biggl( \normh{\dz v_h}{L^2} \normh{v_{hh}}{L^4} \\
	& ~~~~ + \normh{\dz v}{L^4} \normh{v_{hhh}}{L^2} \biggr) \,dz \lesssim \int_0^1 \varepsilon \normh{v'}{H^2} \normh{\xi' }{H^2}\,dz \cdot\int_0^1 \biggl( \normh{\nabla v}{H^2}( \normh{\nabla v}{H^1}\\
	& ~~~~ ~~~~ + \normh{\nabla v}{H^2} ) \biggr) \,dz + \int_0^1 \normh{v'}{H^2}^{1/2}\normh{\nabla  v'}{H^2}^{1/2} \,dz  \cdot \int_0^1 \biggl( \normh{\nabla v}{H^1}^{3/2}\normh{\nabla v}{H^2}^{1/2} \\
	& ~~~~ ~~~~ + \normh{\nabla v}{H^1} \normh{\nabla v}{H^2} \biggr)\,dz \lesssim \bigl(  \varepsilon \norm{ \xi'}{H^2} \norm{ v'}{H^2} + \delta \bigr) \norm{\nabla v}{H^2}^2 \\
	& ~~~~  + \delta \norm{\nabla  v'}{H^2}^2 \norm{v}{H^2}^2 + \varepsilon \norm{ \xi'}{H^2} \norm{ v'}{H^2}  \norm{v}{H^2}^2 \\
	& ~~~~ + C_\delta (\norm{ v'}{H^2}^2 + 1) \norm{v}{H^2}^2 , \\
	& L_4 \lesssim \varepsilon \norm{v'}{H^2} \norm{\nablah \xi'}{L^6} \norm{v_{hh}}{L^2} \norm{v_{hh}}{L^3} \lesssim \varepsilon \norm{v'}{H^2} \\
	& ~~~~ ~~~~ \times \norm{\xi'}{H^2} \norm{v}{H^2}^{3/2} \norm{\nabla v}{H^2}^{1/2}
	 \lesssim \delta \norm{\nabla v}{H^2}^2 \\
	& ~~~~ + C_\delta \varepsilon^{4/3} \norm{ v'}{H^2}^{4/3} \norm{ \xi'}{H^2}^{4/3} \norm{v}{H^2}^2,\\
	& L_5 \lesssim \varepsilon e^{\varepsilon \norm{\xi'}{H^2}} \bigl( \norm{\xi'}{H^2}+ \varepsilon \norm{\xi'}{H^2}^2 \bigr) \norm{\nabla^2 v}{L^3} \norm{v_{hh}}{L^6} \lesssim \varepsilon e^{\varepsilon \norm{\xi'}{H^2}} \\
	& ~~~~ ~~~~ \times \bigl( \norm{\xi'}{H^2}
	+ \varepsilon \norm{\xi'}{H^2}^2 \bigr) \times \norm{v}{H^2}^{1/2} \norm{\nabla v}{H^2}^{3/2} \lesssim \delta \norm{\nabla v}{H^2}^2 \\
	& ~~~~ + C_\delta \varepsilon^{4} \bigl( \norm{\xi'}{H^2}^4 + \varepsilon^4 \norm{\xi'}{H^2}^8 \bigr)e^{4\varepsilon \norm{\xi'}{H^2}} \norm{v}{H^2}^2, \\
	& L_6 + L_7 \lesssim \varepsilon e^{\varepsilon \norm{\xi'}{H^2}} \norm{\xi'}{H^2} \norm{\nabla v}{H^2}^{3/2} \norm{v}{H^2}^{1/2} \lesssim \delta \norm{\nabla v}{H^2}^2 \\
	& ~~~~ + \varepsilon^{4} \norm{\xi'}{H^2}^{4} e^{4\varepsilon \norm{\xi'}{H^2}} \norm{v}{H^2}^2.
\end{align*}

Similar estimates also hold for $ \norm{v_{hz}}{\Lnorm{2}}, \norm{v_{zz}}{\Lnorm{2}}, \norm{v_h}{\Lnorm{2}}, \norm{v_{z}}{\Lnorm{2}}, \norm{v}{\Lnorm{2}} $.
Then we have arrived at the estimate
\begin{equation}\label{local-003}
	\begin{aligned}
	& \dfrac{d}{dt} \norm{v}{H^2}^2 + c_1 \min\lbrace \mu ,\lambda, 1\rbrace \norm{\nabla v}{H^2}^2 \leq  \bigl( \delta + \varepsilon \mathcal H_1(\norm{\xi'}{H^2}, \norm{v'}{H^2})  \bigr) \\
	& ~~~~ \times \norm{\nabla v}{H^2}^2
	+ \delta \norm{\nabla v'}{H^2}^2 \norm{v}{H^2}^2 +  \mathcal H_2(\varepsilon, C_\delta, \norm{\xi'}{H^2}, \norm{v'}{H^2}) \norm{v}{H^2}^2 \\
	 & ~~~~ +  \mathcal H_3(\varepsilon, C_\delta, \norm{\xi'}{H^2}, \norm{v'}{H^2}) ,
	\end{aligned}
\end{equation}
where we have chosen $ \varepsilon \in (0,\varepsilon_0) $, with $ \varepsilon_0 $ small enough such that
\begin{equation}\label{local-constrain-002}
	\varepsilon \alpha \norm{\xi'}{H^2} \leq \varepsilon \alpha (C_0M_0)^{1/2} \leq \log 2, ~~ \text{and thus} ~ 1/2 \leq e^{-\varepsilon\alpha \xi'} < 2.
%	e^{-\varepsilon\alpha \hat \xi} \geq e^{-\varepsilon \alpha \norm{\hat \xi}{H^2}} \geq e^{-\varepsilon \alpha C_0M_0} \geq 1/2.
\end{equation}
Next, we choose $ \varepsilon_0, \delta $ small enough such that
\begin{equation}\label{local-constrain-003}
	\begin{aligned}
	& \delta + \varepsilon_0 \mathcal H_1( \norm{\xi'}{H^2}, \norm{v'}{H^2}) \leq \delta + \varepsilon_0 \mathcal H_1( (C_0M_0)^{1/2}, (C_1M_1)^{1/2}) \\
	& ~~~~ \leq \dfrac{c_1 \min\lbrace \mu, \lambda, 1\rbrace}{2}.
	\end{aligned}
\end{equation}
Then for $ \varepsilon \in (0,\varepsilon_0) $, after applying the Gr\"onwall inequality to \eqref{local-003}, we have
\begin{equation}\label{local-004}
	\begin{aligned}
		& \sup_{0\leq t\leq T_\varepsilon} \norm{v(t)}{H^2}^2 + \dfrac{c_1\min\lbrace\mu,\lambda,1\rbrace}{2} \int_0^{T_\varepsilon} \norm{\nabla v}{H^2}^2 \,dt \\
		& ~~~~ \leq e^{\delta \int_0^{T_\varepsilon} \norm{\nabla v'}{H^2}^2 \,dt + T_\varepsilon \mathcal H_2(\varepsilon, C_\delta, \norm{ \xi'}{H^2},\norm{ v'}{H^2})} \\
		& ~~~~ \times \bigl(  \norm{v_0}{H^2}^2 + T_\varepsilon \mathcal H_3(\varepsilon, C_\delta, \norm{ \xi'}{H^2},\norm{ v'}{H^2}) \bigr).
	\end{aligned}
\end{equation}
Now we let $ \delta $ small enough such that
\begin{equation}\label{local-constrain-004}
	\delta \int_0^{T_\varepsilon} \norm{\nabla v'}{H^2}^2 \,dt \leq \delta C_1M_1 \leq \log 2.
\end{equation}
Then for fixed $ \varepsilon_0, \delta $ satisfying \eqref{local-constrain-003} and \eqref{local-constrain-004}, and $ \varepsilon \in (0,\varepsilon_0) $, let $ T_\varepsilon $ small enough, depending on $ \varepsilon, \delta $, such that
\begin{equation}\label{local-constrain-005}
	\begin{aligned}
		& T_\varepsilon \mathcal H_2(\varepsilon,C_\delta, \norm{\xi'}{H^2},\norm{v'}{H^2}) \leq T_\varepsilon \mathcal H_2(\varepsilon,C_\delta, (C_0M_0)^{1/2},(C_1M_1)^{1/2}) < \log 2,\\
		& T_\varepsilon \mathcal H_3(\varepsilon,C_\delta, \norm{\xi'}{H^2},\norm{v'}{H^2})\leq T_\varepsilon \mathcal H_3(\varepsilon,C_\delta, (C_0M_0)^{1/2},(C_1M_1)^{1/2}) < \dfrac{M_1}{2}.
	\end{aligned}
\end{equation}
Then \eqref{local-004} yields
\begin{equation}\label{local-005}
	\begin{aligned}
		& \sup_{0\leq t\leq T_\varepsilon} \norm{v(t)}{H^2}^2 + \dfrac{c_1\min\lbrace\mu,\lambda,1\rbrace}{2} \int_0^{T_\varepsilon} \norm{\nabla v}{H^2}^2 \,dt \leq 4 (M_1 + M_1/2) \\
		& ~~~~ = 6 M_1  \leq \min\biggl\lbrace \dfrac{1}{3}, \dfrac{c_1 \min\lbrace \mu, \lambda, 1\rbrace}{6} \biggr\rbrace C_1M_1.
	\end{aligned}
\end{equation}
Here we have required $ C_1 $ to be sufficiently large such that
\begin{equation}\label{local-constrain-006}
	6 \leq \min\biggl\lbrace \dfrac{1}{3}, \dfrac{c_1 \min\lbrace \mu, \lambda, 1\rbrace}{6} \biggr\rbrace C_1.
\end{equation}
On the other hand, from \subeqref{eq:linearized-eq}{2}, we have
\begin{align*}
	& \norm{\dt v}{L^2}^2 \lesssim \bigl( \norm{v'}{H^2}^2 + e^{2\varepsilon \alpha \norm{\xi'}{H^2}} + \varepsilon^2 \norm{v'}{H^1}^2 \norm{\xi'}{H^2}^2 \bigr) \norm{v}{H^2}^2\\
	& ~~~~ + \varepsilon^{-2} e^{2\varepsilon \norm{\xi'}{H^2}}\norm{\xi'}{H^1}^2 ,\\
	& \norm{\dt v}{H^1}^2 \lesssim \bigl( \norm{v'}{H^2}^2 + e^{2\varepsilon \alpha \norm{\xi'}{H^2}} + \varepsilon^2 \norm{v'}{H^1}^2 \norm{\xi'}{H^2}^2 \bigr)  \norm{v}{H^2}^2 \\
	& ~~~~ + \bigl( (\norm{\xi'}{H^2}^2 + 1)e^{2\varepsilon \alpha \norm{\xi'}{H^2}} + \varepsilon^2 \norm{\xi'}{H^2}^2 \norm{ v'}{H^2}^2 + \norm{ v'}{H^2}^2 \bigr)\norm{\nabla v}{H^2}^2\\
	& ~~~~ + \varepsilon^{-2}e^{2\varepsilon \norm{\xi'}{H^2}} \norm{ \xi'}{H^2}^2 + e^{2\varepsilon \norm{\xi'}{H^2}} \norm{\xi'}{H^2}^4,
\end{align*}
after applying similar arguments as in \eqref{ue:206} and \eqref{ue:205}.
Then we have the following
\begin{equation}\label{local-006}
	\begin{aligned}
		& \norm{v_t}{L^2}^2 \leq \mathcal H_1(\varepsilon, \norm{\xi'}{H^2}, \norm{v'}{H^2}) \norm{v}{H^2}^2+ \mathcal H_2(\varepsilon, \norm{\xi'}{H^2}) \\
		& ~~~~ \leq \mathcal H_1( \varepsilon, (C_0M_0)^{1/2}, (C_1M_1)^{1/2}) 6 M_1 + \mathcal H_2( \varepsilon, (C_0M_0)^{1/2})  < \infty , \\
%		& \int_0^T \norm{v_t}{L^2}^2 \leq % \int_0^T \mathcal H_1(\varepsilon, \norm{\hat \xi}{H^2}, \norm{\hat v}{H^2}) \norm{v}{H^2}^2+ \mathcal H_2(\varepsilon, \norm{\hat \xi}{H^2}) \,dt\\
%		& ~~~~ \leq
%		\mathcal H_1( \varepsilon, C_0M_0, C_1M_1) 6 M_1 T + \mathcal H_2( \varepsilon, C_0M_0) T  \leq C_1M_1/3,\\
		& \int_0^T \norm{v_t}{H^1}^2 \leq \int_0^T \biggl( \mathcal H_1(\varepsilon, (C_0M_0)^{1/2}, (C_1M_1)^{1/2}) 6M_1 \\
		& ~~~~ + \mathcal H_2(\varepsilon, (C_0M_0)^{1/2},(C_1M_1)^{1/2}) \norm{\nabla v}{H^2}^2 + \mathcal H_3(\varepsilon, (C_0M_0)^{1/2}) \biggr)\,dt < \infty.
	\end{aligned}
\end{equation}
This finishes the proof of $ (1) $.

{\par\noindent \bf Proof of $ (2) $:} Denote by $ (\xi_{12},v_{12}) := (\xi_{1} - \xi_2, v_1 - v_2) $ and $ (\xi'_{12},v'_{12}) := (\xi'_{1} - \xi'_2, v'_1 - v'_2) $. Then $ (\xi_{12}, v_{12}) $ satisfies the following system:
\begin{equation}\label{eq:contract-map}
		\begin{cases}
			\dt \xi_{12}  = - \overline{v}'_{1} \cdot \nablah \xi_{12} - \overline{v}'_{12} \cdot \nablah \xi_2 - \dfrac{\gamma-1}{\varepsilon} \dvh \overline{v}'_{12}
			& \text{in} ~ \Omega_h \times 2\mathbb T,\\
			\dt v_{12} - c_1\bigl( \mu \deltah v_{12} + \lambda \nablah \dvh v_{12} + \partial_{zz} v_{12} \bigr) \\
			~~ = - v'_1 \cdot\nablah v_{12} - v'_{12} \cdot \nablah v_2 - w'_1 \dz v_{12} - (w'_1-w'_2 )\dz v_2\\
			~~~~ - \dfrac{c^2 e^{\varepsilon \xi'_1}}{\varepsilon} \nablah \xi_{12}' - \dfrac{c^2 e^{\varepsilon \xi'_1} - c^2 e^{\varepsilon \xi'_2}}{\varepsilon} \nablah \xi_{2}' + c_1 \bigl( e^{-\varepsilon\alpha \xi'_1} - 1\bigr) \\
			~~~~ ~~~~  \times \bigl( \mu \deltah v_{12} + \lambda \nablah \dvh v_{12} + \partial_{zz} v_{12} \bigr) \\
			~~~~ + c_1 \bigl( e^{-\varepsilon\alpha \xi'_1} - e^{-\varepsilon\alpha\xi'_2}\bigr)
			\bigl( \mu \deltah v_{2} + \lambda \nablah \dvh v_{2} + \partial_{zz} v_{2} \bigr)
			& \text{in} ~ \Omega_h \times 2\mathbb T,
		\end{cases}
	\end{equation}
	with $ (\xi_{12}, v_{12}) \bigr|_{t=0} = 0 $. Then after taking the $ L^2 $-inner product of \subeqref{eq:contract-map}{1} with $ 2 \xi_{12} $ in $ \Omega_h $, one has, for any $ \delta \in (0,1) $ with corresponding $ C_\delta \simeq \delta^{-1} $,
	\begin{align*}
		& \dfrac{d}{dt} \normh{\xi_{12}}{\Lnorm{2}}^2 = \int_{\Omega_h} \dvh \overline v_1' \abs{\xi_{12}}{2} \idxh - 2 \int_{\Omega_h} (\overline v_{12}' \cdot \nablah \xi_2) \xi_{12} \idxh\\
		& ~~~~ - \dfrac{2(\gamma-1)}{\varepsilon} \int_{\Omega_h} \dvh \overline v_{12}' \xi_{12} \idxh \lesssim \normh{\nablah \overline v_1'}{L^\infty} \normh{\xi_{12}}{\Lnorm{2}}^2 \\
		& ~~~~~ + \normh{\overline v_{12}'}{\Lnorm{4}} \normh{\nablah \xi_2}{\Lnorm{4}} \normh{\xi_{12}}{\Lnorm{2}} + \normh{\nablah \overline v_{12}'}{\Lnorm{2}} \normh{ \xi_{12}}{\Lnorm{2}} \lesssim \delta \norm{\nabla v_{12}'}{\Lnorm{2}}^2 \\
		& ~~~~ + C_\delta \bigl(1 + \norm{\nabla v_1'}{\Hnorm{2}} + \normh{\xi_2}{\Hnorm{2}}^2 \bigr) \normh{\xi_{12}}{\Lnorm{2}}^2 + C_\delta \norm{v_{12}'}{\Lnorm{2}}^2,
	\end{align*}
	where we have applied the H\"older, Sobolev embedding and Young inequalities.
	Therefore, applying the Gr\"onwall inequality in the above inequality yields,
	\begin{equation}\label{local-101}
		\begin{aligned}
			& \sup_{0\leq t\leq T_\varepsilon} \norm{\xi_{12}}{\Lnorm{2}}^2 \leq e^{C_\delta \int_0^{T_\varepsilon} (1+ \norm{\nabla v_1'}{\Hnorm{2}} + \norm{\xi_2}{\Hnorm{2}}^2) \,dt
			} \\
			& ~~~~ \times \int_0^{T_{\varepsilon}} \bigl( \delta \norm{\nabla v_{12}'}{\Lnorm{2}}^2 + C_\delta \norm{v_{12}'}{\Lnorm{2}}^2  \bigr) \,dt \\
			& ~~ \leq e^{C_\delta (1+C_0M_0) T_{\varepsilon} + C_\delta (C_1M_1)^{\frac 1 2} T_{\varepsilon}^{\frac 1 2}} \bigl( \delta \norm{\nabla v_{12}'}{\Lnorm{2}(0,T_\varepsilon;L^2(\Omega_h\times 2\mathbb T))}^2\\
			& ~~~~ + C_\delta T_{\varepsilon} \norm{v_{12}'}{\Lnorm{\infty}(0,T_\varepsilon;L^2(\Omega_h\times 2\mathbb T))}^2 \bigr).
		\end{aligned}
	\end{equation}
	On the other hand, after taking the $ L^2 $-inner product of \subeqref{eq:contract-map}{2} with $ 2 v_{12} $ in $ \Omega_h \times 2 \mathbb T $ and applying integration by parts, one has
	\begin{align*}
		& \dfrac{d}{dt} \norm{v_{12}}{\Lnorm{2}}^2 + c_1 \bigl( \mu \norm{\deltah v_{12}}{\Lnorm{2}}^2 + \lambda \norm{\dvh v_{12}}{\Lnorm{2}}^2 + \norm{\dz v_{12}}{\Lnorm{2}}^2 \bigr) \\
		& ~~~~ =   \int ( \dvh v'_1 -  \varepsilon \alpha \widetilde v_1' \cdot \nablah \xi_1' - \dvh \widetilde v'_1 ) \abs{v_{12}}{2} \idx  - \int ( v_{12}' \cdot \nablah v_2 ) \cdot v_{12} \idx\\
		& ~~~~ + \int \int_0^z \bigl( \varepsilon \alpha \widetilde v_{12}' \cdot \nablah \xi_1' + \dvh \widetilde v_{12}' \bigr) \,dz \times (\dz v_2 \cdot v_{12}) \idx \\
		& ~~~~ - \int \biggl( \int_0^z ( \varepsilon\alpha (\dvh \widetilde v_2') \xi_{12}' ) \,dz \times ( \dz v_2 \cdot v_{12}) \\
		& ~~~~ ~~~~ + \int_0^z ( \varepsilon\alpha \xi_{12}' \widetilde v_2' ) \,dz \cdot \nablah (\dz v_2 \cdot v_{12}) \biggr) \idx  \\
		& ~~~~
%		+ \int \int_0^z \varepsilon \alpha \widetilde v_2' \cdot \nablah \xi_{12}' \,dz \times (\dz v_2 \cdot v_{12} )\idx
		+ \int \xi_{12}' \bigl(  \dfrac{c^2 e^{\varepsilon \xi_1'}}{\varepsilon} \dvh v_{12} + c^2 e^{\varepsilon \xi_1'} v_{12} \cdot \nablah \xi_1' \bigr)
%		\dvh\bigl( \dfrac{c^2 e^{\varepsilon \xi_1'}}{\varepsilon} v_{12}\bigr)
		\idx \\
		& ~~~~ - \int \dfrac{c^2(e^{\varepsilon\xi_1'} - e^{\varepsilon \xi_2'})}{\varepsilon} \nablah \xi_2' \cdot v_{12}\idx - c_1 \int \biggl( \bigl( e^{-\varepsilon\alpha\xi_1'} - 1 \bigr) \bigl( \mu \abs{\nablah v_{12}}{2} \\
		& ~~~~ + \lambda \abs{\dvh v_{12}}{2} + \abs{\dz v_{12}}{2} \bigr) \biggr) \idx + c_1 \varepsilon \alpha \int \biggl( e^{-\varepsilon\alpha\xi_{1}'} \bigl( \mu (\nablah \xi_1' \cdot \nablah v_{12}) \cdot v_{12} \\
		& ~~~~ + \lambda ( v_{12} \cdot \nablah \xi_1') (\dvh v_{12}) \biggr) \idx + c_1 \int \biggl( \bigl( e^{-\varepsilon\alpha\xi_1'} - e^{-\varepsilon\alpha\xi_2'}\bigr) \\
		& ~~~~ \times \bigl( \mu \deltah v_2 + \lambda \nablah \dvh v_2 + \partial_{zz} v_2 \bigr) \cdot v_{12} \biggr) \idx =: L_8 + L_9 + L_{10} + L_{11} \\
		& ~~~~ + L_{12} + L_{13} + L_{14} + L_{15} + L_{16},
	\end{align*}
	where we have substituted \eqref{eq:linearized-vertical-velocity} for $ w'_i $, $ i=1,2 $.
	Consequently, after applying similar arguments as before, we have the following estimates of the terms on the right-hand side: for any $ \delta \in (0,1) $ and some constant $ C_\delta \simeq \delta^{-1} $,
	\begin{align*}
		& L_{8} \lesssim \bigl( \norm{\nablah v_1'}{\Lnorm{3}} + \norm{v_1'}{\Lnorm{6}} \norm{\nablah \xi_1'}{\Lnorm{6}}\bigr) \norm{v_{12}}{\Lnorm{2}} \norm{v_{12}}{\Lnorm{6}} \\
		& ~~ \lesssim \bigl( \norm{v_1'}{\Hnorm{2}}+ \norm{v_1'}{\Hnorm{1}} \norm{\xi_1'}{\Hnorm{2}} \bigr) \norm{v_{12}}{\Lnorm{2}} \bigl( \norm{\nabla v_{12}}{\Lnorm{2}} + \norm{v_{12}}{\Lnorm{2}}\bigr)\\
		& ~~ \lesssim \delta \norm{\nabla v_{12}}{\Lnorm{2}}^2 + \mathcal H( C_\delta, C_0M_0, C_1M_1) \norm{v_{12}}{\Lnorm{2}}^2,\\
		& L_9 \lesssim \norm{v_{12}'}{\Lnorm{2}} \norm{\nablah v_2}{\Lnorm{6}} \norm{v_{12}}{\Lnorm{3}} \lesssim \delta \norm{\nabla v_{12}}{\Lnorm{2}}^2 + C_\delta \norm{v_{12}}{\Lnorm{2}}^2 \\
		& ~~~~ + C_\delta \norm{v_2}{\Hnorm{2}}^2 \norm{v_{12}'}{\Lnorm{2}}^2,\\
		& L_{10} \lesssim \int_0^1 \bigl(\varepsilon \normh{v_{12}'}{\Lnorm{4}} \normh{\nablah \xi_1'}{\Lnorm{4}} + \normh{\nablah v_{12}'}{\Lnorm{2}} \bigr) \,dz \times \int_0^1 \normh{\dz v_2}{\Lnorm{4}} \normh{v_{12}}{\Lnorm{4}} \,dz \\
		& ~~ \lesssim \bigl( \varepsilon \norm{v_{12}'}{\Lnorm{2}}^{1/2} \norm{v_{12}'}{\Hnorm{1}}^{1/2} \norm{\xi_1'}{\Hnorm{2}} + \norm{\nablah v_{12}'}{\Lnorm{2}} \bigr) \\
		& ~~~~ \times \norm{\dz v_2}{\Hnorm{1}} \norm{v_{12}}{\Lnorm{2}}^{1/2}\norm{v_{12}}{\Hnorm{1}}^{1/2} \lesssim \delta \norm{\nabla v_{12}}{\Lnorm{2}}^2 + \delta \norm{\nabla v_{12}'}{\Lnorm{2}}^2 \\
		& ~~~~ + \mathcal H(C_\delta, \varepsilon, C_0M_0, C_1M_1) \bigl( \norm{v_{12}}{\Lnorm{2}}^2 + \norm{v_{12}'}{\Lnorm{2}}^2 \bigr),\\
		& L_{11} \lesssim \varepsilon \int_0^1 \normh{\nablah v_2'}{\Lnorm{8}} \normh{\xi_{12}'}{\Lnorm{2}}  \,dz \times \int_0^1 \normh{\dz v_2}{\Lnorm{8}} \normh{v_{12}}{\Lnorm{4}} \,dz  \\
		& ~~~~ + \varepsilon \int_0^1 \normh{\xi_{12}'}{\Lnorm{2}} \normh{v_{2}'}{\Lnorm{\infty}} \,dz \times \int_0^1 \bigl( \normh{\dz v_2}{\Lnorm{\infty}} \normh{\nablah v_{12}}{\Lnorm{2}} \\
		& ~~~~ + \normh{\nablah \dz v_2}{\Lnorm{4}} \normh{v_{12}}{\Lnorm{4}} \bigr) \,dz \lesssim \varepsilon \norm{v_2'}{\Hnorm{2}} \norm{\xi_{12}'}{\Lnorm{2}} \\
		& ~~~~ \times \bigl( \norm{v_2}{\Hnorm{2}} \norm{v_{12}}{\Lnorm{2}}^{1/2} \norm{v_{12}}{\Hnorm{1}}^{1/2} + \norm{\dz v_2}{\Lnorm{2}}^{1/2} \norm{\dz v_2}{\Hnorm{2}}^{1/2} \norm{\nablah v_{12}}{\Lnorm{2}} \\
		& ~~~~ + \norm{\nabla^2 v_2}{\Lnorm{2}}^{1/2} \norm{\nabla^2 v_2}{\Hnorm{1}}^{1/2} \norm{v_{12}}{\Lnorm{2}}^{1/2} \norm{v_{12}}{\Hnorm{1}}^{1/2} \bigr) \lesssim \delta \norm{\nablah v_{12}}{\Lnorm{2}}^2 \\
		& ~~~~ +  \norm{v_{12}}{\Lnorm{2}}^2 +  \mathcal H(C_\delta,\varepsilon, C_0M_0,C_1M_1) \norm{\nabla v_2}{\Hnorm{2}} \norm{\xi_{12}'}{\Lnorm{2}}^2, \\
		& L_{12} \lesssim \norm{\xi_{12}'}{\Lnorm{2}} \bigl(  \varepsilon^{-1} e^{\varepsilon \norm{\xi_1'}{\Hnorm{2}}} \norm{\nablah v_{12}}{\Lnorm{2}} + e^{\varepsilon\norm{\xi_1'}{\Hnorm{2}}} \norm{v_{12}}{\Lnorm{3}} \norm{\nablah \xi_1'}{\Lnorm{6}}\\
		& ~~ \lesssim \delta \norm{\nabla v_{12}}{\Lnorm{2}}^2 + \mathcal H (C_\delta, \varepsilon, C_0M_0, C_1M_1) \bigl( \norm{\xi_{12}'}{\Lnorm{2}}^2 + \norm{v_{12}}{\Lnorm{2}}^2 \bigr), \\
		& L_{13} \lesssim e^{\varepsilon (\norm{\xi_1'}{\Hnorm{2}} + \norm{\xi_2'}{\Hnorm{2}})}\norm{\xi_{12}'}{\Lnorm{2}} \norm{\nablah \xi_2'}{\Lnorm{6}} \norm{v_{12}}{\Lnorm{3}}\\
		& ~~ \lesssim \delta \norm{\nabla v_{12}}{\Lnorm{2}}^2 + \mathcal H (C_\delta, \varepsilon, C_0M_0, C_1M_1) \bigl( \norm{\xi_{12}'}{\Lnorm{2}}^2 + \norm{v_{12}}{\Lnorm{2}}^2 \bigr),\\
		& L_{14} \lesssim \varepsilon e^{\varepsilon \alpha \norm{\xi_1'}{\Hnorm{2}}} \norm{\xi_1'}{\Hnorm{2}} \norm{\nabla v_{12}}{\Lnorm{2}}^2 \lesssim \varepsilon e^{\varepsilon\alpha (C_0M_0)^{1/2}} (C_0M_0)^{1/2} \\
		& ~~~~ ~~~~ \times \norm{\nabla v_{12}}{\Lnorm{2}}^2, \\
		& L_{15} \lesssim \varepsilon e^{\varepsilon\alpha \norm{\xi_1'}{\Hnorm{2}} } \norm{\nablah \xi_1'}{\Lnorm{6}} \norm{\nabla v_{12}}{\Lnorm{2}} \norm{v_{12}}{\Lnorm{3}} \lesssim \delta \norm{\nabla v_{12}}{\Lnorm{2}}^2 \\
		& ~~~~ + \mathcal H(C_\delta, \varepsilon, C_0M_0) \norm{v_{12}}{\Lnorm{2}}^2, \\
		& L_{16} \lesssim \varepsilon e^{\varepsilon\alpha (\norm{\xi_1'}{\Hnorm{2}} + \norm{\xi_2'}{\Hnorm{2}})} \norm{\xi_{12}'}{\Lnorm{2}} \norm{\nabla^2 v_2}{\Lnorm{3}} \norm{v_{12}}{\Lnorm{6}}\\
		& ~~ \lesssim \delta \norm{\nabla v_{12}}{\Lnorm{2}}^2 + \delta \norm{v_{12}}{\Lnorm{2}}^2 + \mathcal H(C_\delta, \varepsilon, C_0M_0) (\norm{\nabla v_2}{\Hnorm{2}} + 1)  \norm{\xi_{12}'}{\Lnorm{2}}^2,
	\end{align*}
	where $ \mathcal H (\cdot) $, as before, is a regular function of the arguments.
	Consequently, one has
	\begin{align*}
		& \dfrac{d}{dt} \norm{v_{12}}{\Lnorm{2}}^2 + c_1 \bigl( \mu \norm{\nablah v_{12}}{\Lnorm{2}}^2 + \lambda \norm{\dvh v_{12}}{\Lnorm{2}}^2 + \norm{\dz v_{12}}{\Lnorm{2}}^2 \bigr) \\
		& ~~ \leq ( \delta + \varepsilon e^{\varepsilon \alpha (C_0M_0)^{1/2}}(C_0M_0)^{1/2}  )  \norm{\nabla v_{12}}{\Lnorm{2}}^2 + \delta \norm{\nabla v_{12}'}{\Lnorm{2}}^2 \\
		& ~~~~ + \mathcal H(C_\delta,\varepsilon, C_0M_0,C_1M_1)  \bigl( \norm{v_{12}}{\Lnorm{2}}^2 + \norm{v_{12}'}{\Lnorm{2}}^2 + \norm{\xi_{12}'}{\Lnorm{2}}' \bigr)\\
		& ~~~~ + \mathcal H(C_\delta,\varepsilon, C_0M_0,C_1M_1) \norm{\nabla v_2}{\Hnorm{2}} \norm{\xi_{12}'}{\Lnorm{2}}^2.
	\end{align*}
	Notice that $ C_0, C_1 $ are independent of $ \varepsilon $. For  $ \varepsilon_0, \delta \in (0,1) $ small enough and $ \varepsilon \in (0,\varepsilon_0] $, applying the Gr\"onwall inequality in the above inequality yields
	\begin{equation}\label{local-102}
		\begin{aligned}
			& \sup_{0\leq t\leq T_\varepsilon} \norm{v_{12}}{\Lnorm{2}}^2 + \dfrac{c_1 \min\lbrace \mu, \lambda, 1\rbrace}{2} \int_0^{T_\varepsilon} \norm{\nabla v_{12}}{\Lnorm{2}}^2 \,dt \\
			& ~~ \leq e^{T_\varepsilon \mathcal H(C_\delta, \varepsilon, C_0M_0, C_1M_1)} \bigl( \delta \norm{\nabla v_{12}'}{\Lnorm{2}(0,T_\varepsilon;L^2(\Omega_h\times 2\mathbb T))}^2 \\
			& ~~~~ + T_\varepsilon \norm{v_{12}'}{\Lnorm{\infty}(0,T_\varepsilon; L^2(\Omega_h\times 2\mathbb T))}^2  + T_\varepsilon \norm{\xi_{12}'}{\Lnorm{\infty}(0,T_\varepsilon; L^2(\Omega_h\times 2\mathbb T))}^2 \\
			& ~~~~ + T_\varepsilon^{1/2} \norm{\nabla v_2}{\Lnorm{2}(0,T_\varepsilon; H^2(\Omega_h\times 2\mathbb T))}\mathcal H(C_\delta, \varepsilon, C_0M_0, C_1M_1) \\
			& ~~~~ \times \norm{\xi_{12}'}{\Lnorm{\infty}(0,T_\varepsilon; L^2(\Omega_h\times 2\mathbb T))}^2 \bigr).
		\end{aligned}
	\end{equation}
	Then for fixed $ \varepsilon \in (0,\varepsilon_0) $, by first choosing $ \delta $ small enough, and then $ T_\varepsilon $ small enough, \eqref{local-101} and \eqref{local-102} yield  inequality \eqref{contraction} with $ q = \dfrac 1 2 $. This finishes the proof of $ (2) $ and completes the proof of Proposition \ref{prop:local-existence}.	
\end{proof}

%%%%%%%%%%%%%%%%%%%%%%
\subsection{Uniform stability}\label{subsec:uniform-stability}
In this section, we will show that the existence time $ T_\varepsilon $ of the strong solutions constructed in Propositive \ref{prop:local-existence} is uniform in $ \varepsilon $ provided $ \varepsilon \in (0, \varepsilon_1) $ for some $ \varepsilon_1 \in (0, \varepsilon_0) $. In order to show this, it suffices to show that there is a positive constant $ \varepsilon_1 $ such that the $ H^2 $-norms of $ (\xi, v) $ remain bounded in a time interval $ (0,T) $ with $ T \in (0,\infty) $  independent of $ \varepsilon $, provided $ \varepsilon \in (0, \varepsilon_1) $ with $ \varepsilon_1 \in (0,1) $ small enough. We perform a continuity argument in the following.

Recall that we are given initial data $ (\xi_0, v_0) \in H^2(\Omega_h\times 2\mathbb T) $, and a positive constant $ M \in (0,\infty) $ satisfying
\begin{equation*}\tag{\ref{initial-data}}
%	\begin{gathered}
		\dfrac{1}{2}\norm{v_0}{H^2}^2 + \dfrac{c^2}{\gamma-1} \norm{\xi_0}{H^2}^2 < M.
%	\end{gathered}
\end{equation*}
Denote by
\begin{equation}\label{25july2018-01}
	M_0 : = \dfrac{8(\gamma-1)}{c^2} M, ~ M_1 : = 4M.
\end{equation}
Then it is obvious that
\begin{equation*}
		\norm{\xi_0}{H^2}^2 < M_0 , ~ \norm{v_0}{H^2}^2 < M_1.
\end{equation*}
From Proposition \ref{prop:local-existence}, there is a strong solution $ (\xi, v) $ satisfying \eqref{le-regularity} and \eqref{le-priori-estimate} in the time interval $ [0, T_\varepsilon] $, for some $ T_\varepsilon \in (0,\infty) $. Then for any $  t \in [0, T_\varepsilon] $, we have
\begin{equation}\label{April302019}
	\norm{\xi(t)}{H^2}^2 \leq C_0 M_0 = \dfrac{8(\gamma-1)}{c^2} C_0 M, ~\norm{v(t)}{H^2}^2 \leq C_1 M_1 = 4 C_1 M,
\end{equation}
where $ C_0, C_1 $ are given in Proposition \ref{prop:local-existence}. We remind readers that $ T_\varepsilon $ depends on $ M_0, M_1 $ and $ \varepsilon $. On the other hand, consider $ \varepsilon_1 $ satisfying
\begin{equation*}
	\begin{gathered}
		\sup_{0\leq t\leq T_{\varepsilon}} \max\lbrace 1, \alpha \rbrace \varepsilon_1  \norm{\xi(t)}{H^2} \leq \max\lbrace 1, \alpha \rbrace  \varepsilon_1 \times \bigl(C_0M_0 \bigr)^{1/2} \leq \log 2, \\
		\sup_{0\leq t\leq T_{\varepsilon}} \varepsilon_1 \mathcal H_1( \norm{\xi(t)}{H^2},\norm{v(t)}{H^2} ) \leq \varepsilon_1 \mathcal H_1 ( C_0 M_0, C_1M_1) \leq \dfrac{\min\lbrace \mu, \lambda, 1 \rbrace c_1}{8},
	\end{gathered}
\end{equation*}
where $ \mathcal H_1 $ is as in the right-hand side of \eqref{ue:spatial-derivative}.
Then
the {\it a priori} estimate in \eqref{ue:spatial-derivative} implies that, for $ \varepsilon \in (0, \varepsilon_1) $ and $ t \in [0,T_\varepsilon] $,
\begin{align*}
	&  \dfrac{1}{2} \norm{v(t)}{H^2}^2 +  \dfrac{c^2}{4(\gamma-1)} \norm{\xi(t)}{H^2}^2  + \dfrac{\min\lbrace \mu,\lambda,1 \rbrace c_1}{2} \int_0^t \norm{\nabla v(s)}{H^2}^2 \,ds \\
	& ~~ \leq \dfrac{1}{2} \norm{v_0}{H^2}^2 +  \dfrac{c^2}{\gamma-1}  \norm{\xi_0}{H^2}^2  + \delta \int_0^t \norm{\nabla v(s)}{H^2}^2 \,ds \\
	& ~~~~  + \dfrac{\min\lbrace \mu,\lambda,1\rbrace c_1}{8} \int_0^t  \norm{\nabla v(s)}{H^2}^2 \,ds  + C_\delta \int_0^t \mathcal H_2(C_0M_0, C_1M_1) \,ds.
\end{align*}
Hence after substituting \eqref{initial-data} and choosing
\begin{equation}\label{local-existence-time}
	\delta = \delta' := \dfrac{\min\lbrace \mu,\lambda,1\rbrace c_1}{8}, ~ \text{and} ~ T : = \dfrac{M}{C_{\delta'} \mathcal H_2 (C_0M_0,C_1M_1)},
\end{equation}
it follows
\begin{equation}\label{local-uniform-estimate}
\begin{aligned}
	&  \dfrac{1}{2} \norm{v(t)}{H^2}^2 +  \dfrac{c^2}{4(\gamma-1)} \norm{\xi(t)}{H^2}^2 + \dfrac{\min\lbrace \mu,\lambda,1 \rbrace c_1}{4} \int_0^t \norm{\nabla v(s)}{H^2}^2 \,ds \\
	& ~~~~ \leq M + C_{\delta'} \mathcal H_2 (C_0M_0,C_1M_1) t < 2M,
\end{aligned}
\end{equation}
for \begin{equation*}
t \in [0, \min\lbrace T_\varepsilon, T\rbrace ]. % \leq t \leq  T : = \dfrac{M_2}{C_{\delta'} \mathcal H_2 (C_0M_0,C_1M_1)}.
\end{equation*}
Notice that $ T $ is independent of $ \varepsilon $ once $ \varepsilon \in(0, \varepsilon_1) $.  Also from \eqref{local-uniform-estimate}, we have
\begin{equation}\label{25july2018-02}
	\begin{gathered}
		\sup_{0\leq t \leq \min\lbrace T_\varepsilon, T\rbrace }\norm{\xi(t)}{H^2}^2 < \dfrac{8(\gamma-1)}{c^2} M = M_0 ,\\
	 	\sup_{0\leq t \leq \min\lbrace T_\varepsilon, T\rbrace } \norm{v(t)}{H^2}^2 < 4 M = M_1, \\
	 	\int_0^{\min\lbrace T_\varepsilon, T \rbrace} \norm{\nabla v(t)}{\Hnorm{2}}^2 \,dt \leq \dfrac{8M}{c_1 \min\lbrace \mu, \lambda, 1\rbrace} = : M_2.
	 \end{gathered}
\end{equation}
Therefore, supposed that  $ T_\varepsilon < T $, we will have $$ \norm{\xi(T_\varepsilon)}{H^2}^2 < M_0 ,   \norm{v(T_\varepsilon)}{H^2}^2 < M_1. $$
Then by setting $ (\xi(T_\varepsilon), v(T_\varepsilon) ) $ as the new initial data, applying Proposition \ref{prop:local-existence} again,  estimate \eqref{April302019} holds in $ [0, 2T_\varepsilon] $. Thus the arguments between \eqref{April302019} and \eqref{local-uniform-estimate} hold with $ T_\varepsilon $ replaced by $ 2 T_\varepsilon $, without needing to choose the smallness of $ \varepsilon_1 $ and $ T $. Then  estimate \eqref{local-uniform-estimate} holds in the time interval $ [0, \min\lbrace 2T_\varepsilon, T\rbrace ] $. Repeat such arguments $ n $ times, $ n \in \mathbb Z^+ $, until $ n T_\varepsilon \geq T $. This extends the existence time of the local strong solution $ (\xi, v) $ to \eqref{eq:perturbation} to $ T> 0 $, which is independent of $ \varepsilon $ provided $ \varepsilon \in (0,\varepsilon_1) $ with $ \varepsilon_1 $ given as above. Consequently, we conclude that:
\begin{prop}\label{prop:uniform-stability}
	Consider $ (\xi_0,v_0) \in H^2(\Omega_h\times 2\mathbb T) $  with the compatibility conditions in \eqref{cmpbc}, bounded as in \eqref{initial-data}. Then there is a positive constant $ \varepsilon_1 \in (0,1) $ such that for any $ \varepsilon \in (0,\varepsilon_1) $, there is a unique strong solution to \eqref{eq:perturbation} in the time interval $ [0,T] $ satisfying the regularity \eqref{le-regularity}. Here $ T \in (0,\infty) $ is a positive time which is independent of $ \varepsilon$. Moreover, there are positive constants $ M_0, M_1, M_2 \in (0,\infty) $, given in \eqref{25july2018-01} and \eqref{25july2018-02} and independent on $ \varepsilon $, such that
	\begin{equation}\label{uniform-stability}
	\begin{gathered}
		\norm{\xi}{L^\infty(0,T;\Hnorm{2}(\Omega_h\times 2\mathbb T))}^2 < M_0, ~ \norm{v}{L^\infty(0,T;\Hnorm{2}(\Omega_h\times 2\mathbb T))}^2 < M_1,\\
		\norm{\nabla v}{L^2(0,T;\Hnorm{2}(\Omega_h\times 2\mathbb T))}^2 < M_2.
	\end{gathered}
	\end{equation}
%	where $ M_1, M_2, M_3 $ are .
\end{prop}

\section{The limit equations}\label{sec:limit-equation}
In this section, we will identify the limit equations of \eqref{eq:perturbation} as $ \varepsilon \rightarrow 0^+ $ in the distribution sense.
%To avoid confusion, we write the variables for \eqref{eq:perturbation} with the subscript $ \varepsilon > 0 $. That is $ (\xi_\varepsilon , v_\varepsilon) $. And the limit variables are denoted as $ (\xi_p, v_p) $.
Recall that,
\begin{equation*}\tag{\ref{eq:perturbation}}
	\begin{cases}
		\dt \xi_\varepsilon + v_\varepsilon\cdot\nablah \xi_\varepsilon + \dfrac{\gamma-1}{\varepsilon} ( \dvh v_\varepsilon+ \dz w_\varepsilon ) = 0 &\text{in} ~ \Omega_h \times 2 \mathbb T,\\
		\dt v_\varepsilon + v_\varepsilon \cdot \nablah v_\varepsilon + w_\varepsilon \dz v_\varepsilon + \dfrac{c^2 e^{\varepsilon \xi_\varepsilon}}{\varepsilon} \nablah \xi_\varepsilon \\
		~~~~ = c_1 e^{-\varepsilon \alpha\xi_\varepsilon} \bigl( \mu \deltah v_\varepsilon + \lambda \nablah \dvh v_\varepsilon + \partial_{zz} v_\varepsilon) & \text{in} ~ \Omega_h\times 2\mathbb T, \\
		\dz \xi_\varepsilon = 0 & \text{in} ~ \Omega_h\times 2\mathbb T, \\
		w_\varepsilon = 0 & \text{on} ~ \Omega_h \times \mathbb Z,
	\end{cases}
\end{equation*}
where the initial data are taken as $ (\xi_\varepsilon, v_\varepsilon)\bigr|_{t=0} = (\xi_0, v_0) \in H^2(\Omega_h\times 2\mathbb T) $ satisfying the compatibility conditions in \eqref{cmpbc}, bounded as in \eqref{initial-data}.
We first list the uniform estimates obtained in the previous section.
In fact, from \eqref{uniform-stability} and \eqref{ue:2001}--\eqref{ue:204}, the following estimates hold uniformly in $ \varepsilon \in (0,\varepsilon_1) $ for some constant $ C, T \in (0,\infty) $ independent of $ \varepsilon $ and depending only on the initial data,
\begin{equation} \label{ue:total}
	\begin{gathered}
		\sup_{0 \leq t \leq T} \bigl\lbrace \norm{v_\varepsilon(t)}{H^2} + \norm{\xi_\varepsilon(t)}{H^2}+ \norm{\dt (e^{\varepsilon\xi_\varepsilon})(t)}{H^1} + \norm{\dt \mathcal P_\sigma v_\varepsilon(t)}{\Lnorm2} \bigr\rbrace \\
		+ \norm{\nabla v_\varepsilon}{L^2(0,T;H^2)} + \norm{\dt \mathcal P_\sigma v_\varepsilon}{L^2(0,T;H^1)} < C.
	\end{gathered}
\end{equation}
Then as $ \varepsilon \rightarrow 0^+ $, there exist $ \xi_p, v_p, \zeta_p, v_{p,\sigma} $ with
\begin{equation}\label{limit:functions}
	\begin{gathered}
		\xi_p \in L^\infty(0,T;H^2), ~~ v_p \in L^\infty(0,T;H^2) \cap L^2(0,T;H^3), \\
		\zeta_p \in L^\infty(0,T;H^1), ~~ v_{p,\sigma} \in L^\infty(0,T;H^2_\sigma) \cap L^2(0,T;H^3_\sigma),\\
		\dt v_{p,\sigma} \in L^\infty(0,T;L^2_\sigma) \cap L^2(0,T;H^1_\sigma),
	\end{gathered}
\end{equation}
such that, % as $ \varepsilon \rightarrow 0^+ $,
\begin{equation}\label{limit:001}
	\begin{aligned}
		& \xi_{\varepsilon} \buildrel\ast\over\rightharpoonup \xi_p ~~ \text{weak-$\ast$ in} ~ L^\infty(0,T;H^2), \\
		& v_\varepsilon ~ (\text{and} ~ \mathcal P_\sigma v_\varepsilon) \buildrel\ast\over\rightharpoonup v_p ~ (\text{respectively} ~ v_{p,\sigma}) \\
		& ~~~~ ~~~~ \text{weak-$\ast$ in} ~ L^\infty(0,T;H^2) ~ (\text{respectively} ~ L^\infty(0,T;H^2_\sigma)),\\
		& v_\varepsilon~ (\text{and} ~ \mathcal P_\sigma v_\varepsilon) \rightharpoonup v_p~ (\text{respectively} ~ v_{p,\sigma}) \\
	 	& ~~~~ ~~~~ \text{weakly in} ~ L^2(0,T;H^3) ~ (\text{respectively} ~ L^2(0,T;H^3_\sigma)), \\
		& \dt(e^{\varepsilon \xi_\varepsilon}) \buildrel\ast\over\rightharpoonup \zeta_p ~~ \text{weak-$\ast$ in} ~ L^\infty(0,T;H^1), \\
		& \dt \mathcal P_\sigma v_\varepsilon \buildrel\ast\over\rightharpoonup \dt v_{p,\sigma} ~~ \text{weak-$\ast$ in} ~ L^\infty(0,T;L^2_\sigma), \\
		& \dt \mathcal P_\sigma v_\varepsilon \rightharpoonup \dt v_{p,\sigma}  ~~ \text{weakly in} ~ L^2(0,T;H^1_\sigma).
	\end{aligned}
\end{equation}
Also applying the Aubin compactness lemma (see, e.g., \cite[Theorem 2.1]{Temam1984} and \cite{Simon1986,Chen2012}) yields, as $ \varepsilon \rightarrow 0^+ $,
\begin{equation}\label{limit:002}
%	\begin{gathered}
		\mathcal P_\sigma v_\varepsilon \rightarrow v_{p,\sigma} ~~ \text{in} ~ L^2(0,T;H^2_{\sigma,loc})\cap C([0,T];H^1_{\sigma,loc}).
%	\end{gathered}
\end{equation}
Moreover, from \eqref{ue:total}, we have
\begin{equation}\label{limit:003}
%	\begin{gathered}
	\varepsilon \norm{\xi_\varepsilon}{H^2} < \varepsilon C \rightarrow 0 ~~~ \text{as} ~ \varepsilon \rightarrow 0^+.
%	\end{gathered}
\end{equation}
Thus we have, as $ \varepsilon \rightarrow 0^+ $,
\begin{equation}\label{limit:004}
\begin{gathered}
	e^{\varepsilon \xi_\varepsilon} \rightarrow 1 ~~ \text{in} ~ L^\infty(0,T;H^2), \\
	~~ \text{and} ~~ \dt(e^{\varepsilon \xi_\varepsilon}) \rightarrow 0 = \zeta_p ~~ \text{in the sense of distribution}.
\end{gathered}
\end{equation}
On the other hand, recall that the vertical velocity is represented by % \eqref{def:vertical-velocity},
\begin{equation*}\tag{\ref{def:vertical-velocity}}
	w_\varepsilon = - \int_0^z \biggl( \dfrac{\varepsilon}{\gamma-1} \widetilde v_\varepsilon \cdot\nablah \xi_\varepsilon + \dvh \widetilde v_\varepsilon \biggr) \,dz.
\end{equation*}
Therefore, a direct calculation yields that
%\begin{equation}\label{limit:005}
%%	\begin{aligned}
%	\norm{w_\varepsilon}{H^1} + \norm{\dz w_\varepsilon}{H^1} \lesssim \varepsilon \norm{v_\varepsilon}{H^2} \norm{\xi_\varepsilon}{H^2} + \norm{v_\varepsilon}{H^2} < C.
%%	\end{aligned}
%\end{equation}
%That is
$ w_\varepsilon, \dz w_\varepsilon \in L^\infty(0,T;H^1)$ and we have the following uniform bounds:
\begin{equation}
\label{ue:w}
\begin{aligned}
&  \norm{w_\varepsilon}{L^\infty(0,T;H^1)} + \norm{\dz w_\varepsilon}{L^\infty(0,T;H^1)} \\
& ~~~~ \lesssim \varepsilon \norm{v_\varepsilon}{L^\infty(0,T;H^2)} \norm{\xi_\varepsilon}{L^\infty(0,T;H^2)} + \norm{v_\varepsilon}{L^\infty(0,T;H^2)} < C. \end{aligned} \end{equation}
Therefore, there exists \begin{equation}\label{limit:vertical-velocity} w_p \in L^\infty(0,T;H^1) ~~~~ \text{with} ~ \dz w_p \in L^\infty(0,T;H^1) \end{equation} such that,
\begin{equation}\label{limit:006}
	w_\varepsilon ~ (\text{and} ~ \dz w_\varepsilon) \buildrel\ast\over\rightharpoonup w_p ~(\text{respectively} ~ \dz w_p) ~~ \text{weak-$\ast$ in} ~ L^\infty(0,T;H^1),
\end{equation}
as $ \varepsilon \rightarrow 0^+ $.
Notice, after applying the trace theorem,
$ \lbrace w_\varepsilon|_{z\in \mathbb Z}\rbrace_{\varepsilon \in (0, \varepsilon_1)} \subset L^\infty(0,T;H^{1/2}(\Omega_h)) $ and is uniformly bounded. Then it follows from \eqref{def:vertical-velocity} that  $ w_\varepsilon|_{z \in \mathbb Z} = 0 $ and hence $ w_p|_{z\in \mathbb Z} = 0 $. Moreover, after multiplying \subeqref{eq:perturbation}{1} with $ \varepsilon e^{\varepsilon\xi_\varepsilon} $, one has
\begin{equation}\label{limit:10001}
	\dt e^{\varepsilon\xi_\varepsilon} + \varepsilon e^{\varepsilon \xi_\varepsilon} v_\varepsilon \cdot \nablah \xi_\varepsilon + (\gamma-1)e^{\varepsilon \xi_\varepsilon} ( \dvh v_\varepsilon + \dz w_\varepsilon) = 0.
\end{equation}
Then, \eqref{ue:total} and \eqref{limit:004} imply that as $ \varepsilon \rightarrow 0^+ $, \eqref{limit:10001} converges to
\begin{equation*}
	\dvh v_p + \dz w_p = 0 ~~~ \text{in the sense of distribution}.
\end{equation*}
Consequently, we have shown that
\begin{equation}\label{limit:continuous-eq}
	\dvh v_p + \dz w_p = 0, ~~ \text{and} ~~ w_p\bigr|_{z\in \mathbb Z} = 0.
\end{equation}

Next, we will identify the limit equation of the momentum equation \subeqref{eq:perturbation}{2}. First,  \eqref{limit:functions} and \eqref{limit:continuous-eq} imply that $ v_p \in L^\infty(0,T;H^2_\sigma)\cap L^2(0,T;H^3_\sigma) $. We first show that $ v_{p,\sigma} \equiv v_p $. Let $ u \in C_0^\infty(0,T; C^\infty_\sigma(\Omega_h\times 2 \mathbb T, \mathbb R^2)) $. Then we have, from \eqref{limit:001},
\begin{align*}
	& \int_0^T \int u \cdot v_p \idx \,dt  =\lim\limits_{\varepsilon \rightarrow 0^+} \int_0^T \int u \cdot v_\varepsilon \idx \,dt \\
	& ~~~~ = \lim\limits_{\varepsilon \rightarrow 0^+} \int_0^T  \int u \cdot \mathcal P_\sigma v_\varepsilon \idx \,dt =\int_0^T \int u \cdot v_{p,\sigma}\idx \,dt,
\end{align*}
which shows that $ v_{p,\sigma} \equiv v_p $. Then \eqref{limit:002} can be written as
\begin{equation}\label{limit:007}
	\mathcal P_\sigma v_\varepsilon \rightarrow v_p  ~~ \text{in} ~ L^2(0,T;H^2_{\sigma,loc})\cap C([0,T];H^1_{\sigma,loc}),
\end{equation}
as $ \varepsilon \rightarrow 0^+ $.
In particular, $ v_p(t=0) = \mathcal P_\sigma v_0 $.
Moreover, from \eqref{limit:001}, since $ v_{p,\sigma} \equiv v_p $, we have the following: as $ \varepsilon \rightarrow 0^+ $,
\begin{equation}\label{limit:008}
	\begin{gathered}
	\mathcal P_\tau v_\varepsilon = v_\varepsilon - \mathcal P_\sigma v_\varepsilon \buildrel\ast\over\rightharpoonup 0 ~~ \text{weak-$\ast$ in} ~ L^\infty(0,T;H^2_\tau),\\
	\mathcal P_\tau v_\varepsilon = v_\varepsilon - \mathcal P_\sigma v_\varepsilon \rightharpoonup 0 ~~ \text{weakly in} ~ L^2(0,T;H^3_\tau).
	\end{gathered}
\end{equation}
%In order to proceed, we will analysis the acoustic wave equations associated with \eqref{eq:perturbation}. That is, after applying the operator $ \mathcal P_\tau $ to \subeqref{eq:perturbation}{2}, together with \subeqref{eq:perturbation}{1}, we have
%\begin{equation}\label{limit:acoustic-wave}
%	\begin{cases}
%		\dt \xi_\varepsilon + \dfrac{\gamma-1}{\varepsilon} \dvh \mathcal P_\tau v_\varepsilon  = - \dfrac{\gamma-1}{\varepsilon} ( \dvh \mathcal P_\sigma v_\varepsilon + \dz w_\varepsilon) + G_1,\\
%		\dt \mathcal P_\tau v_\varepsilon + \dfrac{c^2 e^{\varepsilon \xi_\varepsilon}}{\varepsilon} \nablah \xi_\varepsilon = \mathcal P_\tau G_2,
%	\end{cases}
%\end{equation}
%where
%\begin{equation}
%	G_1:= -  v_\varepsilon\cdot\nablah \xi_\varepsilon , G_2 :=  c_1 e^{-\varepsilon \alpha\xi_\varepsilon} \bigl( \mu \deltah v_\varepsilon + \lambda \nablah \dvh v_\varepsilon + \partial_{zz} v_\varepsilon) - v_\varepsilon \cdot \nablah v_\varepsilon - w_\varepsilon \dz v_\varepsilon.
%\end{equation}
Also, denote $ \mathcal P_\tau v_\varepsilon = \nablah \psi_\varepsilon $ with $ \psi_\varepsilon $ defined by the following elliptic problem (see, e.g., \eqref{def:projection-poission-equation}),
\begin{equation}\label{limit:potential-function}
\begin{gathered}
	\deltah \psi_\varepsilon = \int_0^1 \dvh v_\varepsilon \,dz  ~~ \text{in} ~ \Omega_h , \\
	\begin{cases}
	\lim\limits_{\abs{(x,y)}{} \rightarrow \infty} \psi_\varepsilon = 0 & \text{in the case when} ~ \Omega_h = \mathbb R^2, \\
	\int_{\Omega_h} \psi_\varepsilon \idxh = 0 & \text{in the case when} ~ \Omega_h = \mathbb T^2.	
	\end{cases}
\end{gathered}
\end{equation}
Then $ \mathcal P_\tau \dt v_\varepsilon = \nablah \dt \psi_\varepsilon $ and we have the identity:
\begin{align*}
	& \mathcal P_\tau v_\varepsilon \cdot \nablah \mathcal P_\tau v_\varepsilon + w_\varepsilon \dz \mathcal P_\tau v_\varepsilon  = \nablah \psi_\varepsilon \cdot \nablah \nablah \psi_\varepsilon + w_\varepsilon \dz \nablah \psi_\varepsilon \\
	& ~~~~ = \dfrac{1}{2} \nablah \abs{\nablah \psi_\varepsilon}{2}.
\end{align*}
Therefore, one has
\begin{align*}
	& v_\varepsilon \cdot\nablah v_\varepsilon + w_\varepsilon \dz v_\varepsilon = \mathcal P_\sigma v_\varepsilon \cdot \nablah v_\varepsilon + \mathcal P_\tau v_\varepsilon \cdot \nablah \mathcal P_\sigma v_\varepsilon + w_\varepsilon \dz \mathcal P_\sigma v_\varepsilon \\
	& ~~~~ + \mathcal P_\tau v_\varepsilon \cdot \nablah \mathcal P_\tau v_\varepsilon + w_\varepsilon \dz \mathcal P_\tau v_\varepsilon = \mathcal P_\sigma v_\varepsilon \cdot \nablah v_\varepsilon + \mathcal P_\tau v_\varepsilon \cdot \nablah \mathcal P_\sigma v_\varepsilon \\
	& ~~~~ + w_\varepsilon \dz \mathcal P_\sigma v_\varepsilon + \dfrac{1}{2} \nablah \abs{\nablah \psi_\varepsilon}{2} .
\end{align*}
%
%We will show that the momentum equation \subeqref{eq:perturbation}{2} will converge to the primitive equation in the sense of distribution. In fact,
Then
\subeqref{eq:perturbation}{2} can be written as
\begin{equation}\label{perturbation-eq-in-decomposed}
\begin{aligned}
	& \dt \mathcal P_\sigma v_\varepsilon + \mathcal P_\sigma v_\varepsilon \cdot \nablah v_\varepsilon + \mathcal P_\tau v_\varepsilon \cdot \nablah \mathcal P_\sigma v_\varepsilon + w_\varepsilon \dz \mathcal P_\sigma v_\varepsilon \\
	& ~~~~ - c_1 e^{-\varepsilon \alpha \xi_\varepsilon}( \mu \deltah v_\varepsilon + \lambda \nablah \dvh v_\varepsilon  + \partial_{zz} v_\varepsilon ) + c^2 \bigl( \dfrac{e^{\varepsilon\xi_\varepsilon} - 1}{\varepsilon} - \xi_\varepsilon \bigr)\nablah \xi_h \\
	&  = - \bigl( \nablah \dt \psi_\varepsilon + \dfrac{1}{2} \nablah \abs{\nablah \psi_\varepsilon}{2} + \dfrac{c^2}{\varepsilon} \nablah \xi_\varepsilon + \dfrac{c^2}{2}\nablah\abs{\xi_\varepsilon}{2} \bigr).
\end{aligned}
\end{equation}
Then \eqref{ue:total}, \eqref{ue:w} and \eqref{perturbation-eq-in-decomposed} imply that
\begin{align*}
& \normh{\nablah \dt \psi_\varepsilon + \dfrac{1}{2} \nablah \abs{\nablah \psi_\varepsilon}{2}+ \dfrac{c^2}{\varepsilon} \nablah \xi_\varepsilon + \dfrac{c^2}{2}\nablah\abs{\xi_\varepsilon}{2}}{L^\infty(0,T;L^2)\cap L^2(0,T;H^1)}\\
& ~~~~ = \normh{\text{left-hand side of \eqref{perturbation-eq-in-decomposed}}}{L^\infty(0,T;L^2)\cap L^2(0,T;H^1)} < C,
\end{align*}
where $ C $ is independent of $ \varepsilon $. Hence $ \exists  P \in L^\infty (0,T;H^1(\Omega_h)) \cap L^2(0,T;H^2(\Omega_h)) $ such that
\begin{equation}
\label{limit:pressure}
\begin{gathered}
\nablah \dt \psi_\varepsilon + \dfrac{1}{2} \nablah \abs{\nablah \psi_\varepsilon}{2}+ \dfrac{c^2}{\varepsilon} \nablah \xi_\varepsilon + \dfrac{c^2}{2}\nablah\abs{\xi_\varepsilon}{2} \\
\buildrel\ast\over\rightharpoonup \nablah P ~~ \text{weak-$\ast$ in} ~ L^\infty(0,T;L^2), \\
\nablah \dt \psi_\varepsilon + \dfrac{1}{2} \nablah \abs{\nablah \psi_\varepsilon}{2}+ \dfrac{c^2}{\varepsilon} \nablah \xi_\varepsilon + \dfrac{c^2}{2}\nablah\abs{\xi_\varepsilon}{2}\\
\rightharpoonup \nablah P ~~ \text{weakly in} ~ L^2(0,T;H^1).
\end{gathered}
\end{equation}
Also, using \eqref{ue:total}, we have
\begin{equation}\label{limit:pressure-2}
\normh{\dfrac{e^{\varepsilon\xi_\varepsilon} - 1}{\varepsilon} - \xi_\varepsilon}{\Lnorm{\infty}} \lesssim \varepsilon \normh{\xi_\varepsilon}{\Hnorm{2}} \rightarrow 0 ~~~~ \text{as} ~ \varepsilon \rightarrow 0^+.
\end{equation}
Then the weak and strong convergences in \eqref{limit:001}, \eqref{limit:004}, \eqref{limit:006}, \eqref{limit:007}, \eqref{limit:008}, \eqref{limit:pressure}, and \eqref{limit:pressure-2} yield that, as $ \varepsilon \rightarrow 0^+ $, \eqref{perturbation-eq-in-decomposed} converges to
\begin{equation}\label{limit:momentum-eq}
\dt v_p + v\cdot \nablah v_p + w_p\dz v_p + \nablah P = c_1 \bigl(\mu \deltah v_p + \lambda\nablah \dvh v_p +\partial_{zz}v_p \bigr),
\end{equation}
in the sense of distribution,
where $ w_p $ is determined by \eqref{limit:continuous-eq}, and $ \dz P = 0 $.
% and $ P $ is determined by
%\begin{equation}\label{limit:pressure-eq}
%\begin{aligned}
%	& - \deltah P = \int_0^1 \dvh (v_p\cdot\nablah v_p + w_p\dz v_p) \,dz  ~~ \text{in} ~ \Omega_h,\\
%	& \begin{cases}
%	\lim\limits_{\abs{(x,y)}{} \rightarrow \infty} P  = 0 & \text{in the case when} ~ \Omega_h = \mathbb R^2, \\
%	\int_{\Omega_h} P \idxh = 0 & \text{in the case when} ~ \Omega_h = \mathbb T^2.	
%	\end{cases}
%\end{aligned}
%\end{equation}

To conclude, we have shown the following:
\begin{prop}\label{prop:convergence-eq-distribution}
Let $ (\xi_0, v_0) \in H^2(\Omega_h\times 2\mathbb T) $ satisfying the compatibility conditions in \eqref{cmpbc} and $ \varepsilon \in (0,\varepsilon_1) $ with $ \varepsilon_1 $ given in Proposition \ref{prop:uniform-stability}. Then the unique strong solutions $\lbrace (\xi_{\varepsilon}, v_\varepsilon) \rbrace_{\varepsilon \in (0,\varepsilon_1)} $ obtained in Proposition \ref{prop:uniform-stability} in the time interval $ [0,T] $, $ T \in (0,\infty) $, and $ w_\varepsilon $ given as in \eqref{def:vertical-velocity}, converge to $ (\xi_p, v_p) $ and $ w_p $, with the convergences in \eqref{limit:001}, \eqref{limit:004}, \eqref{limit:006}, \eqref{limit:007} and \eqref{limit:008}, as  $ \varepsilon \rightarrow 0^+ $.
$ \xi_p, v_p, w_p $ satisfy the regularity \eqref{limit:functions} and \eqref{limit:vertical-velocity}.
 Also, $ (v_p, w_p) $ satisfies the primitive equations \eqref{eq:primitive-eq}
with the initial data $ v_p|_{t=0} = v_{p,0} = \mathcal P_\sigma v_0 \in H^2_\sigma(\Omega_h\times 2\mathbb T) $.
\end{prop}

In the rest of this section, we will discuss the converging behaviors of $ \xi_\varepsilon $ and $ \mathcal P_\tau v_\varepsilon $ as $ \varepsilon \rightarrow 0^+ $ in some strong sense. We will consider such a problem in two cases: $ \Omega_h = \mathbb R^2 $ and $ \Omega_h = \mathbb T^2 $.
In fact, we will investigate the following acoustic wave equations associated with \eqref{eq:perturbation}:
\begin{equation}\label{acoustic-wave-001}
	\begin{cases}
		\dt \xi_\varepsilon + \dfrac{\gamma-1}{\varepsilon} \dvh \mathcal P_\tau v_\varepsilon  = -  v_\varepsilon\cdot\nablah \xi_\varepsilon \\
		~~~~  - \dfrac{\gamma-1}{\varepsilon} ( \dvh \mathcal P_\sigma v_\varepsilon + \dz w_\varepsilon)  & \text{in} ~ \Omega_h \times 2 \mathbb T,\\
		\dt \mathcal P_\tau v_\varepsilon + \dfrac{c^2}{\varepsilon} \nablah \xi_\varepsilon = \dfrac{c^2(1-e^{\varepsilon \xi_\varepsilon})}{\varepsilon} \nablah \xi_\varepsilon  - \dt \mathcal P_\sigma v_\varepsilon \\
		~~~~ - v_\varepsilon \cdot \nablah v_\varepsilon -  w_\varepsilon \dz v_\varepsilon + c_1 e^{-\varepsilon \alpha\xi_\varepsilon} \bigl( \mu \deltah v_\varepsilon \\
		~~~~ + \lambda \nablah \dvh v_\varepsilon + \partial_{zz} v_\varepsilon) & \text{in} ~ \Omega_h \times 2 \mathbb T,\\
		\dz \xi_\varepsilon = 0 & \text{in} ~ \Omega_h \times 2 \mathbb T.%\\
%		w_\varepsilon = 0 & \text{on} ~ \Omega_h \times \mathbb Z.
	\end{cases}
\end{equation}
To study the acoustic wave,
notice that, $ \mathcal P_\tau v_\varepsilon = \nablah \psi_\varepsilon $ is a function independent of the $ z $-variable, where $ \psi_\varepsilon $ is defined in \eqref{limit:potential-function}. After averaging \subeqref{acoustic-wave-001}{1} in the $ z $-variable and applying projection operator $ \mathcal P_\tau $ to \subeqref{acoustic-wave-001}{2}, system \eqref{acoustic-wave-001} is reduced to,
\begin{equation}\label{acoustic-wave}
	\begin{cases}
		\dt \xi_\varepsilon + \dfrac{\gamma-1}{\varepsilon} \dvh \nablah \psi_\varepsilon = G_1 & \text{in} ~ \Omega_h, \\
		\dt \nablah \psi_\varepsilon + \dfrac{c^2}{\varepsilon} \nablah \xi_\varepsilon = G_2 + G_3 & \text{in} ~ \Omega_h,
	\end{cases}
\end{equation}
where we have used $ \nablah \psi_\varepsilon $ (defined in \eqref{limit:potential-function}) to represent $ \mathcal P_\tau v_\varepsilon $, and
\begin{equation*}
	\begin{aligned}
		& G_1 : = - \overline v_\varepsilon \cdot \nablah \xi_\varepsilon, \\
		& G_2 : = \dfrac{c^2(1-e^{\varepsilon \xi_\varepsilon})}{\varepsilon} \nablah \xi_\varepsilon + \mathcal P_\tau \bigl( c_1 e^{-\varepsilon \alpha\xi_\varepsilon} \bigl( \mu \deltah  v_\varepsilon + \lambda \nablah \dvh  v_\varepsilon + \partial_{zz} v_\varepsilon ) \bigr), \\
		& G_3: =  - \mathcal P_\tau \bigl( v_\varepsilon \cdot \nablah v_\varepsilon +  w_\varepsilon \dz v_\varepsilon \bigr).
	\end{aligned}
\end{equation*}
Therefore, the acoustic wave system is only two dimensional.
From \eqref{ue:206}, \eqref{ue:205}, \eqref{ue:total}, %\eqref{limit:005},
we have
\begin{equation}\label{aw:bound-of-source}
\begin{gathered}
	\normh{G_1}{L^\infty(0,T;H^1(\Omega_h))} + \normh{G_2}{L^\infty(0,T;L^2(\Omega_h))} + \normh{G_2}{L^2(0,T;H^1(\Omega_h))}\\
	+ \normh{G_3}{L^\infty(0,T;L^2(\Omega_h))} + \normh{G_3}{L^2(0,T;H^1(\Omega_h))}< C.
\end{gathered}
\end{equation}
Now denote by the linear operator $ L $ defined on $ (\mathcal D'(\Omega_h))^{3} $ as
\begin{equation}\label{aw:linear-operator}
	L(\mathcal U) : = \biggl( \begin{array}{c}
		(\gamma - 1) \dvh u\\ c^2 \nablah q
	\end{array} \biggr), ~~~~ \text{where} ~~ \mathcal U = \biggl( \begin{array}{c}
		q\\ u
	\end{array}\biggr),
\end{equation}
where $ q \in \mathcal D'(\Omega_h), ~ u \in ( \mathcal D'(\Omega_h))^2 $.
Then \eqref{acoustic-wave} can be written as
\begin{equation}\label{aw:group-form}
	\dt \mathcal U_\varepsilon + \dfrac{1}{\varepsilon} L \mathcal U_\varepsilon = \mathcal G ~~~~ \text{in} ~ \Omega_h,
\end{equation}
where
\begin{equation*}
	\mathcal G : = \biggl( \begin{array}{c}
		G_1 \\G_2 + G_3
	\end{array}\biggr) , ~~ \mathcal U_\varepsilon : = \biggl( \begin{array}{c}
		\xi_\varepsilon \\ \nablah \psi_h
	\end{array}\biggr) = \biggl( \begin{array}{c}
		\xi_\varepsilon \\ \mathcal P_\tau v_\varepsilon
	\end{array}\biggr).
\end{equation*}
Also, the initial data are given as
\begin{equation}\label{aw:initial-data}
	\mathcal U_\varepsilon(t=0) = \mathcal U_0 = \biggl( \begin{array}{c}
		\xi_0 \\ \mathcal P_\tau v_0
	\end{array} \biggr).
\end{equation}
Also, we have
\begin{gather*}
	{\xi_\varepsilon}\in {L^\infty(0,T;H^2(\Omega_h))}, ~ {\mathcal P_\tau v_\varepsilon} \in {L^\infty(0,T;H^2(\Omega_h))\cap L^2(0,T;H^3(\Omega_h))},\\
	\mathcal G \in L^\infty(0,T;L^2(\Omega_h)) \cap L^2(0,T;H^1(\Omega_h)), ~ \mathcal U_0 \in H^2(\Omega_h),
\end{gather*} and from \eqref{initial-data}, \eqref{ue:total} and \eqref{aw:bound-of-source},
\begin{equation}\label{aw:regularity}
	\begin{gathered}
	\normh{\xi_\varepsilon}{L^\infty(0,T;H^2(\Omega_h))} +  \normh{\mathcal P_\tau v_\varepsilon}{L^\infty(0,T;H^2(\Omega_h))\cap L^2(0,T;H^3(\Omega_h))} \\
	+ \normh{\mathcal G}{L^\infty(0,T;L^2(\Omega_h)) \cap L^2(0,T;H^1(\Omega_h))} + \normh{\mathcal U_0}{H^2(\Omega_h)} < C.
	\end{gathered}
\end{equation}
Define the associated solution operator of $ L $ by \begin{equation}\label{aw:def-group-operator-00} \mathcal L: t \mapsto \mathcal L(t) = \exp (-L t). \end{equation}
That is, given $ \mathcal V_0  \in ( \mathcal D')^3 $, $ \mathcal L(t) \mathcal V_0 $ satisfies the linear equation:
\begin{equation}\label{aw:def-group-operator}
	\dt (\mathcal L(t) \mathcal V_0) + L ( \mathcal L(t) \mathcal V_0) = 0.
\end{equation}
Notice that $ \mathcal L $ is linear and $ \mathcal L(t_1 + t_2) = \mathcal L(t_1) \mathcal L(t_2) $ for $ t_1, t_2 \in (0,\infty) $.
Then the solution $ \mathcal U_\varepsilon $ of  \eqref{aw:group-form} can be represented as, using the Duhamel principle,
\begin{equation}\label{aw:representation-sol}
	\mathcal U_\varepsilon(t) = \mathcal L(\dfrac{t}{\varepsilon}) \mathcal U_0 + \int_0^t \mathcal L(\dfrac{t-s}{\varepsilon}) \mathcal G(s) \,ds.
\end{equation}

\subsection{The case when $ \Omega_h = \mathbb R^2 $: dispersion of the acoustic waves}\label{subsec:convergence-whole}
In this section,
we shall establish the strong convergence of $ \mathcal U_\varepsilon $ in the case when $ \Omega_h = \mathbb R^2 $.
To be precise, we will show the following:
\begin{prop}\label{prop:strong-convergence-whole-space}Under the same assumptions as in Proposition \ref{prop:convergence-eq-distribution}, in the case when $ \Omega_h = \mathbb R^2 $, we have
\begin{equation*}\tag{\ref{aw:strong-convergence-01}}
	\xi_\varepsilon \rightarrow 0, ~ \mathcal P_\tau v_\varepsilon \rightarrow 0 ~~ \text{as} ~ \varepsilon \rightarrow 0^+ ~~ \text{in} ~ L^2(0,T;W^{\frac 1 2 , 6}(\mathbb R^2)).
\end{equation*}
%Together with \eqref{aw:regularity}, interpolation inequalities yield the following convergences as $ \varepsilon \rightarrow 0 $,
%\begin{equation}\label{aw:strong-convergence-02}
%	\begin{gathered}
%		\xi_\varepsilon, \mathcal P_\tau v_\varepsilon \rightarrow 0  ~~ \text{in} ~ L^\alpha(0,T;L^{\frac{6\alpha}{8-3\alpha}}(\mathbb R^2))\cap L^\beta(0,T;W^{1,\frac{3\beta}{4}}(\mathbb R^2)) ,\\
%		\mathcal P_\tau v_\varepsilon \rightarrow 0  ~~ \text{in} ~ L^2(0,T; W^{1,\eta}(\mathbb R^2)\cap W^{2,\iota}(\mathbb R^2)\cap W^{3,\varpi}(\mathbb R^2)), \\
%		\text{where} ~~ 2 \leq \alpha \leq 8/3, 2 \leq \beta < \infty, 3/2 \leq \eta < \infty, 6/7 \leq \iota < \infty, 3/5\leq \varpi < 2.
%	\end{gathered}
%\end{equation}
In particular, we have, as $ \varepsilon \rightarrow 0^+  $,
\begin{equation*}
%	\begin{aligned}
	\norm{v_\varepsilon - v_p}{L^2(0,T;L^6_{loc}(\mathbb R^2))} + \norm{\xi_\varepsilon}{L^2(0,T;L^6(\mathbb R^2))} \rightarrow 0.
%	\end{aligned}
\end{equation*}
\end{prop}

In order to show Proposition \ref{prop:strong-convergence-whole-space}, we will use the dispersion estimates of operator $ \mathcal L $ in the whole space $ \mathbb R^2 $. In particular, we will employ the Strichartz inequalities as in \cite{DesjardinsGrenier1999}, to derive the appropriate decay as $ \varepsilon \rightarrow 0^+ $. Indeed, let $ \phi_0 \in (\mathcal D'(\mathbb R^2))^3 $. The following inequality is from \cite[equation (3.4)]{DesjardinsGrenier1999}:
\begin{equation}\label{aw-wsc:strichartz}
	\normh{\mathcal L(t) \phi_0}{L^q (\mathbb R_+; L^p(\mathbb R^2))} \leq C_{p,q} \normh{\phi_0}{H^\eta(\mathbb R^2)},
\end{equation}
provided that the right-hand side is finite, where $ p, q > 2 $ and $ \eta \in (0,\infty) $ satisfying
\begin{equation}\label{aw-wsc:indices}
\dfrac{2}{q} = \dfrac{1}{2} - \dfrac{1}{p} ~~~ \text{and} ~~~ \eta q = 3.
\end{equation}
We will make use of \eqref{aw-wsc:strichartz} to derive the strong dispersion. To begin with, the following dispersion inequalities are analogies to inequalities (3.6) and (3.7) in \cite{DesjardinsGrenier1999}.
\begin{lm}
For $ \phi_0 $ and $ \phi(s) \in (\mathcal D'(\mathbb R^2))^3 $, $ p,q,\eta $ satisfying \eqref{aw-wsc:indices} and $ s \geq 0 $, we have
%\par\noindent\rule{\linewidth}{1pt}
%%%%%%%%%%%% a Lemma %%%%%%%%%%%%%%%%
%Considering
\begin{align}
	& \normh{\mathcal L\bigl( \dfrac{t}{\varepsilon} \bigr) \phi_0}{L^q(0,T;W^{s,p}(\mathbb R^2))} \leq C \varepsilon^{1/q} \normh{\phi_0}{H^{\eta+s}(\mathbb R^2)} \label{aw-wsc:001} , \\
	& \normh{\int_0^t \mathcal L\bigl(\dfrac{t-s}{\varepsilon}\bigr) \phi(s)\,ds}{L^2(0,T;W^{s,p}(\mathbb R^2))} \nonumber \\
	 & ~~~~ ~~~~\leq C(1+T) \varepsilon^{1/q} \normh{\phi}{L^2(0,T;H^{\eta+s}(\mathbb R^2))} \label{aw-wsc:002},
%	& \normh{\int_0^t \mathcal L\bigl(\dfrac{t-s}{\varepsilon}\bigr) \phi(s)\,ds}{L^q(0,T;W^{s,p}(\mathbb R^2))} \nonumber \\
%	 & ~~~~ ~~~~\leq C(1+T) \varepsilon^{1/q} \normh{\phi}{L^q(0,T;H^{\sigma+s}(\mathbb R^2))} \label{aw-wsc:002},
\end{align}
provided the right-hand sides are finite.
\end{lm}
\begin{proof}
Indeed \eqref{aw-wsc:001} is a direct consequence of rescaling the temporal variable and replacing $ \phi_0 $ with $ (I - \deltah)^{s/2} \phi_0 $ in \eqref{aw-wsc:strichartz}. In order to show \eqref{aw-wsc:002}, notice that after applying the Minkowski and H\"older inequalities, one has
\begin{equation*}
%	\normh{\int_0^t \mathcal L \bigl(\dfrac{t-s}{\varepsilon}\bigr) \phi(s) \,ds }{W^{s,p}(\mathbb R^2)}^q \leq C t^{q-1} \int_0^t \normh{\mathcal L \bigl(\dfrac{t-s}{\varepsilon}\bigr) \phi(s)}{W^{s,p}(\mathbb R^2)}^q \,ds.
	\normh{\int_0^t \mathcal L \bigl(\dfrac{t-s}{\varepsilon}\bigr) \phi(s) \,ds }{W^{s,p}(\mathbb R^2)}^2 \leq C t \int_0^t \normh{\mathcal L \bigl(\dfrac{t-s}{\varepsilon}\bigr) \phi(s)}{W^{s,p}(\mathbb R^2)}^2 \,ds.
\end{equation*}
Then by employing the Fubini theorem, it holds
\begin{align*}
	& \int_0^T\normh{\int_0^t \mathcal L \bigl(\dfrac{t-s}{\varepsilon}\bigr) \phi(s) \,ds }{W^{s,p}(\mathbb R^2)}^2 \,dt \\
	& ~~~~ \leq CT \int_0^T \int_s^T \normh{\mathcal L \bigl(\dfrac{t-s}{\varepsilon}\bigr) \phi(s)}{W^{s,p}(\mathbb R^2)}^2 \,dt  \,ds \\
	& ~~~~  \leq CT \int_0^T (T-s)^{1-2/q} \biggl(  \int_s^T \normh{\mathcal L \bigl(\dfrac{t-s}{\varepsilon}\bigr) \phi(s)}{W^{s,p}(\mathbb R^2)}^q \,dt  \biggr)^{2/q} \,ds \\
	& ~~~~ \leq \varepsilon^{2/q} CT^{2-2/q} \int_0^T  \normh{\phi(s)}{H^{\eta+s}(\mathbb R^2)}^2 \,ds = \varepsilon^{2/q} C T^{2-2/q} \normh{\phi}{L^2(0,T;H^{\eta+s}(\mathbb R^2))}^q.
\end{align*}
This finishes the proof of \eqref{aw-wsc:002}.
\end{proof}
%%%%%%%%%%%% a Lemma %%%%%%%%%%%%%%%%
%\par\noindent\rule{\linewidth}{1pt}
\begin{proof}[Proof of Proposition \ref{prop:strong-convergence-whole-space}]
Consider $ \eta= \frac{1}{2}, s= \frac{1}{2},  q = 6, p = 6 $ in \eqref{aw-wsc:001} and \eqref{aw-wsc:002}, which satisfy \eqref{aw-wsc:indices}. Then from \eqref{aw:representation-sol}, after applying the triangle and H\"older inequalities,
\begin{align*}
	& \normh{\mathcal U_\varepsilon}{L^2(0,T;W^{\frac{1}{2},6}(\mathbb R^2))} \lesssim T^{\frac 1 3}\normh{\mathcal L\bigl(\dfrac{t}{\varepsilon}\bigr)\mathcal U_0}{L^6(0,T;W^{\frac{1}{2},6}(\mathbb R^2))} \\
	& ~~~~ + \normh{\int_0^t\mathcal L\bigl( \dfrac{t-s}{\varepsilon}\bigr) \mathcal G(s)\,ds}{L^2(0,T;W^{\frac{1}{2},6}(\mathbb R^2))} \lesssim T^{\frac 1 3} \varepsilon^{\frac 1 6} \normh{\mathcal U_0}{H^1(\mathbb R^2)} \\
	& ~~~~  + (1+T) \varepsilon^{\frac 1 6} \normh{\mathcal G}{L^2(0,T;H^1(\mathbb R^2))} \lesssim (1+T) \varepsilon^{\frac 1 6},
\end{align*}
where the second  and the third inequalities follow from \eqref{aw-wsc:001}, \eqref{aw-wsc:002} and \eqref{aw:regularity}.
This yields \eqref{aw:strong-convergence-01} and finishes the proof of Proposition \ref{prop:strong-convergence-whole-space}.
\end{proof}

%\pagebreak
\subsection{The case when $ \Omega_h = \mathbb T^2 $: the fast oscillation}\label{subsec:convergence-periodic}
We will establish the convergence behavior of $ \mathcal U_\varepsilon $ in the case when $  \Omega_h = \mathbb T^2 $ in this subsection. Indeed, we shall investigate the fast oscillations of the acoustic waves as $ \varepsilon \rightarrow 0^+ $.
This is motivated by \cite{Schochet1994} (see also \cite{Gallagher1998}).
\subsubsection*{The oscillation equations and the convergence of oscillations}
To begin with, \subeqref{acoustic-wave}{2} can be written as
\begin{equation}\label{aw-per:001}
\begin{aligned}
	& \dt \mathcal P_\tau v_\varepsilon + \dfrac{c^2}{\varepsilon} \nablah \xi_\varepsilon - c_1 ( \mu \deltah \mathcal P_\tau v_\varepsilon + \lambda \nablah \dvh \mathcal P_\tau v_\varepsilon) \\
	& ~~~~ = \dfrac{c^2(1-e^{\varepsilon\xi_\varepsilon})}{\varepsilon} \nablah \xi_\varepsilon - \mathcal P_\tau ( v_\varepsilon \cdot \nablah v_\varepsilon) - \mathcal P_\tau (w_\varepsilon \dz v_\varepsilon) \\
	& ~~~~ ~~~~ + \mathcal P_\tau (c_1 (e^{-\varepsilon\alpha\xi_\varepsilon} -1 )(\mu \deltah  v_\varepsilon + \lambda\nablah \dvh v_\varepsilon + \partial_{zz} v_\varepsilon )).
\end{aligned}
\end{equation}
Moreover, for any $ u \in (\mathcal D'(\mathbb T^2 \times 2\mathbb T))^2 $, consider $ \overline u(x,y) = \int_0^1 u \,dz \in (\mathcal D'(\mathbb T^2 ))^2 $ as a function on $ \mathbb T^2 \times 2 \mathbb T$. One has $ \mathcal P_\tau u = \mathcal P_\tau \overline u $. Then one has, after applying integration by parts and substituting \eqref{vertical-velocity-dz},
\begin{equation}\label{aw-per:0001}
\begin{aligned}
	& \mathcal P_\tau ( w_\varepsilon \dz v_\varepsilon ) = \mathcal P_\tau ( - \int_0^1 \dz w_\varepsilon v_\varepsilon \,dz )\\
	& ~~~~ = \mathcal P_\tau ( \int_0^1 \bigl( (\varepsilon\alpha \widetilde v_\varepsilon \cdot \nablah \xi_\varepsilon + \dvh \widetilde v_\varepsilon) v_\varepsilon \bigr) \,dz ) \\
	& ~~~~ = \mathcal P_\tau (\int_0^1 \bigl( (\varepsilon \alpha \widetilde{\mathcal P_\sigma v_\varepsilon} \cdot \nablah \xi_\varepsilon ) \mathcal P_\sigma v_\varepsilon + \dvh \widetilde{\mathcal P_\sigma v_\varepsilon} \mathcal P_\sigma v_\varepsilon\bigr) \,dz  ) \\
%	& ~~~~  = \mathcal P_\tau ( (\varepsilon \alpha \widetilde v_\varepsilon \cdot \nablah \xi_\varepsilon + \dvh \widetilde v_\varepsilon)  v_\varepsilon ) \\
	& ~~~~ = \mathcal P_\tau ( (\varepsilon \alpha \widetilde{\mathcal P_\sigma v_\varepsilon} \cdot \nablah \xi_\varepsilon ) \mathcal P_\sigma v_\varepsilon + \dvh \widetilde{\mathcal P_\sigma v_\varepsilon} \mathcal P_\sigma v_\varepsilon ) ,  \\
	& \mathcal P_\tau ((e^{-\varepsilon \alpha \xi_\varepsilon}-1) \partial_{zz} v_\varepsilon ) = 0,
\end{aligned}
\end{equation}
where we have used the facts that $ \int_0^1 \widetilde {P_\sigma v_\varepsilon} \,dz = 0, \widetilde v_\varepsilon = \widetilde{P_\sigma v_\varepsilon} $ and that $ \mathcal P_\tau v_\varepsilon, \xi_\varepsilon $ are independent of the $ z $-variable.

On the other hand, notice that
\begin{align*}
	& v_\varepsilon \cdot \nablah \xi_\varepsilon = \mathcal P_\sigma v_\varepsilon \cdot \nablah \xi_\varepsilon + \mathcal P_\tau v_\varepsilon \cdot \nablah \xi_\varepsilon = \dvh (\xi_\varepsilon \mathcal P_\sigma v_\varepsilon ) - \xi_\varepsilon \dvh \mathcal P_\sigma v_\varepsilon \\
	& ~~~~ + \mathcal P_\tau v_\varepsilon \cdot \nablah \xi_\varepsilon.
\end{align*}
\subeqref{acoustic-wave}{1} can be written as
\begin{equation}\label{aw-per:002}
	\dt \xi_\varepsilon + \dfrac{\gamma-1}{\varepsilon} \dvh \mathcal P_\tau v_\varepsilon = - \dvh(\xi_\varepsilon \int_0^1 \mathcal P_\sigma v_\varepsilon \,dz) - \mathcal P_\tau v_\varepsilon \cdot \nablah \xi_\varepsilon.
\end{equation}
Additionally, integrating \eqref{aw-per:002} in $ \mathbb T^2 $ yields
\begin{equation}\label{aw-per:003}
	\dfrac{d}{dt} \int_{\mathbb T^2} \xi_\varepsilon \idxh = - \int_{\mathbb T^2} ( \mathcal P_\tau v_\varepsilon \cdot \nablah \xi_\varepsilon ) \idxh.
\end{equation}

Consequently, after combining equations \eqref{aw-per:001}, \eqref{aw-per:0001}, \eqref{aw-per:002}, \eqref{aw-per:003}, we have the following system of oscillations:
\begin{equation}\label{aw-per:oscillation}
	\begin{cases}
		\dt  \xi_\varepsilon^o + \dfrac{\gamma-1}{\varepsilon} \dvh \mathcal P_\tau v_\varepsilon = - \dvh (\xi_\varepsilon^o \int_0^1 \mathcal P_\sigma v_\varepsilon \,dz) \\
		~~~~  - \mathcal P_\tau v_\varepsilon \cdot \nablah \xi_\varepsilon^o + \int_{\mathbb T^2} ( \mathcal P_\tau v_\varepsilon \cdot \nablah \xi_\varepsilon^o ) \idxh &\text{in} ~ \mathbb T^2, \\
		\dt \mathcal P_\tau v_\varepsilon + \dfrac{c^2}{\varepsilon} \nablah \xi_\varepsilon^o - c_1 (\mu \deltah \mathcal P_\tau v_\varepsilon + \lambda \nablah \dvh \mathcal P_\tau v_\varepsilon ) \\
		~~~~ = - \mathcal P_\tau (v_\varepsilon \cdot \nablah v_\varepsilon +  (\varepsilon \alpha \widetilde{\mathcal P_\sigma v_\varepsilon} \cdot \nablah \xi_\varepsilon ) \mathcal P_\sigma v_\varepsilon) \\
		~~~~ + \mathcal P_\tau (\dvh \widetilde{\mathcal P_\sigma v_\varepsilon} \mathcal P_\sigma v_\varepsilon  )
		+ \dfrac{c^2(1-e^{\varepsilon\xi_\varepsilon})}{\varepsilon} \nablah \xi_\varepsilon^o \\
		~~~~ + \mathcal P_\tau (c_1 (e^{-\varepsilon\alpha\xi_\varepsilon} -1 )(\mu \deltah  v_\varepsilon + \lambda\nablah \dvh v_\varepsilon )) &\text{in} ~ \mathbb T^2,
	\end{cases}
\end{equation}
where
\begin{equation}\label{aw-per:density-osc}
%	\begin{gathered}
		\xi_\varepsilon^o := \xi_\varepsilon - \int_{\mathbb T^2} \xi_\varepsilon \idxh. %\\
%		w_\varepsilon = - \int_0^z \varepsilon \alpha ( v_\varepsilon - \int_0^1 v_\varepsilon\,dz ) \cdot \nablah \xi_\varepsilon + \dvh ( v_\varepsilon - \int_0^1 v_\varepsilon \,dz) \,dz.
%	\end{gathered}
\end{equation}
In the following, denote by
\begin{equation}\label{aw-per:variable}
	\begin{gathered}
	\mathcal U_\varepsilon^o : = \biggl( \begin{array}{c}
		\xi_\varepsilon^o\\ \mathcal P_\tau v_\varepsilon
	\end{array}\biggr) = \biggl( \begin{array}{c}
		\xi_\varepsilon - \int_{\mathbb T^2} \xi_\varepsilon\idxh \\ \nablah \psi_\varepsilon
	\end{array}\biggr), ~~ g_\varepsilon := \int_{\mathbb T^2} \xi_\varepsilon \idxh, ~~~~ \text{and}\\
	 V_\varepsilon^o : = \mathcal L\bigl(-\dfrac{t}{\varepsilon}\bigr) \mathcal U_\varepsilon^o = \biggl(\begin{array}{c}
		 \mathcal L_1\bigl(-\dfrac{t}{\varepsilon}\bigr) \mathcal U_\varepsilon^o\\
		  \mathcal L_2\bigl(-\dfrac{t}{\varepsilon}\bigr) \mathcal U_\varepsilon^o
	\end{array} \biggr) ~~~~ \bigl( \text{so} ~ \mathcal U_\varepsilon^o = \mathcal L \bigl( \dfrac t \varepsilon \bigr) V_\varepsilon^o \bigr),
	\end{gathered}
\end{equation}
where $ \mathcal L_1 : (\mathcal D'(\mathbb T^2))^3 \mapsto \mathcal D'(\mathbb T^2), ~ \mathcal L_2: (\mathcal D'(\mathbb T^2))^3 \mapsto (\mathcal D'(\mathbb T^2))^2   $ will be referred to as the first and the second components of $ \mathcal L $, respectively.
Then $  V_\varepsilon^o $ satisfies
\begin{equation}\label{aw-per:cm-oscillation}
	\begin{aligned}
		& \dt V_\varepsilon^o + \mathcal Q_{\varepsilon,1} (\mathcal P_\sigma v_\varepsilon, V_\varepsilon^o) + \mathcal Q_{\varepsilon,2} (V_\varepsilon^o, V_\varepsilon^o) + \mathcal Q_{\varepsilon,3}(g_\varepsilon, V_\varepsilon^o) - \mathcal A_\varepsilon(D) V_\varepsilon^o \\
		& ~~~~ ~~~~ ~~~~ = \mathcal L\bigl(-\dfrac{t}{\varepsilon}\bigr) \biggl( \begin{array}{c}
			0 \\ K+ L_1 + L_2
			\end{array} \biggr),
	\end{aligned}
\end{equation}
where the linear and bi-linear operators are defined as follows: for $ V_1, V_2 \in (\mathcal D'(\mathbb T^2))^3, u \in (\mathcal D'(\mathbb T^2 \times 2 \mathbb T))^2, g \in \mathbb R $,
\begin{equation}\label{aw-per:operators}
	\begin{aligned}
		& \mathcal A_\varepsilon(D) V_1 : = \mathcal L \bigl(-\dfrac{t}{\varepsilon} \bigr) \biggl(\begin{array}{c}
			0 \\ c_1 ( \mu \deltah \mathcal L_2\bigl( \dfrac{t}{\varepsilon} \bigr) V_1 + \lambda \nablah \dvh \mathcal L_2\bigl( \dfrac{t}{\varepsilon} \bigr) V_1)
		\end{array} \biggr), \\
		& \mathcal Q_{\varepsilon, 1}( \mathcal P_\sigma u,V_1) := \mathcal L\bigl(-\dfrac{t}{\varepsilon}\bigr) \biggl( \begin{array}{c}
		\dvh ( \mathcal L_1 \bigl(\dfrac{t}{\varepsilon}\bigr) V_1 \int_0^1 \mathcal P_\sigma u \,dz )	\\
		\mathcal P_\tau ( \mathcal P_\sigma u \cdot \nablah \mathcal L_2\bigl(\dfrac t \varepsilon\bigr) V_1 + \mathcal L_2\bigl(\dfrac t \varepsilon\bigr) V_1 \cdot \nablah \mathcal P_\sigma u )
		\end{array} \biggr) , \\
		& \mathcal Q_{\varepsilon, 2}(V_{1} , V_{2})\\ &
		 := \mathcal L\bigl(-\dfrac{t}{\varepsilon}\bigr) \biggl( \begin{array}{c}
			\mathcal L_2 \bigl(\dfrac{t}{\varepsilon}\bigr) V_{1} \cdot \nablah \mathcal L_1 \bigl(\dfrac{t}{\varepsilon}\bigr) V_{2} - \int_{\mathbb T^2} \mathcal L_2 \bigl(\dfrac{t}{\varepsilon}\bigr) V_{1} \cdot \nablah \mathcal L_1 \bigl(\dfrac{t}{\varepsilon}\bigr) V_{2} \idxh \\
			\mathcal P_\tau (\mathcal L_2\bigl(\dfrac t \varepsilon\bigr) V_{1} \cdot \nablah \mathcal L_2\bigl(\dfrac t \varepsilon\bigr) V_{2} + c^2 \mathcal L_1\bigl(\dfrac{t}{\varepsilon}\bigr) V_{1} \nablah \mathcal L_1 \bigl(\dfrac{t}{\varepsilon}\bigr) V_{2} )
		\end{array} \biggr) ,\\
		& \mathcal Q_{\varepsilon, 3} (g, V_1) := \mathcal L \bigl( - \dfrac{t}{\varepsilon} \bigr) \biggl( \begin{array}{c}
			0 \\ \mathcal P_\tau( c^2  g
			 \nablah \mathcal L_1\bigl( \dfrac t \varepsilon \bigr) V_1 )
		\end{array} \biggr).
	\end{aligned}
\end{equation}
Here,
\begin{equation}\label{aw-per:quantities}
	\begin{aligned}
		& K : = \mathcal P_\tau ( - \mathcal P_\sigma v_\varepsilon \cdot \nablah \mathcal P_\sigma v_\varepsilon - \dvh \widetilde{\mathcal P_\sigma  v_\varepsilon } P_\sigma v_\varepsilon  ),\\
		& L_1 : =  - \varepsilon \alpha \mathcal P_\tau ( (\widetilde{\mathcal P_\sigma v_\varepsilon} \cdot \nablah \xi_\varepsilon) \mathcal P_\sigma v_\varepsilon) + \mathcal P_\tau (\dfrac{c^2(1+\varepsilon  \xi_\varepsilon - e^{\varepsilon \xi_\varepsilon})}{\varepsilon} \nablah \xi_\varepsilon^o),\\
		& L_2 : =  \mathcal P_\tau (c_1 (e^{-\varepsilon\alpha\xi_\varepsilon} -1 )(\mu \deltah  v_\varepsilon + \lambda\nablah \dvh v_\varepsilon )) . %\\
%		& K_2 := \mathcal P_\tau ( \mathcal P_\sigma v_\varepsilon \cdot \nablah \mathcal L_2\bigl(\dfrac t \varepsilon\bigr) V_\varepsilon^o + \mathcal L_2\bigl(\dfrac t \varepsilon\bigr) V_\varepsilon^o \cdot \nablah \mathcal P_\sigma v_\varepsilon ).% \\
%		& K_3 := \mathcal P_\tau ( \dvh  \widetilde{\mathcal P_\sigma  v_\varepsilon}  \mathcal L_2\bigl(\dfrac t \varepsilon\bigr) V_\varepsilon^o) = 0.
	\end{aligned}
\end{equation}
Also $ g_\varepsilon(t) =\int_{\mathbb T^2} \xi_\varepsilon \idxh $ satisfies, from \eqref{aw-per:003},
\begin{equation}\label{aw-per:average}
\begin{aligned}
	& \dfrac{d}{dt} g_\varepsilon = - \int_{\mathbb T^2} ( (\mathcal U_\varepsilon^o)_2 \cdot \nablah (\mathcal U^o_\varepsilon)_1 ) \idxh \\
	& ~~~~ = - \int_{\mathbb T^2} \mathcal L_2\bigl( \dfrac t \varepsilon \bigr) V_\varepsilon^o \cdot \nablah \mathcal L_1\bigl( \dfrac t \varepsilon \bigr) V_\varepsilon^o \idxh, \\
	& ~~~~ \text{with} ~~ g_{\varepsilon}(t=0) = \int_{\mathbb T^2} \xi_0 \idxh.  %\mathcal P_\tau v_\varepsilon \cdot \nablah \xi_\varepsilon \idxh.
\end{aligned}
\end{equation}

Now we will address the strong convergence of $ V_\varepsilon^o, g_\varepsilon $ as $ \varepsilon \rightarrow 0^+ $. Notice that $ L $ defined in \eqref{aw:linear-operator} is anti-symmetric with respect to $ A:= \mathrm{diag}( c^2, \gamma-1, \gamma - 1) $ (i.e., $ \int V^\top A L U \idx =  - \int U^\top A L V \idx $ ), linear and commutative with $ \partial_h \in \lbrace \partial_x, \partial_y \rbrace $. Consequently the standard $ H^s $ estimate of the solutions to equation \eqref{aw:def-group-operator} implies that the operator $ \mathcal L $ preserves the $ H^s $ norm, i.e.,
\begin{equation}\label{aw-per:norm-preserve}
\normh{\mathcal L(t) V}{\Hnorm{s}} \simeq \normh{V}{\Hnorm{s}} ~ \text{for} ~ t \in (0,\infty), s \in \mathbb N, V \in (\mathcal D'(\mathbb T^2))^3.
\end{equation} Therefore one has, as the consequence of \eqref{aw:regularity} and \eqref{aw-per:variable},
\begin{equation}\label{aw-per:norm-preserving}
	 \normh{V_\varepsilon^o}{L^\infty(0,T;H^2(\mathbb T^2))} + \normh{\mathcal U_\varepsilon^o}{L^\infty(0,T;H^2(\mathbb T^2))} + \normh{(\mathcal U_\varepsilon^o)_2}{L^2(0,T;H^3(\mathbb T^2))}
	< C.
\end{equation}
Then it is easy to see from \eqref{aw-per:average} that
\begin{equation}\label{aw-per:bounds-of-mass-average}
	\sup_{0\leq t\leq T} \lbrace \abs{g_\varepsilon}{} + \abs{\dfrac{d}{dt} g_\varepsilon}{} \rbrace \leq (1+T) \normh{\mathcal U_\varepsilon^o}{L^\infty(0,T;H^2(\mathbb T^2))}^2  < C.
\end{equation}
Therefore the Arzel\`{a}-Ascoli theorem implies that as $ \varepsilon\rightarrow 0^+ $,
\begin{equation}\label{aw-per:lim-average} g_\varepsilon \rightarrow g^o ~ \text{uniformly for some} ~ g^o \in C([0,T]),
\end{equation}
and $ \abs{g^o}{} < C $. On the other hand, after applying the H\"older and Sobolev embedding inequalities, \eqref{ue:total}, \eqref{aw-per:operators}, \eqref{aw-per:norm-preserving}, and \eqref{aw-per:bounds-of-mass-average} imply the following estimates:
\begin{equation}\label{aw-per:bounds-of-operators}
\begin{aligned}
		& \normh{\mathcal A_\varepsilon(D) V_\varepsilon^o }{L^\infty(0,T;L^2(\mathbb T^2))\cap L^2(0,T;H^1(\mathbb T^2)} \lesssim \normh{V_\varepsilon^o}{L^\infty(0,T;H^2(\mathbb T^2))} \\
		& ~~~~ + \normh{(\mathcal U_\varepsilon^o)_2}{L^2(0,T;H^3(\mathbb T^2))} < C, \\
		& \normh{Q_{\varepsilon,1}(\mathcal P_\sigma v_\varepsilon, V_\varepsilon^o)}{L^\infty(0,T; H^1(\mathbb T^2))} \lesssim \normh{V_\varepsilon^o}{L^\infty(0,T;H^2(\mathbb T^2))} \\
		& ~~~~ \times \norm{\mathcal P_\sigma v_\varepsilon}{L^\infty(0,T;H^2(\mathbb T^2))} < C,\\
		&  \normh{Q_{\varepsilon,2}(V_\varepsilon^o , V_\varepsilon^o)}{L^\infty(0,T; H^1(\mathbb T^2))} \lesssim \normh{V_\varepsilon^o}{L^\infty(0,T;H^2(\mathbb T^2))}^2 < C,\\
		&  \normh{Q_{\varepsilon,3}(g_\varepsilon, V_\varepsilon^o)}{L^\infty(0,T; H^1(\mathbb T^2))} \lesssim \sup_{0\leq t \leq T}\abs{g_\varepsilon}{} \\
		& ~~~~ \times \normh{V_\varepsilon^o}{L^\infty(0,T;H^2(\mathbb T^2))} < C,
\end{aligned}
\end{equation}
and
\begin{equation}\label{aw-per:bound-of-source}
	\begin{aligned}
		& \normh{K}{L^\infty(0,T;H^1(\mathbb T^2))} \lesssim \norm{\mathcal P_\sigma v_\varepsilon}{L^\infty(0,T;H^2(\mathbb T^2))}^2 < C ,  \\
		& \normh{L_1}{L^\infty(0,T;H^1(\mathbb T^2))} \lesssim \varepsilon \norm{\mathcal P_\sigma v_\varepsilon}{L^\infty(0,T;H^2(\mathbb T^2))}^2 \norm{\xi_\varepsilon}{L^\infty(0,T;H^2(\mathbb T^2))} \\
		& ~~~~ + \varepsilon e^{\varepsilon \norm{\xi_\varepsilon}{L^\infty(0,T;H^2(\mathbb T^2))}}(\norm{\xi_\varepsilon}{L^\infty(0,T;H^2(\mathbb T^2))}^2 + 1) < \varepsilon C,\\
		& \normh{L_2}{L^\infty(0,T; L^2(\mathbb T^2)) \cap L^2(0,T; H^1(\mathbb T^2))} \lesssim \varepsilon e^{\varepsilon \norm{\xi_\varepsilon}{L^\infty(0,T;H^2(\mathbb T^2))}}\\
		& ~~~~ \times ( \norm{\xi_\varepsilon}{L^\infty(0,T;H^2(\mathbb T^2))} + 1)\norm{v_\varepsilon}{L^\infty(0,T; H^2(\mathbb T^2)) \cap L^2(0,T; H^3(\mathbb T^2))} \\
		& ~~~~  < \varepsilon C.
	\end{aligned}
\end{equation}
Therefore, equation \eqref{aw-per:cm-oscillation} implies
\begin{equation*}
	\normh{\dt V_\varepsilon^o}{L^\infty(0,T; L^2(\mathbb T^2)) \cap L^2(0,T; H^1(\mathbb T^2))} < C.
\end{equation*}
Then the Aubin compactness lemma (see, i.e., \cite[Theorem 2.1]{Temam1984} and \cite{Simon1986,Chen2012}) implies that, there is
\begin{equation}\label{aw-per:limit-of-osc-regulartity}
\begin{gathered}
V^o \in L^\infty(0,T;H^2(\mathbb T^2)) \cap C(0,T;H^1(\mathbb T^2)) ~~  \text{with}\\
\dt V^o \in L^\infty(0,T;L^2(\mathbb T^2)) \cap L^2(0,T;H^1(\mathbb T^2)),
\end{gathered}
\end{equation}
such that as $ \varepsilon \rightarrow 0^+ $,
\begin{equation}\label{aw-per:limit-of-oscillation-variable}
\begin{gathered}
	V_\varepsilon^o \buildrel\ast\over\rightharpoonup V^o ~~ \text{weak-$\ast$ in} ~ L^\infty(0,T;H^2(\mathbb T^2)), \\
	V_\varepsilon^o \rightarrow V^o ~~ \text{in} ~ L^\infty(0,T;H^1(\mathbb T^2)) \cap C(0,T;H^1(\mathbb T^2)),\\
	\dt V_\varepsilon^o \buildrel\ast\over\rightharpoonup \dt V^o ~~ \text{weak-$\ast$ in} ~ L^\infty(0,T;L^2(\mathbb T^2)), \\
	 \dt V_\varepsilon^o \rightharpoonup \dt V^o ~~ \text{weakly in} ~ L^2(0,T;H^1(\mathbb T^2)).
\end{gathered}
\end{equation}
Also
\begin{equation}\label{aw-per:initial-of-oscillation}
	V_\varepsilon^o(t=0) = V^o(t=0) = \biggl( \begin{array}{c} \xi_0 - \int_{\mathbb T^2} \xi_0 \idxh \\ \mathcal P_\tau v_0 \end{array}\biggr).
\end{equation}
Consequently, as $ \varepsilon \rightarrow 0^+ $, since $ \mathcal L $ preserves the $ H^1 $ norm,
\begin{equation}\label{aw-per:strong-convergence-001}
	\mathcal U_\varepsilon^o -  \mathcal L\bigl( \dfrac t \varepsilon \bigr) V^o = \mathcal L \bigl( \dfrac t \varepsilon \bigr) \bigl(  V_\varepsilon^o  - V^o \bigr) \rightarrow 0 ~~ \text{in} ~ L^\infty(0,T;H^1(\mathbb T^2)).
\end{equation}
In conclusion, as $ \varepsilon \rightarrow 0^+ $,  \eqref{limit:007}, \eqref{aw-per:lim-average} and \eqref{aw-per:strong-convergence-001}  imply that
\begin{equation}\label{aw-per:strong-cnvgn-1}
\begin{aligned}
	&  \norm{v_\varepsilon - v_p - \mathcal L_2 \bigl(\dfrac t \varepsilon \bigr) V^o}{L^\infty(0,T;H^1(\mathbb T^2 \times 2 \mathbb T))} \\
	& ~~~~ + \norm{\xi_\varepsilon - g^o - \mathcal L_1 \bigl(\dfrac t \varepsilon \bigr) V^o }{L^\infty(0,T;H^1(\mathbb T^2\times 2 \mathbb T))} + \norm{g_\varepsilon - g^o}{L^\infty(0,T)}\\
	& ~~ \lesssim \norm{g_\varepsilon - g^o}{L^\infty(0,T)} + \normh{\mathcal U_\varepsilon^o -  \mathcal L\bigl( \dfrac t \varepsilon \bigr) V^o}{L^\infty(0,T;H^1(\mathbb T^2))} \\
	& ~~~~ + \norm{\mathcal P_\sigma v_\varepsilon - v_p}{L^\infty(0,T;H^1(\mathbb T^2 \times 2\mathbb T))} \rightarrow 0.
\end{aligned}
\end{equation}
Next, we will identify the limit equations of \eqref{aw-per:cm-oscillation} and \eqref{aw-per:average}; that is, the equation satisfied by $ V^o $ and $ g^o $.

\subsubsection*{The limit equations of oscillations}
In order to identify the limit equations of \eqref{aw-per:cm-oscillation} and \eqref{aw-per:average} as $ \varepsilon \rightarrow 0^+ $, we first introduce the Fourier representation of the operators defined in \eqref{aw-per:operators}. Notice that \eqref{aw:def-group-operator} implies $ \int_{\mathbb T^2} V_\varepsilon^o \idxh = \int_{\mathbb T^2} \mathcal U_\varepsilon^o \idxh = 0 $. It suffices to study in the space consisting of functions in $ \mathcal D'(\mathbb T^2)^3 $ with zero average.

%Without further mentioned, all the function spaces in the following are with zero average. This will enable us to introduce the basis in \eqref{aw-per:basis001}. For simplicity, $ \mathbb T^2 $ is the periodic domain with periods $ (1,1) $.

Notice that, on the other hand, $ \mathcal U_\varepsilon^o $ is inside the orthogonal complement of the kernel of operator $ L $ with respect to the $ L^2 $-inner product, and thanks to \eqref{aw:def-group-operator}, so is $ V_\varepsilon^o $. In fact, $ (\ker{L})^\perp = \lbrace (q,u)  \subset \mathcal D'(\mathbb T^2) \times (\mathcal D'(\mathbb T^2) )^2 |  \int_{\mathbb T^2} q \idxh = 0,u = \nablah \psi, \psi \in \mathcal D'(\mathbb T^2) \rbrace  $. Next, we introduce, as in \cite{Danchin2002per} and \cite{Masmoudi2001}, the following basis of $ (\ker L)^\perp $: for $ \mathbf k \in 2\pi \mathbb Z^2 \setminus \lbrace (0,0)\rbrace  $,
\begin{equation}\label{aw-per:basis001}
	\begin{aligned}
		& V_{\mathbf{k}}^+ (\vech{x}) : = \dfrac{1}{\sqrt{\gamma-1 + c^2} \abs{\mathbf k}{}} \biggl( \begin{array}{c}
			\sqrt{\gamma - 1} \abs{\mathbf k}{} \\  -  c  \sg (\mathbf k) \mathbf k
		\end{array}\biggr) e^{i\mathbf k \cdot \vech{x}}, \\
		& V_{\mathbf{k}}^- (\vech{x}) : = \dfrac{1}{\sqrt{\gamma-1 + c^2} \abs{\mathbf k}{}} \biggl( \begin{array}{c}
			\sqrt{\gamma-1} \abs{\mathbf k}{} \\  + c \sg (\mathbf k) \mathbf k
		\end{array}\biggr) e^{i\mathbf k \cdot \vech{x}},
	\end{aligned}
\end{equation}
and the corresponding conjugates
\begin{equation*}
	\begin{aligned}
		& V_{\mathbf{k}}^{*,+} (\vech{x}) : = \dfrac{\sqrt{\gamma-1 + c^2} }{2c\sqrt{\gamma-1}\abs{\mathbf k}{}} \biggl( \begin{array}{c}
			c \abs{\mathbf k}{} \\  - \sqrt{\gamma-1}  \sg (\mathbf k) \mathbf k
		\end{array}\biggr) e^{i\mathbf k \cdot \vech{x}}, \\
		& V_{\mathbf{k}}^{*,-} (\vech{x}) : = \dfrac{\sqrt{\gamma-1 + c^2} }{2c\sqrt{\gamma-1}\abs{\mathbf k}{}} \biggl( \begin{array}{c}
			c \abs{\mathbf k}{} \\  + \sqrt{\gamma-1} \sg (\mathbf k) \mathbf k
		\end{array}\biggr) e^{i\mathbf k \cdot \vech{x}}.
	\end{aligned}
\end{equation*}
Here $ \sg (\mathbf{k}) $ is the generalized sign function for vector $ \mathbf k \in 2\pi \mathbb T^2 \setminus \lbrace(0,0)\rbrace $, defined as,
\begin{equation*}
	\sg (\mathbf{k}) :=
	\begin{cases}
		1 & \text{if and only if $ k_1 > 0 $ or $ k_1= 0, k_2 > 0 $},\\
		-1 & \text{otherwise}.
	\end{cases}
\end{equation*}
Notice that $ V_{\mathbf k}^\pm = \overline{V_{ - \mathbf k}^\pm}^c $, where, hereafter, $ \overline{\cdot}^c $ represents the complex conjugate.

Then by defining $ \varsigma = c \sqrt{\gamma-1} $, $ \lbrace V_{\mathbf k}^\pm \rbrace_{\mathbf k \in 2\pi \mathbb T^2 \setminus \lbrace(0,0)\rbrace} $ are the eigenfunctions of $ L $ with the eigenvalues $ \lbrace \mp i \varsigma \sg (\mathbf k) \abs{\mathbf k}{}\rbrace_{\mathbf k \in 2\pi \mathbb T^2 \setminus \lbrace(0,0)\rbrace} $, i.e.,
\begin{equation*}
	L V_{\mathbf k}^+ = - i \varsigma \sg (\mathbf k) \abs{\mathbf k}{}  V_{\mathbf k}^+, ~~ \text{and} ~~ L V_{\mathbf k}^- =  i \varsigma \sg (\mathbf k) \abs{\mathbf k}{}  V_{\mathbf k}^-, ~ \mathbf k \in 2\pi \mathbb T^2 \setminus \lbrace(0,0)\rbrace.
\end{equation*}
%where $ \lambda_+ : = , \lambda_- : = $
Also any $ V \in (\ker L)^\perp $ can be represented in terms of Fourier series as
\begin{equation}\label{aw-per:fourier-rprsnt}
V = \sum_{\mathbf k \in 2\pi \mathbb T^2 \setminus \lbrace (0,0)\rbrace} \bigl( a_{\mathbf k}^+ V_{\mathbf k}^+ +   a_{\mathbf k}^- V_{\mathbf k}^- \bigr)  \end{equation}
with the Fourier coefficients $ \lbrace a_{\mathbf k}^\pm % =  a_{\mathbf k}^\pm(t)
 \rbrace_{\mathbf k \in 2\pi \mathbb T^2 \setminus \lbrace (0,0) \rbrace} $ determined by
 \begin{equation}\label{aw-per:fourier-rprsnt01}
	a_{\mathbf k}^\pm = \langle V ,{V_{\mathbf k}^{*,\pm}}\rangle_{\mathbb C}:= \int_{\mathbb T^2} V \cdot \overline{V_{\mathbf k}^{*,\pm}}^c \idxh.
\end{equation}
If $ V $ is real-valued, $ a_{-\mathbf k}^{\pm} = \overline{a_{\mathbf k}^{\pm}}^c $.
In addition,
one has
\begin{equation}\label{20July2018}
	\mathcal L(s) V = \sum_{\mathbf k\in 2\pi \mathbb T^2 \setminus \lbrace (0,0)\rbrace} \bigl( a_{\mathbf k}^+ V_{\mathbf k}^+ e^{i\varsigma \sg (\mathbf k) \abs{\mathbf k}{}s} + a_{\mathbf k}^- V_{\mathbf k}^- e^{-i\varsigma\sg (\mathbf k) \abs{\mathbf k}{}s}\bigr) .
\end{equation}
%Notice, for $ \mathbf k \in 2\pi \mathbb T^2 \setminus \lbrace(0,0)\rbrace $, $ a_{\mathbf k}^\pm $ is determined by
%\begin{equation*}
%	a_{\mathbf k}^\pm = (V ,{V_{\mathbf k}^{*,\pm}})_{\mathbb C}:= \int_{\mathbb T^2} V \cdot \overline{V_{\mathbf k}^{*,\pm}} \idxh.
%\end{equation*}
%where $ (\cdot,\cdot)_{\mathbb C} : = \int_{\mathbb T^2}  $
%If $ V $ is real-valued, $ a_{-\mathbf k}^{\pm} = \overline{a_{\mathbf k}^{\pm}} $.

Denote by $ \mathcal P_{(\ker L)^\perp}: (\mathcal D'(\mathbb T^2))^3 \mapsto (\ker L)^\perp \subset (\mathcal D'(\mathbb T^2))^3 $, being the projection operator on $ (\ker L)^\perp $ with respect to the $ L^2 $-inner product. Then for any $ q \in \mathcal D'(\mathbb T^2) $ and $ u \in (\mathcal D'(\mathbb T^2\times 2 \mathbb T))^2 $, it satisfies,
\begin{equation}\label{aw-per:projection}
	\mathcal P_{(\ker L)^\perp} \biggl(\begin{array}{c}
		q \\ \int_0^1 u \,dz
	\end{array} \biggr) =  \biggl( \begin{array}{c}
		q - \int_{\mathbb T^2} q \idxh \\ \mathcal P_\tau  \int_0^1 u\,dz
	\end{array} \biggr)  = \biggl( \begin{array}{c}
		q - \int_{\mathbb T^2} q \idxh \\ \mathcal P_\tau u
	\end{array} \biggr).
\end{equation}
Moreover, as in \eqref{aw-per:fourier-rprsnt} and \eqref{aw-per:fourier-rprsnt01}, one has
\begin{equation}\label{aw-per:projection-01}
	\mathcal P_{(\ker L)^\perp} \biggl(\begin{array}{c}
		q \\ \int_0^1 u \,dz
	\end{array} \biggr) = \sum_{\mathbf k \in 2\pi \mathbb T^2 \setminus \lbrace (0,0)\rbrace} \bigl( b_{\mathbf k}^+ V_{\mathbf k}^+ +   b_{\mathbf k}^- V_{\mathbf k}^- \bigr),
\end{equation}
with $ \lbrace b_{\mathbf k}^\pm  \rbrace_{\mathbf k \in 2\pi \mathbb T^2 \setminus \lbrace (0,0) \rbrace} $ determined by
\begin{equation*}
	b_{\mathbf k}^\pm = \langle\biggl(\begin{array}{c}
		q \\ \int_0^1 u \,dz
	\end{array} \biggr) ,{V_{\mathbf k}^{*,\pm}}\rangle_{\mathbb C}= \int_{\mathbb T^2} \biggl(\begin{array}{c}
		q \\ \int_0^1 u \,dz
	\end{array} \biggr) \cdot \overline{V_{\mathbf k}^{*,\pm}}^c \idxh.
\end{equation*}

We will now investigate the action of the operators in \eqref{aw-per:operators} on $ V_{\mathbf k}^\pm $. Without further mentioning, we will use the representations \eqref{aw-per:projection} and \eqref{aw-per:projection-01}, below. To shorten the notations, the multi-indices $ \mathbf k, \mathbf l, \mathbf m, \cdots $ in the following calculations are always subject to the set $ 2\pi \mathbb T^2 \setminus \lbrace (0,0)\rbrace $.
For instance,
\begin{equation*}
	\sum_{\mathbf k} = \sum_{\mathbf k \in 2\pi \mathbb T^2 \setminus \lbrace (0,0)\rbrace}, ~~~~ \text{etc}.
\end{equation*}

{\par \noindent \bf Calculation of $ \mathcal A_{\varepsilon}(D) $. } To calculate $ \mathcal A_{\varepsilon}(D) V_{\mathbf k}^\pm $, notice that,
\begin{align*}
	& c_1 \bigl( \mu \deltah \mathcal L_2 (\frac{t}{\varepsilon}) V_{\mathbf k}^\pm + \lambda \nablah \dvh \mathcal L_2 (\frac{t}{\varepsilon}) V_{\mathbf k}^\pm \bigr)\\
	& ~~ = \pm c_1(\mu+\lambda) \dfrac{c\sg (\mathbf k) \abs{\mathbf k}{} \mathbf k }{\sqrt{\gamma-1+c^2}} e^{\pm i \varsigma \sg(\mathbf k) \abs{\mathbf k}{} \frac{t}{\varepsilon}} e^{i\mathbf k\cdot \vech x}.
\end{align*}
Denote by
\begin{align*}
	& \biggl(\begin{array}{c}
			0 \\ c_1 ( \mu \deltah \mathcal L_2\bigl( \dfrac{t}{\varepsilon} \bigr) V_{\mathbf k}^\pm + \lambda \nablah \dvh \mathcal L_2\bigl( \dfrac{t}{\varepsilon} \bigr) V_{\mathbf k}^\pm)
		\end{array} \biggr) =: A_{\mathbf k} = \alpha_{\mathbf k} V_{\mathbf k}^\pm + \beta_{\mathbf k} V_{\mathbf k}^\mp.
\end{align*}
Then
\begin{align*}
	& \alpha_{\mathbf k} = \int_{\mathbb T^2} A_{\mathbf k} \cdot \overline{V_{\mathbf k}^{*,\pm}}^c \idxh  = \dfrac{\sqrt{\gamma-1 + c^2} }{2c\sqrt{\gamma-1}\abs{\mathbf k}{}}e^{\pm i \varsigma \sg(\mathbf k) \abs{\mathbf k}{} \frac{t}{\varepsilon}}\\
	& ~~~~ \times \biggl( \begin{array}{c}
		0 \\
		\pm c_1(\mu+\lambda) \dfrac{c\sg (\mathbf k) \abs{\mathbf k}{} \mathbf k }{\sqrt{\gamma-1+c^2}} e^{\pm i \varsigma \sg(\mathbf k) \abs{\mathbf k}{} \frac{t}{\varepsilon}}
	\end{array} \biggr) \cdot  \biggl( \begin{array}{c}
			c \abs{\mathbf k}{} \\  \mp \sqrt{\gamma-1}  \sg (\mathbf k) \mathbf k
		\end{array}\biggr)\\
		& ~~~~ = - \dfrac{c_1(\mu+\lambda)}{2}\abs{\mathbf k}{2} e^{\pm i\varsigma \sg(\mathbf k) \abs{\mathbf k}{} \frac{t}{\varepsilon}}, ~~~~ ~~~~ \text{and similarly}\\
		& \beta_{\mathbf k} = \int_{\mathbb T^2} A_{\mathbf k} \cdot \overline{V_{\mathbf k}^{*,\mp}}^c \idxh = \dfrac{c_1(\mu+\lambda)}{2}\abs{\mathbf k}{2} e^{\pm i\varsigma \sg(\mathbf k) \abs{\mathbf k}{} \frac{t}{\varepsilon}}.
\end{align*}
Therefore
\begin{align*}
	& \mathcal A_{\varepsilon}(D) V_{\mathbf k}^\pm = \mathcal L\bigl( - \dfrac{t}{\varepsilon} \bigr) A_{\mathbf k} = - \dfrac{c_1(\mu+\lambda)}{2} \abs{\mathbf k}{2} V_{\mathbf k}^\pm \\
	& ~~~~ +  \dfrac{c_1(\mu+\lambda)}{2} \abs{\mathbf k}{2}  V_{\mathbf k}^{\mp}e^{\pm 2 i \varsigma \sg(\mathbf k) \mathbf k \frac{t}{\varepsilon}}.
\end{align*}

{\par \noindent \bf Calculation of $ \mathcal Q_{\varepsilon, 1} $. } To calculate $ \mathcal Q_{\varepsilon,1}(\mathcal P_\sigma u, V_{\mathbf k}^\pm ) $,
we first represent $ \mathcal P_\sigma u  $ as, after expanding it as Fourier series in the horizontal direction,
\begin{equation}\label{aw-per:fourier-u}
	  \mathcal P_\sigma u = \sum_{\mathbf k} u_{\mathbf k}(z,t) e^{i\mathbf k\cdot \vech{x}}, ~~ \int_0^1 \mathcal P_\sigma u  \,dz = \sum_{\mathbf k}\overline{u}_{\mathbf k} e^{i\mathbf k\cdot \vech{x}},
\end{equation}
satisfying
\begin{equation}\label{aw-per:fourier-u-01}
	\overline{u}_{\mathbf k} \cdot \mathbf k = 0.
\end{equation}
Here, recall that $ \overline{u}_{\mathbf k}= \int_0^1  u_{\mathbf k} \,dz $. Then the relations \eqref{aw-per:projection} and \eqref{aw-per:projection-01} allow us the write down the following:
\begin{align*}
	& \biggl( \begin{array}{c}
		\dvh ( \mathcal L_1 \bigl(\dfrac{t}{\varepsilon}\bigr) V_{\mathbf k}^\pm \int_0^1 \mathcal P_\sigma u \,dz )	\\
		\mathcal P_\tau ( \mathcal P_\sigma u \cdot \nablah \mathcal L_2\bigl(\dfrac t \varepsilon\bigr) V_{\mathbf k}^\pm + \mathcal L_2\bigl(\dfrac t \varepsilon\bigr) V_{\mathbf k}^\pm \cdot \nablah \mathcal P_\sigma u )
		\end{array} \biggr) \\
	& = \mathcal P_{(\ker L)^\perp} \biggl( \begin{array}{c}
		\dvh ( \mathcal L_1 \bigl(\dfrac{t}{\varepsilon}\bigr) V_{\mathbf k}^\pm \int_0^1 \mathcal P_\sigma u \,dz )	\\
		\int_0^1 \mathcal P_\sigma u \,dz \cdot \nablah \mathcal L_2\bigl(\dfrac t \varepsilon\bigr) V_{\mathbf k}^\pm + \mathcal L_2\bigl(\dfrac t \varepsilon\bigr) V_{\mathbf k}^\pm \cdot \nablah \int_0^1 \mathcal P_\sigma u \,dz )
		\end{array} \biggr)\\
	& =: \sum_{ \mathbf  m}\bigl( \xi_{\mathbf m} V_{\mathbf m}^\pm +  \zeta_{\mathbf m}  V_{\mathbf m}^\mp \bigr).
\end{align*}
Denote by
\begin{align*}
	& B_{\mathbf k} := \biggl( \begin{array}{c}
		\dvh ( \mathcal L_1 \bigl(\dfrac{t}{\varepsilon}\bigr) V_{\mathbf k}^\pm \int_0^1 \mathcal P_\sigma u \,dz )	\\
		\int_0^1 \mathcal P_\sigma u \,dz \cdot \nablah \mathcal L_2\bigl(\dfrac t \varepsilon\bigr) V_{\mathbf k}^\pm + \mathcal L_2\bigl(\dfrac t \varepsilon\bigr) V_{\mathbf k}^\pm \cdot \nablah \int_0^1 \mathcal P_\sigma u \,dz )
		\end{array} \biggr) \\
		&  =  \sum_{\mathbf l} \biggl( \begin{array}{c}
			i  \dfrac{\sqrt{\gamma-1}}{\sqrt{\gamma-1+c^2}} (\mathbf k + \mathbf l) \cdot \hat u_{\mathbf l} \\
			i \dfrac{\mp c\sg (\mathbf k)}{\sqrt{\gamma-1+c^2}\abs{\mathbf k}{}} ((\hat u_{\mathbf l} \cdot \mathbf k) \mathbf k + (\mathbf k \cdot \mathbf l) \hat u_{\mathbf l})
			\end{array}\biggr)
			e^{\pm i \varsigma \sg(\mathbf k)\abs{\mathbf k}{} \frac{t}{\varepsilon}}
			e^{i (\mathbf k + \mathbf l) \cdot \vech x}.
\end{align*}
Then
\begin{align*}
	& \xi_{\mathbf m} = \int_{\mathbb T^2} B_{\mathbf k} \cdot \overline{V_{\mathbf m}^{*,\pm}}^c \idxh = \dfrac{\sqrt{\gamma-1 + c^2} }{2c\sqrt{\gamma-1}\abs{\mathbf m}{}} e^{\pm i \varsigma \sg(\mathbf k) \abs{\mathbf k}{} \frac{t}{\varepsilon}} \\
	& \times \biggl( \begin{array}{c}
			i  \dfrac{\sqrt{\gamma-1}}{\sqrt{\gamma-1+c^2}} (\mathbf k + \mathbf l) \cdot \hat u_{\mathbf l} \\
			i \dfrac{\mp c\sg (\mathbf k)}{\sqrt{\gamma-1+c^2}\abs{\mathbf k}{}} ((\hat u_{\mathbf l} \cdot \mathbf k) \mathbf k + (\mathbf k \cdot \mathbf l) \hat u_{\mathbf l})
			\end{array}\biggr) \cdot  \biggl( \begin{array}{c}
			c \abs{\mathbf m}{} \\  \mp \sqrt{\gamma-1} \sg (\mathbf m) \mathbf m
		\end{array}\biggr) \\
	&  = \dfrac{i}{2} (\mathbf k \cdot \hat u_{\mathbf l}) \bigl( 1 + (\mathbf m \cdot\mathbf k + \mathbf k \cdot \mathbf l)\dfrac{\sg(\mathbf m)\sg(\mathbf k)}{\abs{\mathbf m}{}\abs{\mathbf k}{}} \bigr) e^{\pm i \varsigma \sg(\mathbf k) \abs{\mathbf k}{} \frac{t}{\varepsilon}}, ~~ \text{and similarly}
  \\
	& \zeta_{\mathbf m} = \int_{\mathbb T^2} B_{\mathbf k} \cdot \overline{V_{\mathbf m}^{*,\mp}}^c \idxh \\
	&  = \dfrac{i}{2} (\mathbf k \cdot \hat u_{\mathbf l}) \bigl( 1 - (\mathbf m \cdot\mathbf k + \mathbf k \cdot \mathbf l)\dfrac{\sg(\mathbf m)\sg(\mathbf k)}{\abs{\mathbf m}{}\abs{\mathbf k}{}} \bigr) e^{\pm i \varsigma \sg(\mathbf k) \abs{\mathbf k}{} \frac{t}{\varepsilon}},
\end{align*}
where $ \mathbf l = \mathbf m - \mathbf k $ and we have used \eqref{aw-per:fourier-u-01}.
Therefore
\begin{align*}
	& \mathcal Q_{\varepsilon,1}(\mathcal P_\sigma u, V_{\mathbf k}^\pm ) = \mathcal L\bigl(- \dfrac{t}{\varepsilon}\bigr) \sum_{ \mathbf  m}\bigl( \xi_{\mathbf m} V_{\mathbf m}^\pm +  \zeta_{\mathbf m}  V_{\mathbf m}^\mp \bigr) \\
	& = \dfrac{ i}{2} \sum_{\mathbf m - \mathbf l = \mathbf k} \biggl(
	\bigl( (\mathbf k \cdot \hat u_{\mathbf l}) ( 1 + \dfrac{\sg(\mathbf m) \sg (\mathbf k ) }{\abs{\mathbf m}{}\abs{\mathbf k}{}} \mathbf k \cdot  ( \mathbf m +\mathbf l) )  \bigr)   V_{\mathbf m}^\pm \\
	& ~~~~ ~~~~ \times e^{\pm i \varsigma (\sg (\mathbf k) \abs{\mathbf k}{} - \sg(\mathbf m) \abs{\mathbf m}{}) \frac{t}{\varepsilon} }\\
	& ~~~~ + \bigl( (\mathbf k \cdot \hat u_{\mathbf l}) ( 1 -\dfrac{\sg(\mathbf m) \sg (\mathbf k ) }{\abs{\mathbf m}{}\abs{\mathbf k}{}} \mathbf k \cdot  ( \mathbf m +\mathbf l) )  \bigr)    V_{\mathbf m}^\mp\\
	& ~~~~ ~~~~ \times e^{\pm i \varsigma (\sg (\mathbf k) \abs{\mathbf k}{} + \sg(\mathbf m) \abs{\mathbf m}{}) \frac{t}{\varepsilon}}\biggr).
\end{align*}

{\par \noindent \bf Calculation of $ \mathcal Q_{\varepsilon, 2} $. }To calculate $ \mathcal Q_{\varepsilon,2}(V_{\mathbf k}^{\pm}, V_{\mathbf l}^{\pm}) $ and $ \mathcal Q_{\varepsilon,2}(V_{\mathbf k}^{\pm}, V_{\mathbf l}^{\mp}) $, as before, notice that, for $ V_1, V_2 \in (\mathcal D'(\mathbb T^2))^3 $, the relation \eqref{aw-per:projection} implies the following identity:
\begin{align*}
	& \biggl( \begin{array}{c}
			\mathcal L_2 \bigl(\dfrac{t}{\varepsilon}\bigr) V_{1} \cdot \nablah \mathcal L_1 \bigl(\dfrac{t}{\varepsilon}\bigr) V_{2} - \int_{\mathbb T^2} \mathcal L_2 \bigl(\dfrac{t}{\varepsilon}\bigr) V_{1} \cdot \nablah \mathcal L_1 \bigl(\dfrac{t}{\varepsilon}\bigr) V_{2} \idxh \\
			\mathcal P_\tau (\mathcal L_2\bigl(\dfrac t \varepsilon\bigr) V_{1} \cdot \nablah \mathcal L_2\bigl(\dfrac t \varepsilon\bigr) V_{2} + c^2 \mathcal L_1\bigl(\dfrac{t}{\varepsilon}\bigr) V_{1} \nablah \mathcal L_1 \bigl(\dfrac{t}{\varepsilon}\bigr) V_{2} )
		\end{array} \biggr) \\
		& = \mathcal P_{(\ker L)^\perp}  \biggl( \begin{array}{c}
			\mathcal L_2 \bigl(\dfrac{t}{\varepsilon}\bigr) V_{1} \cdot \nablah \mathcal L_1 \bigl(\dfrac{t}{\varepsilon}\bigr) V_{2} \\
			\mathcal L_2\bigl(\dfrac t \varepsilon\bigr) V_{1} \cdot \nablah \mathcal L_2\bigl(\dfrac t \varepsilon\bigr) V_{2} + c^2 \mathcal L_1\bigl(\dfrac{t}{\varepsilon}\bigr) V_{1} \nablah \mathcal L_1 \bigl(\dfrac{t}{\varepsilon}\bigr) V_{2}
		\end{array} \biggr).
\end{align*}
Thus following the same arguments as in the calculation of $ \mathcal Q_{\varepsilon, 1} $ yields:
\begin{align*}
	& \mathcal Q_{\varepsilon,2}(V_{\mathbf k}^{\pm}, V_{\mathbf l}^{\pm}) = \mp  \dfrac{i c}{2 \sqrt{\gamma-1 + c^2}}\bigl(
	\dfrac{\sg(\mathbf k) \mathbf k \cdot \mathbf l}{\abs{\mathbf k}{}} + \dfrac{\sg(\mathbf m) \sg(\mathbf k) \sg(\mathbf l) (\mathbf l \cdot \mathbf k)(\mathbf l \cdot \mathbf m)}{\abs{\mathbf m}{}\abs{\mathbf k}{}\abs{\mathbf l}{}} \\
	& ~~~~ ~~~~ + \dfrac{(\gamma-1)\sg (\mathbf m) \mathbf l \cdot \mathbf m}{\abs{\mathbf m}{}}
	\bigr) V_{\mathbf m}^\pm e^{\pm i \varsigma ( \sg (\mathbf k) \abs{ \mathbf k}{} +  \sg (\mathbf l) \abs{ \mathbf l}{} -  \sg (\mathbf m) \abs{ \mathbf m}{} )\frac t \varepsilon} \\
	& ~~~~ ~~~~  \mp  \dfrac{i c}{2 \sqrt{\gamma-1 + c^2}}\bigl(
	\dfrac{\sg(\mathbf k) \mathbf k \cdot \mathbf l}{\abs{\mathbf k}{}} - \dfrac{\sg(\mathbf m) \sg(\mathbf k) \sg(\mathbf l) (\mathbf l \cdot \mathbf k)(\mathbf l \cdot \mathbf m)}{\abs{\mathbf m}{}\abs{\mathbf k}{}\abs{\mathbf l}{}} \\
	& ~~~~ ~~~~ - \dfrac{(\gamma-1)\sg (\mathbf m) \mathbf l \cdot \mathbf m}{\abs{\mathbf m}{}}
	\bigr)  V_{\mathbf m}^\mp e^{\pm i \varsigma ( \sg (\mathbf k) \abs{ \mathbf k}{} +  \sg (\mathbf l) \abs{ \mathbf l}{} +  \sg (\mathbf m) \abs{ \mathbf m}{} )\frac t \varepsilon}, \\
	& \mathcal Q_{\varepsilon,2}(V_{\mathbf k}^{\pm}, V_{\mathbf l}^{\mp}) = \mp  \dfrac{i c}{2 \sqrt{\gamma-1 + c^2}}\bigl(
	\dfrac{\sg(\mathbf k) \mathbf k \cdot \mathbf l}{\abs{\mathbf k}{}} - \dfrac{\sg(\mathbf m) \sg(\mathbf k) \sg(\mathbf l) (\mathbf l \cdot \mathbf k)(\mathbf l \cdot \mathbf m)}{\abs{\mathbf m}{}\abs{\mathbf k}{}\abs{\mathbf l}{}} \\
	& ~~~~ ~~~~ + \dfrac{(\gamma-1)\sg (\mathbf m) \mathbf l \cdot \mathbf m}{\abs{\mathbf m}{}}
	\bigr) V_{\mathbf m}^\pm  e^{\pm i \varsigma ( \sg (\mathbf k) \abs{ \mathbf k}{} -  \sg (\mathbf l) \abs{ \mathbf l}{} -  \sg (\mathbf m) \abs{ \mathbf m}{} )\frac t \varepsilon}\\
	& ~~~~ ~~~~  \mp  \dfrac{i c}{2 \sqrt{\gamma-1 + c^2}}\bigl(
	\dfrac{\sg(\mathbf k) \mathbf k \cdot \mathbf l}{\abs{\mathbf k}{}} + \dfrac{\sg(\mathbf m) \sg(\mathbf k) \sg(\mathbf l) (\mathbf l \cdot \mathbf k)(\mathbf l \cdot \mathbf m)}{\abs{\mathbf m}{}\abs{\mathbf k}{}\abs{\mathbf l}{}} \\
	& ~~~~ ~~~~ - \dfrac{(\gamma-1)\sg (\mathbf m) \mathbf l \cdot \mathbf m}{\abs{\mathbf m}{}}
	\bigr)  V_{\mathbf m}^\mp e^{\pm i \varsigma ( \sg (\mathbf k) \abs{ \mathbf k}{} -  \sg (\mathbf l) \abs{ \mathbf l}{} +  \sg (\mathbf m) \abs{ \mathbf m}{} )\frac t \varepsilon},
\end{align*}
where  $ \mathbf m = \mathbf k + \mathbf l $.

{\par \noindent \bf Calculation of $ \mathcal Q_{\varepsilon, 3} $.} This will be similar to the calculation of $ \mathcal A_\varepsilon (D) $. The result is
\begin{align*}
	& \mathcal Q_{\varepsilon,3}(g, V_{\mathbf k}^\pm) = \mp \dfrac{ic g \sqrt{\gamma-1}}{2} \sg (\mathbf k) \abs{\mathbf k}{} V_{\mathbf k}^\pm \\
	& ~~~~ \pm \dfrac{ic g \sqrt{\gamma-1}}{2} \sg (\mathbf k) \abs{\mathbf k}{} V_{\mathbf k}^\mp e^{\pm 2 i \varsigma \sg (\mathbf k) \abs{\mathbf k}{} \frac{t}{\varepsilon}}.
\end{align*}

{\par \noindent \bf  The limits of the operators $ \mathcal A_\varepsilon(D), \mathcal Q_{\varepsilon, 1}, \mathcal Q_{\varepsilon, 2}, \mathcal Q_{\varepsilon, 3} $ as $ \varepsilon \rightarrow 0^+ $. }
For fixed $ \mathbf k, \mathbf l $, in the sense of distribution, we have the following: as $ \varepsilon \rightarrow 0^+ $,
\begin{align}
	& \mathcal A_{\varepsilon}(D) V_{\mathbf k}^\pm \rightharpoonup - \dfrac{(\mu+\lambda)c_1}{2} \abs{\mathbf k}{2} V_{\mathbf k}^\pm = \dfrac{(\mu+\lambda)c_1}{2} \deltah V_{\mathbf k}^\pm =: \mathcal A(D) V_{\mathbf k}^\pm  , {\label{aw-lim:op01}}\\
	& \mathcal Q_{\varepsilon,1}(\mathcal P_\sigma u, V_{\mathbf k}^\pm ) \rightharpoonup i \sum_{ \tiny \begin{array}{c}\mathbf m - \mathbf l = \mathbf k \\ \sg(\mathbf k) \abs{\mathbf k}{} = \sg(\mathbf m) \abs{\mathbf m}{}  \end{array}}
	 \dfrac{ (\hat u_{\mathbf l} \cdot  \mathbf k) ( \mathbf m \cdot \mathbf k)}{\abs{\mathbf k}{2}}  V_{\mathbf m}^\pm {\nonumber} \\
	& ~~~~ ~~~~ + i \sum_{\tiny \begin{array}{c}\mathbf m - \mathbf l = \mathbf k \\ \sg(\mathbf k) \abs{\mathbf k}{} = - \sg(\mathbf m) \abs{\mathbf m}{}  \end{array}} \dfrac{ (\hat u_{\mathbf l} \cdot  \mathbf k) ( \mathbf m \cdot \mathbf k)}{\abs{\mathbf k}{2}}   V_{\mathbf m}^\mp = : \mathcal Q_1(\mathcal P_\sigma u, V_{\mathbf k}^\pm) , {\label{aw-lim:op02}} \\
	& \mathcal Q_{\varepsilon, 2}( V_{\mathbf k}^\pm,  V_{\mathbf l}^\pm)  \rightharpoonup  \mp  \dfrac{i c(\gamma+1)}{2 \sqrt{\gamma-1 + c^2}} \sg(\mathbf l) \abs{\mathbf l}{} V_{\mathbf m}^\pm =: \mathcal Q_2( V_{\mathbf k}^\pm,  V_{\mathbf l}^\pm) , {\label{aw-lim:op03}} \\
	& \text{where} ~ \mathbf m = \mathbf k + \mathbf l, ~ \sg(\mathbf k) \abs{\mathbf k}{} + \sg(\mathbf l) \abs{\mathbf l}{} = \sg(\mathbf m) \abs{\mathbf m}{}, ~  \text{$ \mathbf k $ and $ \mathbf l $ are co-linear},  {\nonumber} \\ % \sg(\mathbf k) \sg(\mathbf l) \abs{\mathbf k}{}\abs{\mathbf l}{} = \mathbf k \cdot \mathbf l  \\
	& \mathcal Q_{\varepsilon, 2}( V_{\mathbf k}^\pm,  V_{\mathbf l}^\mp) \rightharpoonup 0 =: \mathcal Q_2( V_{\mathbf k}^\pm,  V_{\mathbf l}^\mp) , {\label{aw-lim:op04}} \\
	& \mathcal Q_{\varepsilon,3}(g_\varepsilon , V_{\mathbf k}^\pm) \rightharpoonup  \mp \dfrac{ic g \sqrt{\gamma-1}}{2} \sg (\mathbf k) \abs{\mathbf k}{} V_{\mathbf k}^\pm =: \mathcal Q_3 (g, V_{\mathbf k}^\pm), {\label{aw-lim:op05}}
\end{align}
where we have used the facts that
\begin{gather*}
\bigl\lbrace \mathbf m = \mathbf k + \mathbf l, ~ \sg(\mathbf k) \abs{\mathbf k}{} + \sg(\mathbf l) \abs{\mathbf l}{} = \sg(\mathbf m) \abs{\mathbf m}{} \bigr\rbrace \\
= \bigl\lbrace  \mathbf k, \mathbf l,\mathbf m ~ \text{are co-linear, and} ~
	\mathbf l \cdot \mathbf k = \sg(\mathbf l) \sg(\mathbf k) \abs{\mathbf l }{}\abs{\mathbf k}{}, \\
	 \mathbf l \cdot \mathbf m = \sg(\mathbf l) \sg(\mathbf m) \abs{\mathbf l }{}\abs{\mathbf m}{} \bigr\rbrace , \\
	 \text{and} ~  \bigl\lbrace \mathbf m = \mathbf k + \mathbf l, ~ \sg(\mathbf k) \abs{\mathbf k}{} + \sg(\mathbf l) \abs{\mathbf l}{} = - \sg(\mathbf m) \abs{\mathbf m}{} \bigr\rbrace = \emptyset, \\
	 \bigl\lbrace \mathbf m = \mathbf k + \mathbf l, ~ \sg(\mathbf k) \abs{\mathbf k}{} - \sg(\mathbf l) \abs{\mathbf l}{} = \pm \sg(\mathbf m) \abs{\mathbf m}{} \bigr\rbrace = \emptyset.
\end{gather*}

{\par \noindent \bf  The limit equations. }
Now we have prepared enough to identify the limit equations of equations \eqref{aw-per:cm-oscillation} and \eqref{aw-per:average} as $ \varepsilon \rightarrow 0^+ $.

We rewrite equation \eqref{aw-per:cm-oscillation} in the following fashion:
\begin{equation}\label{aw-per:equation-new-fashion}
\begin{aligned}
	& \dt V_\varepsilon^o + \mathcal Q_{\varepsilon,1}( v_p, V^o) + \mathcal Q_{\varepsilon,2}( V^o, V^o) + \mathcal Q_{\varepsilon, 3}(g^o, V^o) - \mathcal A_\varepsilon(D) V^o\\
	& ~~ = - \mathcal Q_{\varepsilon,1}(\mathcal P_\sigma v_\varepsilon - v_p, V_\varepsilon^o) - \mathcal Q_{\varepsilon,1}( v_p, V_\varepsilon^o - V^o) \\
	& ~~~~ - \mathcal Q_{\varepsilon,2} ( V^o_\varepsilon - V^o, V_\varepsilon^o) - \mathcal Q_{\varepsilon,2} (V^o, V_\varepsilon^o - V^o) \\
	& ~~~~ - \mathcal Q_{\varepsilon, 3}(g_\varepsilon - g^o, V_\varepsilon^o)
	- \mathcal Q_{\varepsilon,3}(g^o, V_\varepsilon^o - V^o) \\
	& ~~~~ + \mathcal A_\varepsilon(D) (V_\varepsilon^o - V^o) + \mathcal L\bigl( - \dfrac t \varepsilon \bigr) \biggl( \begin{array}
		{c} 0 \\ K + L_1 + L_2
	\end{array}\biggr).
\end{aligned}
\end{equation}
Then, from the regularity in \eqref{limit:functions} and \eqref{aw-per:limit-of-osc-regulartity}, the weak convergence of $ \dt V^o_\varepsilon $ in \eqref{aw-per:limit-of-oscillation-variable}, and the convergence of operators in \eqref{aw-lim:op01}, \eqref{aw-lim:op02}, \eqref{aw-lim:op03}, \eqref{aw-lim:op04} and \eqref{aw-lim:op05}, as $ \varepsilon \rightarrow 0^+ $, the left-hand side of \eqref{aw-per:equation-new-fashion} converges in the sense of distribution to
\begin{equation*}
	\dt V^o + \mathcal Q_1(v_p, V^o) + \mathcal Q_2(V^o, V^o) + \mathcal Q_3(g^o, V^o) - \mathcal A(D) V^o.
\end{equation*}
Indeed, one can replace $ v_p, V^o $ with their finite dimensional Fourier truncations on the left-hand side of \eqref{aw-per:equation-new-fashion}, and similar estimates as in \eqref{aw-per:bounds-of-operators} of the operators for such truncations imply that the actions of the operators on the remainders are bounded by certain norms of the remainders uniformly in $ \varepsilon $. Then by letting $ \varepsilon $ go to zero, since the truncations approximate the identity operator, one can show the convergence of the actions of the operators.

We claim that the right-hand side of \eqref{aw-per:equation-new-fashion} converges in the sense of distribution to $ 0 $. Indeed, from the definition of the operators in \eqref{aw-per:operators} and the norm preserving property \eqref{aw-per:norm-preserve}, applying  the H\"older and Sobolev embedding inequalities implies that, as $ \varepsilon \rightarrow 0^+ $,
\begin{align*}
	& \normh{\mathcal Q_{\varepsilon,1}(\mathcal P_\sigma v_\varepsilon - v_p, V_\varepsilon^o)}{L^\infty(0,T; L^2(\mathbb T^2))} \lesssim \normh{ \mathcal L(\frac{t}{\varepsilon}) V_\varepsilon^o}{L^\infty(0,T;L^\infty(\mathbb T^2))}\\
	& ~~~~ \times \norm{\nablah (\mathcal P_\sigma v_\varepsilon - v_p )}{L^\infty(0,T;L^2(\mathbb T^2))} + \normh{\nablah \mathcal L(\frac{t}{\varepsilon}) V_\varepsilon^o}{L^\infty(0,T;L^4(\mathbb T^2))} \\
	& ~~~~ \times \norm{\mathcal P_\sigma v_\varepsilon - v_p}{L^\infty(0,T;L^4(\mathbb T^2))} \lesssim \normh{V_\varepsilon^o}{L^\infty(0,T;H^2(\mathbb T^2))}\\
	& ~~~~ \times \norm{\mathcal P_\sigma v_\varepsilon - v_p}{L^\infty(0,T;H^1(\mathbb T^2))} \rightarrow 0, \\
	& \normh{\mathcal Q_{\varepsilon,1}(v_p, V_\varepsilon^o - V^o)}{L^\infty(0,T;L^2(\mathbb T^2))} \lesssim \normh{V_\varepsilon^o - V^o}{L^\infty(0,T;H^1(\mathbb T^2))}\\
	& ~~~~ \times \norm{v_p}{L^\infty(0,T;H^2(\mathbb T^2))} \rightarrow 0, \\
	& \normh{\mathcal Q_{\varepsilon,2}(V_\varepsilon^o- V^o,V_\varepsilon^o )}{L^\infty(0,T;L^2(\mathbb T^2))} \lesssim \normh{\mathcal L(\frac{t}{\varepsilon})(V_\varepsilon^o - V^o)}{L^\infty(0,T;L^4(\mathbb T^2))} \\
	& ~~~~ \times \normh{\nablah \mathcal L(\frac{t}{\varepsilon})  V_\varepsilon^o}{L^\infty(0,T;L^4(\mathbb T^2))} \lesssim \normh{V_\varepsilon^o - V^o}{L^\infty(0,T;H^1(\mathbb T^2))} \\
	& ~~~~ \times \normh{V_\varepsilon^o}{L^\infty(0,T;H^2(\mathbb T^2))} \rightarrow 0, \\
	& \normh{\mathcal Q_{\varepsilon,2}(V^o, V_\varepsilon^o - V^o)}{L^\infty(0,T;L^2(\mathbb T^2))} \lesssim \normh{V^o}{L^\infty(0,T;H^2(\mathbb T^2))} \\
	& ~~~~ \times \normh{V_\varepsilon^o - V^o}{L^\infty(0,T;H^1(\mathbb T^2))} \rightarrow 0, \\
	& \normh{\mathcal Q_{\varepsilon,3}(g_\varepsilon- g^o, V_\varepsilon^o )}{L^\infty(0,T;L^2(\mathbb T^2))} \lesssim \normh{g_\varepsilon-g^o}{L^\infty(0,T)}\\
	& ~~~~ \times \normh{ V_\varepsilon^o}{L^\infty(0,T;H^1(\mathbb T^2))} \rightarrow 0,\\
	& \normh{\mathcal Q_{\varepsilon,3}(g^o, V_\varepsilon^o-V^o)}{L^\infty(0,T;L^2(\mathbb T^2))} \lesssim \normh{g^o}{L^\infty(0,T)} \\
	& ~~~~ \times \normh{V_\varepsilon^o-V^o}{L^\infty(0,T;H^1(\mathbb T^2))}  \rightarrow 0,
\end{align*}
where we have applied \eqref{limit:007}, \eqref{aw-per:lim-average}, \eqref{aw-per:norm-preserving}, \eqref{aw-per:limit-of-osc-regulartity}, \eqref{aw-per:limit-of-oscillation-variable}. On the other hand, the norm preserving property \eqref{aw-per:norm-preserve} implies that, for $ t \in (0,\infty) $ and any $ V_1, V_2 \in  (\mathcal D'(\mathbb T^2))^3 $,
\begin{gather*}
	0 = \normh{V_1 + \mathcal L(t) V_2}{L^2(\mathbb T^2)}^2 -  \normh{\mathcal L(- t)V_1 +  V_2}{L^2(\mathbb T^2)}^2\\
	 = 2 \int_{\mathbb T^2} V_1 \cdot \mathcal L(t) V_2 \idxh + 2 \int_{\mathbb T^2} \mathcal L(-t) V_1 \cdot V_2 \idxh.
\end{gather*}
Consider any $ \psi \in (\mathcal D'(\mathbb T^2))^3 $. By making use of the above relation, we have
\begin{align*}
	& \int_{\mathbb T^2} \mathcal A_\varepsilon(D)(V_\varepsilon^o - V^o) \cdot \psi \idxh \\
	& = - \int_{\mathbb T^2} \biggl( \begin{array}{c}
		0\\
		c_1(\mu \deltah \mathcal L_2(\dfrac{t}{\varepsilon})(V_\varepsilon^o - V^o) + \lambda \nablah \dvh \mathcal L_2(\dfrac{t}{\varepsilon})(V_\varepsilon^o - V^o))
	\end{array} \biggr) \cdot \mathcal L(\dfrac{t}{\varepsilon}) \psi \idxh \\
	& = c_1 \int_{\mathbb T^2} \biggl( \mu \nablah \mathcal L_2 (\dfrac t \varepsilon ) (V_\varepsilon^o - V^o) \cdot \nablah \mathcal L_2 (\dfrac t \varepsilon) \psi \\
	& ~~~~ + \lambda (\dvh \mathcal L_2 (\dfrac t \varepsilon ) (V_\varepsilon^o - V^o) ) \times ( \dvh \mathcal L_2 (\dfrac t \varepsilon) \psi ) \biggr) \idxh \lesssim \normh{\nablah \psi}{H^1(\mathbb T^2)} \\
	& ~~~~ \times \normh{V_\varepsilon^o - V^o}{H^1(\mathbb T^2)} \rightarrow 0 ~~~~ \text{in} ~ L^\infty(0,T), ~ \text{as} ~ \varepsilon \rightarrow 0^+,
\end{align*}
where we have applied the H\"older and Sobolev embedding inequalities, and \eqref{aw-per:limit-of-oscillation-variable}. Consequently, $ \mathcal A_\varepsilon(V_\varepsilon^o - V^o) \rightarrow 0 $ as $ \varepsilon \rightarrow 0^+ $ in the sense of distribution.

What is left is to show the convergence of the last term on the right-hand side of \eqref{aw-per:equation-new-fashion}. That is, we will show in the following, as $ \varepsilon \rightarrow 0^+ $, that
\begin{equation}\label{aw-lim:source}
	\mathcal F : = \mathcal L\bigl( - \dfrac t \varepsilon \bigr) \biggl( \begin{array}
		{c} 0 \\ K + L_1 + L_2 \end{array}\biggr) \rightarrow 0, ~~~~ \text{in the sense of distribution}.
\end{equation}
On the one hand, the estimates in \eqref{aw-per:bound-of-source} and the norm-preserving property \eqref{aw-per:norm-preserve} yield that,
\begin{equation*}
	 \mathcal L\bigl( - \dfrac t \varepsilon \bigr) \biggl( \begin{array}
		{c} 0 \\  L_1 + L_2 \end{array}\biggr) \rightarrow 0, ~~ \text{in}~ L^\infty(0,T; L^2(\mathbb T^2)) \cap L^2(0,T; H^1(\mathbb T^2)),
\end{equation*}
as $ \varepsilon \rightarrow 0^+ $.
On the other hand,
\begin{align*}
	& K = \mathcal P_\tau\bigl( - v_p \cdot \nablah  v_p - (\dvh \widetilde{v}_p) v_p\bigr) \\
	& ~~~~ + \mathcal P_\tau \bigl( -\mathcal P_\sigma v_\varepsilon \cdot \nablah (\mathcal P_\sigma v_\varepsilon - v_p ) - (\dvh \widetilde{\mathcal P_\sigma v_\varepsilon})( \mathcal P_\sigma v_\varepsilon - v_p ) \\
	& ~~~~ ~~~~  -  (\mathcal P_\sigma v_\varepsilon - v_p ) \cdot \nablah v_p - (\dvh (\widetilde{\mathcal P_\sigma v_\varepsilon} - \widetilde{v}_p )) v_p \bigr)  =: K' + K'',
\end{align*}
where, as the consequence of \eqref{limit:007} and the H\"older and Sobolev embedding inequalities, we have, as $ \varepsilon \rightarrow 0^+ $,
\begin{equation*}
	K'' \rightarrow 0 ~~ \text{in} ~ L^\infty (0,T;L^2(\Omega_h)) \cap L^2(0,T; H^1(\Omega_h)).
\end{equation*}
Therefore, one only has to investigate,
\begin{equation}\label{20July2018-01}
	\mathcal L \bigl( - \dfrac t \varepsilon \bigr) \biggl( \begin{array}
		{c} 0 \\ K'
	\end{array} \biggr) = \sum_{\mathbf k \in 2\pi \mathbb T^2\setminus\lbrace (0,0) \rbrace} \bigl( \hat{ K'}_{\mathbf k}^+ V_{\mathbf k}^+ e^{- i\varsigma \sg (\mathbf k) \abs{\mathbf k}{}\frac t \varepsilon} + \hat{K'}_{\mathbf k}^- V_{\mathbf k}^- e^{ i\varsigma \sg (\mathbf k) \abs{\mathbf k}{}\frac t \varepsilon}\bigr),
\end{equation}
where we have substituted the representation \eqref{20July2018}, and it is represented, using the relation \eqref{aw-per:projection} and \eqref{aw-per:projection-01},
\begin{equation*}
	\biggl( \begin{array}
		{c} 0 \\ K'
	\end{array} \biggr) = \sum_{\mathbf k \in 2\pi \mathbb T^2\setminus\lbrace (0,0) \rbrace} \bigl( \hat{K'}_{\mathbf k}^+ V_{\mathbf k}^+ + \hat{K'}_{\mathbf k}^- V_{\mathbf k}^- \bigr).
\end{equation*}
Together with the norm preserving property \eqref{aw-per:norm-preserve} and the fact that
\begin{equation*}
	\normh{K'}{L^\infty(0,T;H^1(\mathbb T^2))} \leq \norm{v_p}{L^\infty(0,T;H^2(\mathbb T^2 \times 2 \mathbb T))}^2 < \infty,
\end{equation*}
\eqref{20July2018-01} yields, as $ \varepsilon \rightarrow 0^+ $,
\begin{equation*}
	\mathcal L \bigl( - \dfrac t \varepsilon \bigr) \biggl( \begin{array}
		{c} 0 \\ K'
	\end{array} \biggr) \buildrel\ast\over\rightharpoonup 0 ~~ \text{weak-$\ast$ in} ~ L^\infty(0,T; H^1(\mathbb T^2)).
\end{equation*}
This finishes the proof of  \eqref{aw-lim:source}. Therefore, we have identified the limit equation of \eqref{aw-per:cm-oscillation}:
\begin{equation}\label{aw-per:limit-eq-osc}
	\dt V^o + \mathcal Q_1(v_p, V^o) + \mathcal Q_2(V^o, V^o) + \mathcal Q_3(g^o, V^o)  - \mathcal A(D) V^o = 0.
\end{equation}

To identity the limit equation of \eqref{aw-per:average},
% satisfied by $ g^o $,
notice that equation \eqref{aw-per:average} can be written as,
\begin{align*}
	& \dfrac{d}{dt} g_\varepsilon = - \int_{\mathbb T^2} \mathcal L_2 \bigl( \dfrac t \varepsilon \bigr) V^o \cdot \nablah \mathcal L_1\bigl( \dfrac t \varepsilon \bigr) V^o \idxh \\
	& ~~~~ ~~~~ + \int_{\mathbb T^2} \mathcal L_2 \bigl( \dfrac t \varepsilon \bigr) (V^o - V_\varepsilon^o ) \cdot \nablah \mathcal L_1 \bigl( \dfrac t \varepsilon \bigr) V^o \idxh \\
	& ~~~~ ~~~~ + \int_{\mathbb T^2} \mathcal L_2 \bigl( \dfrac t \varepsilon \bigr) V_\varepsilon^o \cdot \nablah \mathcal L_1 \bigl( \dfrac t \varepsilon \bigr) ( V^o - V_\varepsilon^o ) \idxh := M_1+ M_2 + M_3.
\end{align*}
Due to the norm preserving property of $ \mathcal L $, we have $ M_2 + M_3 \rightarrow 0 $ as $ \varepsilon \rightarrow 0^+ $ from \eqref{aw-per:norm-preserving}, \eqref{aw-per:limit-of-osc-regulartity} and \eqref{aw-per:limit-of-oscillation-variable}. To investigate $ M_1 $, since $ V^o $ is real-valued, we denote, with $ \hat{V}_{- \mathbf k}^\pm = \overline{\hat{V}_{\mathbf k}^\pm}^c $
\begin{equation}\label{aw-lim:expansion-of-oscillation}
	V^o = \sum_{\mathbf k\in 2\pi \mathbb T^2\setminus\lbrace (0,0)\rbrace} \bigl( \hat{V}_{\mathbf k}^+ V_{\mathbf k}^+ + \hat{V}_{\mathbf k}^- V_{\mathbf k}^-\bigr) .
\end{equation}
Then after applying \eqref{20July2018}, direct calculations show that
\begin{align*}
	& M_1 % = \sum_{\mathbf k, \mathbf l}  - \int_{\mathbb T^2} \hat{V}_{\mathbf k}^+ \hat{V}_{\mathbf l}^+ \mathcal L_2 \bigl( \dfrac{t}{\varepsilon}\bigr) V_{\mathbf k}^+ \cdot \nablah \mathcal L_1\bigl(\dfrac t \varepsilon\bigr) V_{\mathbf l}^+ \idxh  \\
%	& ~~~~ ~~~~ - \int_{\mathbb T^2} \hat{V}_{\mathbf k}^+ \hat{V}_{\mathbf l}^- \mathcal L_2 \bigl( \dfrac{t}{\varepsilon}\bigr) V_{\mathbf k}^+ \cdot \nablah \mathcal L_1\bigl(\dfrac t \varepsilon\bigr) V_{\mathbf l}^- \idxh  \\
%	& ~~~~ ~~~~ - \int_{\mathbb T^2} \hat{V}_{\mathbf k}^- \hat{V}_{\mathbf l}^+ \mathcal L_2 \bigl( \dfrac{t}{\varepsilon}\bigr) V_{\mathbf k}^- \cdot \nablah \mathcal L_1\bigl(\dfrac t \varepsilon\bigr) V_{\mathbf l}^+ \idxh  \\
%	& ~~~~ ~~~~ - \int_{\mathbb T^2} \hat{V}_{\mathbf k}^- \hat{V}_{\mathbf l}^- \mathcal L_2 \bigl( \dfrac{t}{\varepsilon}\bigr) V_{\mathbf k}^- \cdot \nablah \mathcal L_1\bigl(\dfrac t \varepsilon\bigr) V_{\mathbf l}^- \idxh \\
%	& ~~~~ = \sum_{\mathbf k} - i \hat{V}_{\mathbf k}^+ \hat{V}_{ - \mathbf k}^+ \dfrac{c\sqrt{\gamma-1} \sg (\mathbf k) \abs{\mathbf k}{}}{\gamma-1+c^2} + i \hat{V}_{\mathbf k}^- \hat{V}_{- \mathbf k}^- \dfrac{c\sqrt{\gamma-1} \sg (\mathbf k) \abs{\mathbf k}{}}{\gamma-1+c^2} \\
%	& ~~~~ - i \hat{V}_{\mathbf k}^+ \hat{V}_{- \mathbf k}^- \dfrac{c\sqrt{\gamma-1} \sg (\mathbf k) \abs{\mathbf k}{}}{\gamma-1+c^2}e^{2 i \varsigma \sg(\mathbf k) \abs{\mathbf k}{}\frac{t}{\varepsilon}} + i \hat{V}_{\mathbf k}^- \hat{V}_{-\mathbf k}^+ \dfrac{c\sqrt{\gamma-1} \sg (\mathbf k) \abs{\mathbf k}{}}{\gamma-1+c^2} e^{-2i \varsigma \sg(\mathbf k)\abs{\mathbf k}{}\frac{t}{\varepsilon}}\\
%	& ~~~~
	= \sum_{\mathbf k \in 2\pi \mathbb T^2 \setminus \lbrace (0,0) \rbrace} \biggl( - i \hat{V}_{\mathbf k}^+ \hat{V}_{ - \mathbf k}^+ \dfrac{c\sqrt{\gamma-1} \sg (\mathbf k) \abs{\mathbf k}{}}{\gamma-1+c^2} \\
	& ~~~~ ~~~~ + i \hat{V}_{\mathbf k}^- \hat{V}_{- \mathbf k}^- \dfrac{c\sqrt{\gamma-1} \sg (\mathbf k) \abs{\mathbf k}{}}{\gamma-1+c^2} \\
	& ~~~~ ~~~~ - 2 i \hat{V}_{\mathbf k}^+ \hat{V}_{- \mathbf k}^- \dfrac{c\sqrt{\gamma-1} \sg (\mathbf k) \abs{\mathbf k}{}}{\gamma-1+c^2}e^{2 i \varsigma \sg(\mathbf k) \abs{\mathbf k}{}\frac{t}{\varepsilon}} \biggr)\\
	& ~~~~ \buildrel\ast\over\rightharpoonup \dfrac{c\sqrt{\gamma-1}}{\gamma-1+c^2}  \sum_{\mathbf k\in 2\pi \mathbb T^2 \setminus\lbrace(0,0)\rbrace } \biggl(- i \hat{V}_{\mathbf k}^+ \hat{V}_{ - \mathbf k}^+  \sg (\mathbf k) \abs{\mathbf k}{} + i \hat{V}_{\mathbf k}^- \hat{V}_{- \mathbf k}^- \sg (\mathbf k) \abs{\mathbf k}{} \biggr)\\
	& =  i \dfrac{c\sqrt{\gamma-1}}{\gamma-1+c^2}  \sum_{\mathbf k\in 2\pi \mathbb T^2 \setminus\lbrace(0,0)\rbrace } \biggl(- \abs{\hat{V}_{\mathbf k}^+}{2}  +  \abs{\hat{V}_{\mathbf k}^-}{2} \biggr) \sg (\mathbf k) \abs{\mathbf k}{} \\
	&  \text{in} ~ L^\infty(0,T).
\end{align*}
Moreover, direct integrating \eqref{aw-per:average} in the temporal variable yields
\begin{equation*}
	g_\varepsilon(t) - \int_{\mathbb T^2} \xi_0 \idxh = \int_0^t M_1 \,dt +  \int_0^t ( M_2 + M_3 ) \,dt.
\end{equation*}
After taking $ \varepsilon \rightarrow 0 $, we have for any $ t \in [0,T] $,
\begin{equation*}
\begin{aligned}
	& g^o(t) - \int_{\mathbb T^2} \xi_0 \idxh \\
	& = i \dfrac{c\sqrt{\gamma-1}}{\gamma-1+c^2}  \sum_{\mathbf k\in 2\pi \mathbb T^2 \setminus\lbrace(0,0)\rbrace } \biggl(- \abs{\hat{V}_{\mathbf k}^+}{2}  +  \abs{\hat{V}_{\mathbf k}^-}{2} \biggr) \sg (\mathbf k) \abs{\mathbf k}{},
\end{aligned}
\end{equation*}
and since $ g^o $ is real-valued, this implies
\begin{equation}\label{aw-lim:average}
	g^o \equiv \int_{\mathbb T^2} \xi_0 \idxh.
\end{equation}

 We summarize the discussion in this subsection in the following:
 \begin{prop}\label{prop:convergence-periodic} Under the same assumptions as in Proposition \ref{prop:convergence-eq-distribution}, in the case when $ \Omega_h = \mathbb T^2 $, there are a function $ V^o $ satisfying the regularity in \eqref{aw-per:limit-of-osc-regulartity} and a constant $ g^o $ given by \eqref{aw-lim:average}, such that $ (V^o, g^o) $ satisfies equation \eqref{aw-per:limit-eq-osc}. The convergence in \eqref{aw-per:strong-cnvgn-1} holds as $ \varepsilon \rightarrow 0^+ $.
 \end{prop}

\paragraph{Acknowledgements}
This work was supported in part by the Einstein Stiftung/Foundation - Berlin, through the Einstein Visiting Fellow Program, and by the John Simon Guggenheim Memorial Foundation. X.L. and E.S.T. would like to thank the Isaac Newton Institute for Mathematical Sciences, Cambridge, for support and hospitality during the programme {\it``Mathematical aspects of turbulence: where do we stand?"} where part of the work on this paper was undertaken.
This work was supported in part by EPSRC grant no EP/R014604/1.
X.L.'s work was partially supported by a grant from the Simons Foundation, during his visit to the Isaac Newton Institute for Mathematical Sciences.

\bibliographystyle{plain}
%\bibliography{Compressible_Primitive_Equations,Low_Mach_Limit,My_Works,Thesis-Books}

\end{document}